\documentclass[12pt,preprint]{imsart}
\RequirePackage[OT1]{fontenc}
\RequirePackage{amsthm,amsmath}

\usepackage{graphicx,amsmath,amsthm,amsfonts,amssymb}
\usepackage{booktabs, rotating, float}
\usepackage[hmargin=3cm,vmargin=3cm]{geometry}
\usepackage{mathtools}
\usepackage{upgreek}
\usepackage{yhmath}

\usepackage{mathrsfs}
\usepackage[numbers]{natbib}
\usepackage{enumerate}
\usepackage[usenames,dvipsnames,svgnames,table]{xcolor}
\usepackage[pagebackref,letterpaper=true,colorlinks=true,pdfpagemode=none,urlcolor=blue,linkcolor=blue,citecolor=BrickRed,pdfstartview=FitH]{hyperref}

\newtheoremstyle{thmstyle}
  {3pt}{3pt}
  {\itshape}     
  {}
  {\bfseries}    
  {.}
  { }
  {\thmname{#1}\thmnumber{ #2}\thmnote{ (#3)}}

\newtheoremstyle{defstyle}
  {3pt}{3pt}
  {\normalfont}  
  {}
  {\bfseries}
  {.}
  { }
  {\thmname{#1}\thmnumber{ #2}\thmnote{ \textbf{(#3)}}}

\theoremstyle{thmstyle}

\newtheorem{thm}{Theorem}[section]

\newtheorem{prop}[thm]{Proposition}
\newtheorem{lemma}[thm]{Lemma}
\newtheorem{assumption}[thm]{Assumption}
\theoremstyle{defstyle}
\newtheorem{defn}[thm]{Definition}
\newtheorem{rem}[thm]{Remark}

\newcommand{\CB}{\mathcal{B}}
\newcommand{\CC}{\mathcal{C}}
\newcommand{\CA}{\mathcal{A}}
\newcommand{\CD}{\mathcal{D}}

\newcommand{\CF}{\mathcal{F}}
\newcommand{\CG}{\mathcal{G}}
\newcommand{\CZ}{\mathcal{Z}}
\newcommand{\CM}{\mathcal{M}}
\newcommand{\CV}{\mathcal{V}}
\newcommand{\CU}{\mathcal{U}}
\newcommand{\CW}{\mathcal{W}}
\newcommand{\R}{\mathbb{R}}
\newcommand{\Z}{\mathbb{Z}}

\newcommand{\N}{\mathbb{N}}
\newcommand{\T}{\mathbb{T}}
\renewcommand{\P}{\mathbb{P}}
\newcommand{\E}{\mathbb{E}}

\newcommand{\bY}{\mathbf{Y}}
\newcommand{\bX}{\mathbf{X}}
\newcommand{\X}{\mathbb{X}}
\newcommand{\supp}{\operatorname{supp}}
\renewcommand{\emptyset}{\varnothing}
\newcommand{\Caputo}{{}^{\scriptscriptstyle \mathsf{C}} D}
\newcommand{\RL}{{}^{\scriptscriptstyle \mathsf{RL}} D}
\newcommand{\vvvert}{\vert\!\vert\!\vert}

\def\SF{\mathscr{F}}
\def\fC{\mathfrak{C}}

\def\d{\mathrm{d}}
\def\Id{\mathrm{Id}}
\def\Li{\mathrm{Li}}
\def\ini{\mathrm{ini}}

\newcommand{\1}[1]{\mathbf{1}_{#1}}

\numberwithin{equation}{section}
\numberwithin{table}{section}

\startlocaldefs
\endlocaldefs

\begin{document}

\begin{frontmatter}

\title{Space-time fractional stochastic Burgers-type equation}
\runtitle{Space-time fractional stochastic Burgers-type equation}
\runauthor{M.-W. Kuo, K. Matetski}
\author{
  \fnms{Ming-Wei} \snm{Kuo}\thanks{E-mail: kuomingw@msu.edu}
  \and
  \fnms{Konstantin} \snm{Matetski}\thanks{E-mail: matetski@msu.edu}
}
\address{Michigan State University, Department of Mathematics \\ East Lansing, MI 48824, USA}

\begin{abstract}
We establish the existence and uniqueness of solutions to a $(1+1)$-dimensional fractional stochastic Burgers-type equation featuring a Caputo fractional time derivative, a fractional Laplacian, and space-time white noise forcing modified by a Riemann--Liouville fractional integral. From a physical perspective, equations of this type describe non-Markovian dynamics with long-range temporal and spatial dependencies. The main difficulty is that the solution to the linearized equation has low spatial H\"{o}lder regularity, rendering the singular nonlinearity $G(u)\partial_x u$ classically ill-defined. The proof relies on two main ingredients: first, combining the Da Prato--Debussche decomposition with rough path theory to rigorously interpret the nonlinear product; and second, deriving Schauder-type estimates for fractional heat kernels given in terms of slowly decaying Mittag--Leffler functions. The nonlocality of the fractional time derivative requires a strong assumption on the initial condition.
\end{abstract}

\end{frontmatter}

\setcounter{tocdepth}{1}
\tableofcontents

\section{Introduction}

The aim of this article is to study the existence and uniqueness of a solution to the following fractional stochastic Burgers-type equation in $(1+1)$-dimensions
\begin{equation}\label{eq:FSBL}
    \partial_t^{\beta} u = - (-\Delta)^{\alpha/2} u + F(u) + G(u) \partial_x u - (-\Delta)^{\delta/2} I_t^{1-\beta} \xi
\end{equation}
for a function $u : [0, \infty) \times \T \to \R^n$ with $n \geq 1$ and an initial condition $u(0,x) = u_0(x)$, where $t \geq 0$ is the time variable and the spatial variable $x$ takes values on the circle $\T \coloneqq \R / 2\pi\Z$, which we identify with the interval $[-\pi, \pi]$. Here, $-(-\Delta)^{\alpha/2}$ and $-(-\Delta)^{\delta/2}$ are the fractional Laplace operators on the circle with $\alpha > 0$ and $\delta \geq 0$, the operator $\partial_t^\beta$ is a Caputo fractional time derivative with $\beta \in (0,1)$, and $I_t^{1-\beta}$ is a Riemann--Liouville integral with respect to the time variable. We provide the definitions of these operators in Section~\ref{sec: definition} and make assumptions on the parameters below. The functions $F : \R^n \to \R^n$ and $G : \R^n \to \R^{n \times n}$ are $\CC^1$ and $\CC^3$ respectively. The derivative $\partial_x u$ is defined in the distributional sense, and we write $G(u) \partial_x u$ for a matrix-vector multiplication.  The equation is driven by an $\R^n$-valued space-time white noise $\xi$, i.e., a time derivative of an $n$-dimensional $L^2$-cylindrical Wiener process on a filtered probability space $(\Omega, \SF, (\SF_t)_{t \geq 0}, \P)$. This implies that the solution of the equation is a random function $u : [0, \infty) \times \T \times \Omega \to \R^n$, but we prefer to omit the probability space $\Omega$ from the notation. The name ``Burgers-type" refers to the nonlinearity $G(u) \partial_x u$ in the equation. Namely, in the case $G(u)=u$, we recover the standard nonlinear term in the Burgers equation.

Our main motivation to study equation \eqref{eq:FSBL} is due to the recent progress in the study of singular stochastic PDEs and their emergence as scaling limits of interacting particle systems. For example, the dynamical $\Phi^4$ equation describes the scaling limit of the Ising--Kac model near criticality \cite{MR4869263}, and the KPZ equation appears as a scaling limit of exclusion processes with weak asymmetry \cite{MR1462228}. These singular equations cannot be solved using the classical theory à la Walsh because of classically undefined nonlinearities (this is what the name ``singular" refers to), and one needs to exploit one of the recently developed theories, such as rough paths \cite{HairerRoughSPDEs}, regularity structures \cite{RegularityStructures}, or paracontrolled distributions \cite{MR3406823}. In the case of the KPZ equation, it can be linearized by the Cole--Hopf transformation, which does not seem to work well when studying scaling limits of exclusion processes with non-nearest neighbor jumps. The main motivation for our work is to understand a time-fractional KPZ equation and its emergence as a scaling limit of non-Markovian exclusion processes \cite{MR2805266}. Solving a simpler equation \eqref{eq:FSBL} and understanding its approximations is the first step in this direction. The non-fractional Burgers-type equation \eqref{eq:FSBL} with $\beta = 1$, $\alpha = 2$, and $\delta = 0$ was the first and the simplest singular stochastic PDE solved using the theory of rough paths \cite{HairerRoughSPDEs}. This solution method was later generalized in \cite{HairerKPZ} to solve the KPZ equation and led to the development of the theory of regularity structures \cite{RegularityStructures}.

Fractional stochastic PDEs are natural models studied by physicists and economists because they describe, at least formally, physical processes with nonlocal interactions and long memory. More precisely, a fractional Laplacian arises from long-range interactions of particles. On the other hand, the fractional time derivative describes the evolution of non-Markovian processes with long memory. We refer to the monograph on the topic \cite{MR3971272}, where such approximations are explained. Such evolutions exhibit an unusual tendency to remain in a given state for extended periods of time, resulting in dynamics characterized by long waiting times and memory effects. Processes with such behavior are usually referred to as anomalous diffusion. Defining such non-Markovian processes and studying their scaling limits are very difficult problems, especially when the limiting stochastic PDEs are singular. Only a few such scaling limits appear to be known (see some relevant results in \cite{10.1214/25-AIHP1579, MR2519002, MR3971272}). In this article, we demonstrate how the theory of rough paths can be applied to solve the simplest singular fractional equation \eqref{eq:FSBL}, and we hope it will initiate research in this direction. 

We should note that the form of the driving noise $I_t^{1-\beta} \xi$ will play an important role when we apply the fractional Duhamel's principle to write the equation in a mild form in Section~\ref{sec: stationary sol}. More precisely, the inverse of the linear operator $\partial_t^{\beta} + (-\Delta)^{\alpha/2}$ will cancel the fractional integral $I_t^{1-\beta}$ and simplify our analysis of the equation. Having noise in this form is, however, not an artificial trick but a very natural assumption \cite{MR4957469}, and it corresponds to a long-term effect of an external force.

The main difficulties with solving equation \eqref{eq:FSBL} are twofold: the nonlinearity $G(u) \partial_x u$ is classically undefined, and the Green function for the linear operator does not have a fast decay. By analogy with \cite{HairerRoughSPDEs}, one expects a solution to \eqref{eq:FSBL}, provided it exists, to have the same spatial regularity as the stationary solution to the linearized equation 
\begin{equation}\label{eq:linear}
    \partial_t^{\beta} \psi = - (-\Delta)^{\alpha/2} \psi - (-\Delta)^{\delta/2} I_t^{1-\beta}\Pi_0 \xi,
\end{equation}
where $\Pi_0$ is the projection map that removes the constant Fourier mode. We need to use $\Pi_0$ to have a stationary solution. We prove in Section~\ref{sec: stationary sol} that for any $t > 0$, the process $\psi(t,\cdot)$ has almost surely spatial H{\"o}lder regularity 
\begin{equation}\label{eq:gamma-0}
    \gamma < \gamma_0 \coloneqq \frac{1}{2}\left(\frac{\alpha}{\beta} - 2 \delta - 1\right), 
\end{equation}
assuming the parameters satisfy our standing Assumption~\ref{assump:parameters}. Hence, to define the product $G(\psi) \partial_x \psi$, we need to multiply a $\CC^\gamma$ function with a $\CC^{\gamma-1}$ distribution. A classical result in harmonic analysis \cite{MR2768550, MR631751} says, however, that the product 
\begin{equation*}
    \CC^\alpha \times \CC^\beta \ni (u,v) \mapsto u v
\end{equation*}
is well-defined as a bilinear map when $\alpha + \beta > 0$. Hence, in the case $\gamma < \frac{1}{2}$, which we are interested in, this product is classically undefined, which prevents us from solving equation \eqref{eq:FSBL} classically. We should note that in a special case when $G = \nabla \CG$, the issue with defining the product does not arise, because we can use the chain rule and write it as 
\begin{equation*}
    G(u) \partial_x u = \partial_x \CG(u),
\end{equation*}
which is a well-defined function or distribution as long as $u$ is continuous.

Following \cite{HairerRoughSPDEs} (and its generalization \cite{MR3010394} for a multiplicative noise), we will use the theory of rough paths to define the product $G(u) \partial_x u$ in terms of a rough path integral. The main tool for this is the Da Prato--Debussche ansatz, which says that, provided a solution $u$ exists and $u(t, \cdot) \in \CC^\gamma$, the process $v(t,x) = u(t,x) - \psi(t,x)$ has a higher spatial regularity. Our choice of the parameters in the equation will guarantee that $v(t,\cdot) \in \CC^1$. Then the singular product can be defined as 
\[
\bigl(G(u) \partial_x u\bigr)(t,x) = \partial_x \left( \int_{-\pi}^x G(v + \psi)(t,y)\, \d_y (v + \psi)(t,y) \right),
\]
where the derivative $\partial_x$ is computed in the distributional sense, the integral with respect to $v$ is defined in the classical sense as 
\[
\int_{-\pi}^x G(v + \psi)(t,y) \partial_y v(t, y)\, \d y,
\]
and the integral with respect to $\psi$ is defined in the rough path sense, after lifting the process $\psi$ to a rough path. The definition of a rough path and a rough path integral can be found in Section~\ref{sec: rough path theory}. This definition allows us to solve the stochastic PDE \eqref{eq:FSBL} in the mild form using the standard Banach fixed-point argument on a suitable Banach space. 

The second difficulty of the equation is the fractional time derivative. In the classical setting, one can use the classical Schauder-type estimate to analyze the regularity of the solution \cite{HairerRoughSPDEs}. In the fractional setting, however, the situation is considerably more delicate, because the Schauder-type regularity estimates for such operators are less standard, as the respective Green functions do not have fast decays, as opposed to the heat kernel. Several related results are available in the literature: \cite{MR1675170} considers equations in which the sum of the orders of the temporal and spatial derivatives is less than $2$, \cite{MR3038123} establishes a De Giorgi--Nash-type regularity result for time-fractional diffusion equations in divergence form. These results provide important precedents for the analysis of fractional evolution equations, but they do not directly address the setting considered in this article. We should also note that the literature on fractional stochastic PDEs is considerably more limited than in the deterministic setting. 
In \cite{MR3980916}, an $L_p$-theory is developed for equations involving a time-fractional derivative acting both on the solution and on the stochastic forcing. In \cite{MR3357610}, Kolmogorov's continuity criterion is then used to extend classical regularity results for the stochastic heat equation. These works demonstrate that fractionalization of either the differential operator or the noise can substantially alter the regularity properties of the resulting stochastic PDE.

Another interesting phenomenon observed in a fractional equation is that the regularity of the solution will strongly depend on the regularity of the initial data, which is not true for the classical heat equation, which exhibits an instant smoothing property. In fact, this phenomenon is expected, as time-fractional diffusion equations are used to model dynamic processes in materials with memory, in which the current state depends on the past in a strong way. 

In this article, we define a solution of the fractional stochastic Burgers-type equation \eqref{eq:FSBL} and prove that it admits a mild solution with H{\"o}lder regularity of any order $\gamma < \gamma_0$ provided the Assumptions~\ref{assump:parameters} and~\ref{assump:initial} are satisfied. The main result is stated in Theorem~\ref{thm:main}. The remainder of this article is organized as follows. We give a brief review of rough path theory and the theory of controlled rough paths in Section~\ref{sec: rough path theory}, with emphasis on the results that will be used throughout this work. 
Section~\ref{sec: fractional heat eq} is devoted to the fractional heat kernel and associated Mittag--Leffler functions. We recall the relevant properties of these objects and develop the estimates needed in the subsequent analysis. In Section~\ref{sec: stationary sol}, we apply these results to study the properties of the stationary solution of the corresponding linearized equation. One of the main technical parts of the article is contained in Section~\ref{sec:schauder}, where we establish Schauder-type estimates for the general kernels that have only polynomial decay of Fourier coefficients. 
Finally, Section~\ref{sec: definition of sol and main proof} gives the notion of a mild solution to equation \eqref{eq:FSBL} and the proof of the main Theorem~\ref{thm:main}. 
Appendix~\ref{sec:besov space} provides several results about Besov spaces used throughout the article, and Appendix~\ref{sec:Duhamel} explains the fractional Duhamel's principle.

\subsection{Main result}

We make the following assumptions on the parameters $\alpha, \beta, \delta$ in equation \eqref{eq:FSBL} and $\gamma_0$ defined in \eqref{eq:gamma-0}.

\begin{assumption}\label{assump:parameters} We assume that $\alpha > \frac{5}{3}$, $\beta \in (\frac{1}{2}, 1)$, $\gamma_0 \in (\frac{1}{3}, \frac{1}{2}]$, and $\delta \geq 0$ where the parameters are moreover related by the identity \eqref{eq:gamma-0}.
\end{assumption}

As follows from our analysis of equation \eqref{contraction_M}, this assumption is necessary to guarantee that the solution has the spatial $\CC^1$ regularity. For example, we can take $\delta = 0$ and any $(\alpha, \beta)$ sufficiently close to $(2,1)$ such that $\alpha < 2 \beta$. We could have included the value $\beta = 1$ in the assumption, which would correspond to the classical time derivative in \eqref{eq:FSBL}, but this would require modifying our analysis for this particular case, which is not the main goal of this article. 

Furthermore, we make the following assumption on the initial condition. 

\begin{assumption}\label{assump:initial}
We assume that the initial condition $u_0$ of equation \eqref{eq:FSBL} is of one of the following types:
\begin{enumerate}[(i)]
    \item $u_0 - \psi_0 \in \CC^1(\T)$, where $\psi$ is the stationary solution of the linear equation \eqref{eq:linear}, which is defined in \eqref{eq:stationary_sol}, \label{assump: initial 1st kind}
    \item $u_0 \in \CC^{\zeta_0} (\T)$ for some $\zeta_0 > \frac{1}{2}$. \label{assump: initial 2nd kind}
\end{enumerate}
\end{assumption}

As one can see from the analysis in Section~\ref{sec:Z-bounds}, we need the integral $\int_0^\cdot u_0(x) \d u_0(x)$ to be defined as a function in a suitable H\"{o}lder space. Assumption \eqref{assump: initial 2nd kind} allows to define it as a Young integral \cite{MR1555421}. On the contrary, assumption \eqref{assump: initial 1st kind} allows to define the integral by defining $\int_0^\cdot \psi_0(x) \d \psi_0(x)$ in the rough paths sense. The first assumption on the initial condition may look artificial, but, as we prove, this is exactly the form of the solution of \eqref{eq:FSBL} at any time $t > 0$. Namely, if we get a local solution of the equation on an interval $[0, T]$, it can be written as $u(T, x) = \psi(T,x) + v(T, x)$ where $v(T, \cdot) \in \CC^1(\T)$. Then, considering the equation on the interval $[T, \infty)$ with $u(T, \cdot)$ being a new initial condition, we use assumption \eqref{assump: initial 1st kind} to get a local solution on a larger time interval. The details can be found in the last part of Section~\ref{sec:main-proof}. 
 
The following is the main result of this article. 

\begin{thm}\label{thm:main}
Let Assumptions~\ref{assump:parameters} and \ref{assump:initial} be satisfied. Then for almost every realization of the noise $\xi$ there exists $\bar T \in (0, \infty]$ such that the stochastic fractional Burgers-type equation \eqref{eq:FSBL} has a unique solution $u(t, \cdot) \in \CC^\gamma(\T)$ on the time interval $t \in [0, \bar T)$, for any $\gamma < \gamma_0$. This solution is defined in the sense of Definition~\ref{def: sol of FSBL}. 

In the case $\bar T < \infty$ this solution satisfies  
\begin{equation}\label{eq:convergence-lifetime}
\lim_{t \nearrow \bar T} \| u_t - \psi_t \|_{\CC^{\widetilde \gamma}} = \infty
\end{equation}
almost surely for any $\widetilde \gamma > 1$, where $\psi$ is the stationary solution of equation \eqref{eq:linear}. 

If $\|F\|_{\CC^1}$ and $\|G\|_{\CC^3}$ are bounded, then $\bar T = \infty$ almost surely.
\end{thm}

This theorem follows from a more precise result proved in Theorem~\ref{thm:main}. 

\subsection{Notation and definitions} \label{sec: definition}

We use the standard notation $\N$ for the set of natural numbers $\{1, 2, \ldots \}$ and denote $\N_0 = \N \cup \{0\}$ and $\R_+ = [0, \infty)$.

For $\gamma \in (0,1)$ we write $\CC^{\gamma}(\T)$ for the standard space of $\gamma$-H\"{o}lder continuous functions. For $\gamma \in \N_0$, $\CC^\gamma(\T)$ denotes the space of functions whose derivatives up to order $\gamma$ are continuous. More generally, for $\gamma > 0$, $\CC^\gamma(\T)$ consists of functions in $\CC^{ \lfloor \gamma \rfloor }(\T)$ whose derivative of order $\lfloor \gamma \rfloor$ belong to $\CC^{ \gamma - \lfloor \gamma \rfloor }(\T)$. 
When $\gamma=0$, we write $\CC(\T)$, and we often omit the domain $\T$ of the spaces to keep notation lighter. To work with space-time functions $f : [0, T] \times \T \to \R$ we define the norm
\begin{equation*}
    \| f \|_{\CC^{\gamma_1} ([0, T]; \CC^{\gamma_2})} \coloneqq \sup_{t \in [0, T]} \| f(t, \cdot) \|_{\CC^{\gamma_2}} + \sup_{ 0 \leq s < t \leq T } \frac{ \| f(t, \cdot) - f(s, \cdot) \|_{\CC^{\gamma_2}} }{ |t-s|^{\gamma_1}}
\end{equation*}
with $\gamma_1 \in (0,1)$ and $\gamma_2 > 0$, and define $\CC^{\gamma_1} ([0, T]; \CC^{\gamma_2}(\T))$ to be the space of such functions with a finite norm. We define 
\begin{equation*}
    \|f\|_{\CC^\gamma_T} \coloneqq \sup_{t \in [0,T]} \|f(t,\cdot)\|_{\CC^\gamma}, \qquad \|f\|_{L^\infty_T} \coloneqq \sup_{t \in [0,T]} \|f(t,\cdot)\|_{L^\infty}. 
\end{equation*}
Sometimes, we will work with functions $f$ which may have a blow-up at $t=0$. In this case, we define 
\begin{equation}
\|f\|_{\CC^{\gamma, \lambda}_T} \coloneqq \sup_{t \in (0,T]} t^\lambda \|f(t,\cdot)\|_{\CC^\gamma}, \label{def: C gamma t-lambda}
\end{equation}
for a parameter $\lambda \geq 0$ measuring the speed of the blow-up. For a function $f$ on a space-time domain, we use the notation $f(t,x)$ and $f_t(x)$ interchangeably, depending on the context. 

We use the notation $a \lesssim b$ throughout the article, which means that there is a constant $C > 0$, independent of the relevant quantities, such that $a \leq Cb$. The ``relevant quantities'' will always be clear from the context and refer to the space-time variables $x$, $t$, etc. 

We are going to define fractional generalizations of integrals and derivatives, which we use throughout the article. The provided material can be found in \cite{ABDELJAWAD2019122494}. 

\begin{defn}[Riemann--Liouville integral] 
    For $\beta \in (0,1)$ and a fixed point $a \in \R$, \emph{the left-side Riemann--Liouville integral} of order $\beta$ starting at $a$ is defined as 
    \begin{equation}\label{eq:RL-integral}
    \bigl(I^{\beta}_{a+} f\bigr)(t) \coloneqq \frac{1}{\Gamma(\beta)} \int_{a}^t \frac{f(r)}{(t-r)^{1-\beta}} \, \d r,
    \end{equation}
    for $t > a$ and a continuous function $f : [a, t] \to \R$. Here, $\Gamma(\beta)$ is the Gamma function. In the case $a = 0$ we simply write $I^{\beta} = I^{\beta}_{0+}$.
    
    For an end point $b \in \R$, \emph{the right-side Riemann--Liouville integral} of order $\beta$ ending at $b$ is defined as 
    \begin{equation}\label{eq:RL-integral-minus}
    \bigl(I^{\beta}_{b-} f\bigr)(t) \coloneqq \frac{1}{\Gamma(\beta)} \int_t^b \frac{f(r)}{(r-t)^{1-\beta}} \, \d r,
\end{equation}
    for $t < b$ and a continuous function $f : [t, b] \to \R$.
\end{defn}

We are going to use the following two standard definitions of fractional derivatives. 

\begin{defn}[Caputo fractional derivative] \label{def: Caputo derivative}
 For $\beta \in (0,1)$ and a starting point $a \in \R$, \emph{the left-side Caputo fractional derivative} of order $\beta$ starting at $a$ is defined as 
    \begin{equation}\label{eq:frac-deriv}
    \Caputo^{\beta}_{a+} f (t) \coloneqq \bigl( I^{1-\beta}_{a+} \partial \bigr) f(t) = \frac{1}{\Gamma(1-\beta)} \int_a^t \frac{f'(r)}{(t-r)^\beta} \, \d r,
    \end{equation}
    for $t > a$ and a $\CC^1$ function $f : [a, t] \to \R$. We denote $\partial^{\beta} f(t) = \Caputo^{\beta}_{0+} f(t)$.
    
    \emph{The left-side Riemann--Liouville fractional derivative} of order $\beta$ starting at $a \in \R$ is defined as 
    $$
    \RL^{\beta}_{a+} f(t) \coloneqq \big( \partial I^{1-\beta}_{a+} \big) f(t) = \frac{1}{\Gamma(1-\beta)} \frac{\d}{\d t} \int_a^t \frac{f(r)}{(t-r)^\beta} \, \d r,
    $$
    for $a < t$ and a function $f : [a, t] \to \R$ for which the expression is defined. 
    
    \emph{The right-side Riemann--Liouville fractional derivative} of order $\beta$ ending at $b \in \R$ is defined as 
    \begin{equation}\label{eq:RL-derivative}
    \RL^{\beta}_{b-} f(t) \coloneqq - \big( \partial I^{1-\beta}_{b-} \big) f(t) = - \frac{1}{\Gamma(1-\beta)} \frac{\d}{\d t} \int_t^b \frac{f(r)}{(r-t)^\beta} \, \d r,
    \end{equation}
    for $t < b$ and a function $f : [t,b] \to \R$ for which the expression is defined.
\end{defn}

The fractional operators have several properties which are similar to their non-fractional counterparts. We are going to list only a few of them, and more information can be found in \cite{ABDELJAWAD2019122494}. In particular, the Caputo derivative of order $1$ coincides with the classical derivative $f'$. Moreover, the fractional operators have the semigroup properties: 
\begin{equation*}
\partial^{\beta_1} \partial^{\beta_2} = \partial^{\beta_1 + \beta_2}, \qquad I^{\beta_1}_{0+} I^{\beta_2}_{0+} = I^{\beta_1 + \beta_2}_{0+},
\end{equation*}
for any $\beta_1, \beta_2 \in (0,1)$ such that $\beta_1 + \beta_2 < 1$. Combining \eqref{eq:frac-deriv} with the preceding identities, we get
\begin{equation}\label{eq:derivative-of-integral}
\partial^{\beta} I^{\beta}_{0+} = \Id,
\end{equation}
which we use throughout the paper. 

One important result, which makes the fractional operators hard to work with, is a more complicated version of the integration by parts formula:
\begin{equation}\label{eq:fractional-IBP}
    \int_a^b f(t) (\Caputo_{a+}^\beta g )(t) \, \d t = \int_a^b ( \RL_{b-}^\beta f)(t) g(t) \, \d t + \bigl[(I_{b-}^{1-\beta} f)(t) g(t) \bigr]_{t=a}^b,
\end{equation}
for any $a < b$ and any suitable continuous function $f$ for which $\RL_{b-}^\beta f$ is defined and a $\CC^1$ function $g$ on $[a,b]$.

When working with a function of several variables, it will be necessary to specify to which variable these operators are applied. In this case, we write $I^{\beta}_{a+, t}$,  $I^{\beta}_{b-, t}$, etc., which means that we apply the operators to the variable $t$.

To work with operators in the spatial domain, we will use the Fourier transform 
\[
\widehat{f}(k) := \frac{1}{ \sqrt{2\pi} } \int_{-\pi}^{\pi} f(x) e^{-ikx} \,\d x, \qquad k \in \Z,
\]
where $f \in L^1(\T)$. Whenever it is convenient, we use the notations $\widehat{f}$ and $\CF(f)$ interchangeably. The inverse Fourier transform is given by the series 
\[
\CF^{-1} \bigl[ \widehat{f}\, \bigr] (x) = \frac{1}{ \sqrt{2 \pi} } \sum_{k \in \Z} \widehat{f}(k) e^{ikx}, \qquad x \in \T.
\]
Then the fractional Laplacian $(-\Delta)^{\alpha/2}$ is defined in the standard way via the Fourier multiplier
\begin{equation*}
    (-\Delta)^{\alpha/2} f(x) \coloneqq \CF^{-1} \left[ |\cdot|^{\alpha} \widehat{f}(\cdot) \right] (x) = \frac{1}{ \sqrt{2 \pi} } \sum_{k \in \Z} |k|^{\alpha} \widehat{f}(k) e^{ikx}.
\end{equation*}

\section{Elements of rough path theory} \label{sec: rough path theory}

In this section, we review some elements of rough path theory, which was originally developed by T. Lyons in \cite{MR1654527} (a broader exposition can be found in \cite{MR4174393, MR2604669, MR2036784, MR2314753}).

We denote by $\CC_2( [a,b]^2, \R^n)$ the space of continuous functions from $[a,b]^2$ to $\R^n$ which vanish on the diagonal and define 
\begin{equation*}
    \CC^\gamma_2( [a,b]^2, \R^n) \coloneqq \left\{ R \in \CC_2( [a,b]^2, \R^n) : \|R\|_{\CC_2^\gamma} < \infty \right\}
\end{equation*}
with the norm 
\[
\|R\|_{\CC_2^\gamma} \coloneqq \displaystyle\sup_{x < x' \in [a,b] } \frac{ |R_{x,x'}| }{|x'-x|^\gamma}. 
\]
Note that the space $\CC^\gamma_2( [a,b]^2, \R^n)$ is a Banach space under this norm. For a function $f : [a,b] \to \R$ and $a \leq x < x' \leq b$, we denote its increment by
\[
\delta f_{x,x'} \coloneqq f(x') - f(x).
\]

A rough path $\bX$ on the interval $[a,b]$ has two components $\bX = (X, \X)$, where $X \in \CC([a,b], \R^n)$, $\X \in \CC_2( [a,b]^2, \R^{n \times n})$, and the Chen's relation holds 
\begin{equation*}
    \X^{ij}_{x,x'} - \X^{ij}_{x,y} - \X^{ij}_{y,x'} = \delta X^i_{x,y} \delta X^j_{y,x'},
\end{equation*}
for all $1 \leq i,j \leq n$ and $x \leq y \leq x' \in [a,b]$. One should think of $\X$ as being the integral
\begin{equation*}
    \X^{ij}_{x,x'} \,\, ``=" \, \int_x^{x'} \delta X^i_{x,y}\, \d X^j_y, 
\end{equation*}
although the expression on the right-hand side is not always defined. So, one can think that the integral on the right-hand side is defined by the expression on the left-hand side. 

For $\gamma \in (0, \frac{1}{2}]$, we denote by $\CD^{\gamma} ( [a,b], \R^n)$ the space of those rough paths $\bX = (X, \X)$ such that $X \in \CC^{\gamma} ( [a,b], \R^n)$ and $\X \in \CC^{2 \gamma}_2 ( [a,b]^2, \R^{n \times n} )$, and equip it with the natural semi-norm $\vvvert \bX \vvvert_{\gamma} \coloneqq \|X\|_{\gamma} + \sqrt{ \|\X\|_{2 \gamma} }$, where 
\[
\|X\|_{\gamma} \coloneqq \sup_{x \neq x' \in [a, b]} \frac{|X(x) - X(x')|}{|x-x'|^\gamma}.
\]
The space $\CD^\gamma ( [a,b], \R^n )$ of $\gamma$-rough paths is not even a vector space because Chen's relation is nonlinear. However, it is still a complete metric space with the metric $d(\bX, \bX') = \vvvert \bX - \bX' \vvvert_{\gamma}$, with the difference $\bX - \bX'$ defined component-wise, which is a closed subset of $\CC^\gamma \times \CC_2^{2 \gamma}$.

For a $\gamma$-rough path valued function $\bX_t = (X_t, \X_t)$ on $t \in [0,T]$, we define respectively the semi-norms
\begin{equation*}
    \| \X \|_{2 \gamma; T} \coloneqq \sup_{s \in [0,T]} \| \X_s \|_{2 \gamma}, \qquad\qquad \vvvert \bX \vvvert_{\gamma; T} \coloneqq \sup_{s \in [0,T]} \vvvert \bX_s \vvvert_{\gamma},
\end{equation*}
and define the space of such functions by $\CD^{\gamma}_T ( [a,b]^2, \R^n)$.

\subsection{Controlled rough paths}

The controlled rough path is an important notion introduced in \cite{MR2091358}, which allows us to extend Young's integration theory to a larger class of functions. Given a rough path $\bX \in \CD^{\gamma}([a,b], \R^n)$, we say that a pair of functions $\bY = (Y, Y')$ is controlled by $\bX$ if $Y \in \CC^\gamma ([a,b], \R^m)$, $Y' \in \CC^{\gamma} ([a,b], \R^{m \times n} )$, and the remainder term $R$, defined as 
\begin{equation}\label{eq:remainder-rough-path}
    R_{x,x'} \coloneqq \delta Y_{x,x'} - Y'_x \delta X_{x,x'},
\end{equation}
satisfies $\|R\|_{2 \gamma} < \infty$. Here, the product in $Y'_x \delta X_{x,x'}$ is the matrix multiplication. The process $Y'$ is called \emph{the Gubinelli derivative} of $Y$ with respect to $\bX$. We define by $\CC^\gamma_{\bX}$ the space of such functions $\bY = (Y, Y')$, controlled by $\bX$, and equip it with the semi-norm
\begin{equation*}
    \|\bY\|_{\bX, \gamma} \coloneqq \|Y\|_{\gamma} + \|Y'\|_{\gamma} + \|R\|_{2 \gamma}. 
\end{equation*}
Then, for a fixed rough path $\bX$, the space $\CC^\gamma_{\bX}$ is a Banach space with the norm induced by this semi-norm. 

When we consider functions $\bY_t = (Y_t, Y'_t)$, parametrized by $t \in [0, T]$ and with values in $\CC^\gamma_{\bX}$, we denote, as above, the space by $\CC^\gamma_{\bX; T}$ and the respective semi-norm by $\|\bY\|_{\bX, \gamma; T}$.

We will use the following result later, which shows that a twice-differentiable function of a rough path yields a controlled rough path. 

\begin{lemma}\label{Gubinelli deriv of G(v+x)}
    Let $\bX = (X, \X) \in \CD^\gamma( [a,b], \R^n)$ be a $\gamma$-rough path, $v \in \CC^{2 \gamma}( [a,b], \R^n)$ and $g \in \CC^2(\R^n, \R^n)$. Then $(g(X+v), \nabla g(X+v) ) \in \CC^\gamma_{\bX}$. Moreover, we have the estimates 
        \begin{align}
          \| R^g \|_{2 \gamma} 
          &\leq \frac{1}{2} \|g\|_{\CC^2} \left( \|X\|_{\gamma}^2 + 2 \|v\|_{ {2 \gamma} } \right), \label{bdd: remainder R of G} \\
        \| ( g(X+v), \nabla g(X+v) ) \|_{\bX, \gamma} 
        &\leq \frac{1}{2}\|g\|_{\CC^2} \left(4 \|X\|_{\gamma} + \|X\|_{\gamma}^2 + 4 ( \|v\|_{\gamma} + \|v\|_{ {2 \gamma} } )\right), \notag 
    \end{align}
    where $R^g$ is the remainder \eqref{eq:remainder-rough-path} for this controlled rough path.
\end{lemma}

\begin{proof}
    Since $g \in \CC^{2}$ and $v \in \CC^{2 \gamma}$, the function $g(v+X)$ is in $\CC^\gamma$ on $[a,b]$. Furthermore, for any $a \leq x < x' \leq b$ we have
    \begin{align*}
        &\big| g(X_{x'} + v_{x'}) - g(X_{x} + v_{x}) - \nabla g(X_{x} + v_{x}) ( X_{x'} - X_{x} ) \big| \\
          &\hspace{3cm} \leq \big| g(X_{x'} + v_{x}) - g(X_{x} + v_{x}) - \nabla g(X_{x} + v_{x}) ( X_{x'} - X_{x} ) \big| \\
        &\hspace{7cm} + \big| g(X_{x'} + v_{x'}) - g(X_{x'} + v_{x}) \big| \\
          &\hspace{3cm} \leq \frac{1}{2} \|g\|_{\CC^2} \|X\|_{\gamma}^2 |x' - x|^{2 \gamma} + \|g\|_{\CC^1} \|v\|_{ {2 \gamma} } |x' - x|^{2 \gamma}
    \end{align*}    
    where we used the standard bound for a Taylor approximation. Thus, the remainder term $R^g$ can be bounded as 
    \begin{align*}
        \|R^g\|_{2 \gamma} 
        &\leq \frac{1}{2} \|g\|_{\CC^2} \|X\|_{\gamma}^2 + \|g\|_{\CC^1} \|v\|_{ {2 \gamma} },
    \end{align*}     
    which implies that $\nabla g(X+v)$ is a Gubinelli derivative of $g(X+v)$ with respect to $\bX$. Moreover, we can obtain a bound on $\| ( g(X+v), \nabla g(X+v) ) \|_{\bX, \gamma}$ as
    \begin{align*}
        &\| ( g(X+v), \nabla g(X+v) ) \|_{\bX, \gamma} = \|g(X+v)\|_\gamma + \|\nabla g(X+v)\|_\gamma + \|R^g\|_{2 \gamma} \\
          &\hspace{1cm} \leq \|g\|_{\CC^1} ( \|v\|_\gamma + \|X\|_\gamma ) + \|g\|_{\CC^2} ( \|v\|_\gamma + \|X\|_\gamma ) + \frac{1}{2} \|g\|_{\CC^2} \|X\|_{\gamma}^2 + 2 \|g\|_{\CC^1} \|v\|_{ {2 \gamma} },
    \end{align*}    
    which yields the desired estimate.
\end{proof}

\subsection{Integration of controlled rough paths}

The theory of rough paths allows to define the integral $\int_a^b \bY_x \otimes \d \bX_x$ for $\bX \in \CD^{\gamma}$ and $\bY \in \CC^\gamma_{\bX}$ with $\gamma \in (\frac{1}{3}, \frac{1}{2}]$. In the classical setting, we can define an integral $\int_a^b Y_x \otimes \, \d X_x$ as a limit of Riemann sums for $Y \in \CC^\alpha$ and $X \in \CC^\beta$ whose regularities satisfy $\alpha + \beta > 1$. This is a generalization of the Riemann integral, called Young's integral \cite{MR1555421}. \emph{The rough path integral} is defined as 
\begin{equation}
    \int_a^b \bY_x \otimes \d \bX_x \coloneqq \lim_{ |\mathcal{P}| \to 0 } \sum_{[x,x'] \in \mathcal{P} } \Big( Y_x \otimes \delta X_{x,x'} + Y'_x \X_{x,x'} \Big),
    \label{def: rough path integral}
\end{equation}
where $\mathcal{P}$ is a partition of $[a,b]$ with mesh size $|\mathcal{P}|$. See \cite[Theorem~4.10]{MR4174393} for the proof that the limit exists and is independent of the partition. 

At the end, we state a result about the regularity of the rough path integral, which can be found in \cite[Theorem~4.4]{MR4174393}. In the statement of the lemma, we use the matrix-vector multiplication in the integral $\int_x^{x'} G(X_y)\, \d \bX_y$, where $G: \R^n \to \R^{n \times n}$. It can be written in terms of the tensor products \eqref{def: rough path integral}, where the controlled processes $\bY$ are the columns of the matrix $G(X)$, as follows from Lemma~\ref{Gubinelli deriv of G(v+x)}. In this case, we say that $(G(X), D G(X))$ is controlled by $\bX$, where $D G$ is the Jacobian of $G$.

\begin{prop}\label{prop:controlled-function}
    Let $\bX = (X, \X) \in \CD^\gamma( [a,b], \R^n )$ for some $T > 0$ and $\gamma \in (\frac{1}{3}, \frac{1}{2}]$, and let $G: \R^n \to \R^{n \times n}$ be a $\CC^2$ function with bounded derivatives. Then the rough path integral defined in \eqref{def: rough path integral} exists and has the following bound
    \begin{align*}
        &\left| \int_x^{x'} G(X_y)\, \d \bX_y - G(X_x) \delta X_{x,x'} - DG(X_x) \X_{x,x'} \right| \\
          &\qquad\qquad\qquad\qquad\qquad \leq C \|G\|_{\CC^2} \left( \|X\|^3_\gamma + \|X\|_{\gamma} \|\X\|_{2 \gamma} \right) |x'-x|^{3 \gamma},
    \end{align*}
    where the constant $C$ depends on $\gamma$. Moreover, the function $x' \mapsto \int_a^{x'} G(X_y)\, \d \bX_y$ is $\gamma$-H\"{o}lder continuous on $[a, b]$ with the bound
    \begin{equation*}
        \left\| \int_a^{\cdot} G(X_y)\, \d \bX_y \right\|_{\gamma} \leq \tilde C \|G\|_{\CC^2} \left( \vvvert \bX \vvvert_{\gamma} \vee \vvvert \bX \vvvert^{1/\gamma}_{\gamma} \right),
    \end{equation*}
    where the constant $\tilde C$ depends on $b-a$ and $\gamma$.
\end{prop}

\section{Fractional heat equation} \label{sec: fractional heat eq}

Let us first analyze a linearized version of equation \eqref{eq:FSBL}, i.e., the fractional heat equation 
\begin{equation}
    \partial_t^{\beta} \psi(t, x) = - (-\Delta)^{\alpha/2} \psi(t, x) - (-\Delta)^{\delta/2} I_t^{1-\beta} \Pi_0 \xi(t, x). \label{eq:FSH}
\end{equation} 
for $(t,x) \in \R_+ \times \T$ with a suitable initial condition $\psi(0, \cdot) = \psi_0 (\cdot)$. Here, $I_t^{1-\beta}$ is the fractional integral \eqref{eq:RL-integral} applied to the time variable $t$, and $\Pi_0$ is the orthogonal projection onto the non-zero frequency Fourier space, i.e.,
\[
\Pi_0 f(x) := \displaystyle\sum_{k \in \Z \setminus \{0\}} \widehat{f}(k) e^{ikx}.
\] 
From Appendix~\ref{sec:Duhamel}, the solution can be written using a fractional Duhamel's principle as 
\begin{equation*}
    \psi(t, x) = \int_\T P_{\beta, \alpha}(t, x-y) \psi_0(y) \, \d y - \int_0^t \int_{\T} P_{\beta, \alpha}(t-s, x-y) (-\Delta)^{\delta/2} \Pi_0 \xi(s, y) \, \d y\, \d s,
\end{equation*}
where $P_{\beta, \alpha}(t, x)$ is the fractional heat kernel, which is defined via its Fourier transform as 
\begin{equation}\label{eq:heat-kernel}
\widehat{P}_{\beta, \alpha}(t, k) = \frac{1}{\sqrt{2\pi}} E_{\beta}(- t^\beta |k|^\alpha),  
\end{equation}
where $E_{\beta}$ is the Mittag--Leffler function given by its Taylor series \cite{gorenflo2020mittag}
\begin{equation}\label{eq:ML}
E_{\beta}(x) = \sum_{j=0}^\infty \frac{x^{j}}{\Gamma(1+ \beta j)}.
\end{equation}
We note that for $\beta=1$ and $\alpha=2$ we have $E_{1}(x) = e^x$ and $\widehat{P}_{1, 2}(t, k) = \frac{1}{\sqrt{2\pi}} e^{-t k^2}$, so that $P_{1, 2}$ is the standard heat kernel. As follows from Lemma~\ref{L^1-bound of P} with $\ell_\star= 1$, if $\alpha < 2$ then $P_{\beta, \alpha}(t,x)$ is not differentiable at $x=0$ even when $t>0$. 

In contrast to the standard heat kernel, the convolution with $P_{\beta, \alpha}$ does not give a strong smoothing effect. This is due to a very slow decay of the Mittag--Leffler function $E_{\beta}$ at negative infinity. We demonstrate this in Figure~\ref{fig:ML function}, where we plot the function for several values of the parameter $\beta$. As you can see, there is a significant difference in the behavior of the function in the cases $\beta=1$ and $\beta < 1$. When $\beta=1$, the function $E_{1}(-x)$ decays exponentially as $x \to \infty$, whereas for $\beta < 1$ it decays only polynomially. More precisely, we have the following bound, which follows from the integral representation of the function \cite{gorenflo2020mittag}:

\begin{lemma}\label{lem:ML-bound}
For any $\beta \in (0,1)$ and integer $k \geq 0$ there is a constant $C > 0$ such that 
\begin{equation}\label{eq:ML-bound}
\bigl| E^{(k)}_{\beta} (-x)  \bigr| \leq \frac{C}{1+ x^{k+1}}
\end{equation}
uniformly over $x \geq 0$.
\end{lemma}

In particular, when $k=0$, the Mittag--Leffler function has the bounds \cite[Theorem~4]{MR3164769} 
\begin{equation*}
    \frac{1}{ 1+ \Gamma(1-\beta) x } \leq E_\beta(-x) \leq \frac{1}{ 1+ \Gamma(1+\beta) x }
\end{equation*}
for $x \geq 0$.
\begin{figure}[h]
    \centering
    \includegraphics[width=0.7\linewidth]{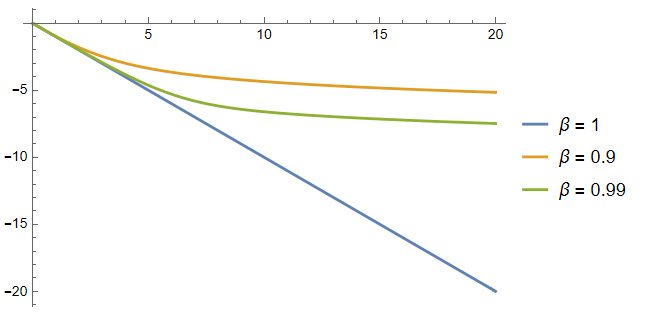}

    \caption{A plot of the logarithm of the Mittag--Leffler function $E_{\beta}(-x)$ for different values of $\beta$.} 
    \label{fig:ML function}
\end{figure}

We aim to compute the regularity of the solution to the fractional heat equation \eqref{eq:FSH}. For this, we need to study properties of the fractional heat kernel $P_{\beta, \alpha}$, which boils down to studying properties of the Mittag--Leffler functions. 

\subsection{Mittag--Leffler functions and fractional calculus} 

We start by proving a relation between Mittag--Leffler functions and fractional integrals.

\begin{lemma}\label{lem:LM-integral}
    Let $\beta \in (0,1)$. 
    For any $s< t$ and $\lambda > 0$ one has 
    \begin{equation*}
        I^{\beta}_{t-, s} E_{\beta}(- \lambda (t-s)^\beta) = - \frac{1}{\lambda} \left( E_{\beta}(- \lambda (t-s)^\beta) -1 \right).
    \end{equation*}
    This identity also holds for $\lambda = 0$, by computing the limit
    \begin{equation}\label{eq:LM-integral-2}
        \lim_{\lambda \searrow 0} I^{\beta}_{t-, s} E_{\beta}(- \lambda (t-s)^\beta) = \frac{(t-s)^\beta}{\Gamma ( 1+ \beta ) }.
    \end{equation}
\end{lemma}

\begin{proof}
    Using the definition of Riemann--Liouville integral \eqref{eq:RL-integral-minus}, we get for any $s < t$ and integer $j \geq 0$
    \begin{align*}
        I^{\beta}_{t-, s} ( (t-s)^{\beta j} ) 
        &= \frac{1}{\Gamma(\beta)} \int_s^t \frac{(t-r)^{\beta j}}{(r-s)^{1-\beta}} \, \d r = \frac{1}{\Gamma(\beta)} \int_0^{t-s} \frac{(t-s- r)^{\beta j}}{r^{1-\beta}} \, \d r \\
        &= \frac{(t-s)^{\beta (j+1)}}{\Gamma(\beta)} B(\beta, \beta j + 1) = \frac{\Gamma(1+\beta j)}{\Gamma(1+ (j+1)\beta)} (t-s)^{\beta (j+1)},
    \end{align*}
    where $B$ is the beta function. Recalling the definition \eqref{eq:ML}, we use the just-proved identity to get 
    \begin{align*}
        I^{\beta}_{t-, s} E_{\beta}(- \lambda (t-s)^\beta) 
        &= \sum_{j=0}^\infty I^{\beta}_{t-, s} \left( \frac{( -\lambda (t-s)^\beta)^{j}}{\Gamma(1+ \beta j)} \right) = \sum_{j=0}^\infty \frac{(-\lambda)^j}{\Gamma(1+ \beta j)} I^{\beta}_{t-, s} ( (t-s)^{\beta j}) \\
        &= \sum_{j=0}^\infty \frac{(-\lambda)^j}{\Gamma(1+ (j+1)\beta)} (t-s)^{\beta (j+1)} = - \frac{1}{\lambda} \left( E_{\beta}(- \lambda (t-s)^\beta ) -1 \right).
    \end{align*}
    Using $E_{\beta}(0) = 1$, the limit \eqref{eq:LM-integral-2} equals $(t-s)^\beta E_{\beta}'(0) = \frac{(t-s)^\beta}{\Gamma ( 1+ \beta ) }$.
\end{proof}

As a corollary of the preceding lemma, we can compute a Riemann--Liouville derivative of the fractional heat kernel \eqref{eq:heat-kernel}.

\begin{lemma}\label{lem:frac_D_of_ML}
Let $\beta \in (0,1)$ and $\alpha > 0$. For any $s< t$ and $k \in \Z$ one has
    \begin{equation*}
        \RL^{1-\beta}_{t-, s} E_{\beta}( -(t-s)^\beta |k|^\alpha ) = \frac{\beta}{(t-s)^{1-\beta}} E'_{\beta}(-(t-s)^\beta |k|^\alpha).
    \end{equation*}
\end{lemma}

\begin{proof}
    By the definitions \eqref{eq:RL-derivative} we have 
    \begin{align*}
        \RL^{1-\beta}_{t-, s} E_{\beta}( -(t-s)^\beta |k|^\alpha ) = - (\partial_s I^\beta_{t-, s}) E_{\beta}(-(t-s)^\beta |k|^\alpha).
    \end{align*}
    We use furthermore Lemma~\ref{lem:LM-integral} to compute for $k \neq 0$
    \begin{align*}
        (\partial_s I^\beta_{t-, s}) E_{\beta}(-(t-s)^\beta |k|^\alpha) &= - \frac{1}{|k|^\alpha} \partial_s E_{\beta}(- (t-s)^\beta |k|^\alpha) \\
        &= -\beta (t-s)^{\beta-1} E'_{\beta}(- (t-s)^\beta |k|^\alpha).
    \end{align*}
    These formulas are defined for all non-zero $k \in \R$. Taking the limit $k \to 0$ and using \eqref{eq:LM-integral-2}, we get the desired formula.
\end{proof}

\subsection{Stationary solution of the linear equation} 
\label{sec: stationary sol}

The fractional Duhamel's principle, described in Appendix~\ref{sec:Duhamel}, yields that the function
\begin{equation*}
    \psi(t, x) \coloneqq \int_{-\infty}^t \int_{-\pi}^{\pi} P_{\beta, \alpha}(t-s, x-y) \, \partial_s^{1-\beta} \left( - (-\Delta)^{\delta/2} I_s^{1-\beta} \Pi_0 \xi(s, y) \right)  \d y\,\d s
\end{equation*}
is a stationary solution of \eqref{eq:FSH}. As pointed out in Appendix~\ref{sec:Duhamel}, we use the identity \eqref{eq:derivative-of-integral} and interpret the right-hand side as 
    \begin{equation}\label{eq:stationary_sol}
        \psi(t,x) 
        = - \int_{-\infty}^t \int_{-\pi}^{\pi} P_{\beta, \alpha}(t-s, x-z) \, (-\Delta)^{\delta/2} \Pi_0 \xi(s, z)\,\d z\, \d s.
    \end{equation}
The well-posedness of this expression can be verified by computing its covariance. 

\begin{lemma}
    The function \eqref{eq:stationary_sol} is a stationary solution of the equation \eqref{eq:FSH} and its covariance function is 
    \begin{equation}\label{eq:psi-covariance}
        \E \big[ \psi(t,x) \otimes \psi(t',x') \big] = \left( \frac{1}{\pi} \sum_{k=1}^\infty k^{2 \delta - \frac{\alpha}{\beta}} \fC_\beta(k^{\alpha / \beta}(t' - t)) \cos(k(x-x')) \right) \Id,
    \end{equation}
    for all $t \leq t'$ and $x, y \in \T$, where 
    \begin{equation}\label{eq:psi-covariance-function-constant}
    \fC_\beta(z) \coloneqq \int^{\infty}_{0} E_{\beta}(- s^\beta) E_{\beta}(- (z+s)^\beta) \,\d s.
    \end{equation}
\end{lemma}

We note that the bound \eqref{eq:ML-bound} on the Mittag--Leffler function and our assumption $\beta \in (\frac{1}{2}, 1)$ guarantee that the integral in \eqref{eq:psi-covariance-function-constant} converges. The sum in \eqref{eq:psi-covariance} is convergent due to $\frac{\alpha}{\beta} - 2 \delta > 1$, which follows from our assumptions on the parameters. 

\begin{proof}
    Because the noise is independent across components, it suffices to compute the diagonal components. Parseval's identity and the definitions of the involved operators allow us to write it as 
    \begin{equation*}
        - \int_{-\infty}^t \int_{-\pi}^{\pi} (-\Delta)^{\delta/2} \Pi_0 P_{\beta, \alpha}(t-s, x-z) \, \xi(s, z) \,\d z\, \d s.
    \end{equation*}
    Therefore, for $\psi = ( \psi^j )_{j = 1,\, \dots ,\, n}$, It\^{o}'s isometry yields 
\begin{align*}
        &\E \big[ \psi^j (t,x) \psi^j (t',x') \big] \\
          &\qquad = \int_{-\infty}^{t} \int_{-\pi}^{\pi} (-\Delta)^{\delta/2} \Pi_0 P_{\beta, \alpha}(t-s, x-z) (-\Delta)^{\delta/2} \Pi_0 P_{\beta, \alpha}(t'-s, x'-z) \, \d z\, \d s
\end{align*}
for $t \leq t'$. Parseval's identity allows us to write it as 
\begin{align*}
    &\int_{-\infty}^{t} \sum_{k \in \Z \setminus \{0\}} \left( - |k|^{\delta} \right) e^{-ikx} \widehat{P}_{\beta, \alpha} (t-s, -k) \overline{\left( - |k|^{\delta} \right) e^{-ik x'} \widehat{P}_{\beta, \alpha} (t'-s, -k)} \,\d s \\
      &\qquad = \frac{1}{2\pi} \int_{-\infty}^{t} \sum_{k \in \Z \setminus \{0\}} e^{-ik(x-x')} |k|^{2 \delta} E_{\beta}(- (t-s)^\beta |k|^\alpha) E_{\beta}(- (t'-s)^\beta |k|^\alpha) \,\d s,
\end{align*}
where in the last identity we used formula \eqref{eq:heat-kernel}. The imaginary part of this sum vanishes by antisymmetry. Changing the integration variable, we get \eqref{eq:psi-covariance}.

    Because the expectation of the process vanishes and the covariance function \eqref{eq:psi-covariance} depends on the difference of the time points, the process $\psi$ is stationary in time.
\end{proof}

Now, we are going to estimate the regularity of the process $\psi$.

\begin{lemma}\label{lem:reg-of-stationary-sol}
    For any $\gamma_1, \gamma_2 \geq 0$ such that $\frac{\alpha}{\beta} \gamma_1 + \gamma_2 < \gamma_0$ (where the constant $\gamma_0$ is defined in \eqref{eq:gamma-0}) and for any $T > 0$ and $p \geq 1$ there is a constant $C = C(\gamma_1, \gamma_2, T, p) > 0$ such that
    \begin{equation}\label{eq:psi-regularity}
        \E \| \psi\|^p_{\CC^{\gamma_1} ([0, T]; \CC^{\gamma_2})} \leq C.
    \end{equation}
\end{lemma}

\begin{proof}
We are going to use Kolmogorov's continuity theorem \cite{MR4226142}, which requires bounding moments of the process. We note that $\psi$ is Gaussian, and it is sufficient to bound only the second moments. More precisely, we will prove that there exists a constant $L = L(\alpha, \beta, \delta)$ such that
\begin{equation}\label{eq:psi-bound-1}
\E \big[ | \psi(t, x) - \psi(t, x') |^2 \big] \leq L |x-x'|^{2\gamma'}
\end{equation}
and 
\begin{equation}\label{eq:psi-bound-2}
\E \big[ | \psi(t, x) - \psi(t', x) |^2 \big] \leq L |t-t'|^{2 \beta'},
\end{equation}
for any $0 \leq \gamma' < \frac{1}{2}(\frac{\alpha}{\beta} - 2 \delta - 1)$ and $0 \leq \beta' < \frac{\beta}{2 \alpha}(\frac{\alpha}{\beta} - 2 \delta - 1)$ such that $\beta' \leq \frac{\beta}{2}$, and for all spatial and temporal points $x, x' \in \T$ and $t, t' \in [0, T]$. The proportionality constants are independent of the spatial and temporal points. Then Kolmogorov's continuity theorem and interpolation yield \eqref{eq:psi-regularity}. We note that the condition $\beta' \leq \frac{\beta}{2}$ is redundant due to the assumptions on the parameters.

To prove \eqref{eq:psi-bound-1} we use \eqref{eq:psi-covariance} and write  
    \begin{align*}
        \E \big[ | \psi(t, x) - \psi(t, x') |^2 \big] = \fC_\beta(0) \frac{2 n}{\pi} \sum_{k=1}^\infty k^{2 \delta - \frac{\alpha}{\beta}} \bigl(1- \cos(k(x-x'))\bigr),
    \end{align*} 
    where we recall that $\psi$ is $\R^n$-valued. 
     We use the standard bound $1 - \cos(z) \lesssim |z|^{2 \gamma}$ for any $\gamma \in [0, 1]$ to estimate the preceding sum by a constant times 
    \[
    |x-x'|^{2 \gamma} \sum_{k=1}^\infty k^{2 \delta - \frac{\alpha}{\beta} + 2 \gamma},
    \]
    which is convergent if $2 \gamma < \frac{\alpha}{\beta} - 2 \delta - 1$. By our assumptions, we have $\frac{\alpha}{\beta} - 2 \delta > 1$ which yields the desired bound \eqref{eq:psi-bound-1}.

    Similarly, we can prove \eqref{eq:psi-bound-2}. We take $t < t'$ and use \eqref{eq:psi-covariance} to write 
    \begin{equation}\label{eq:psi-moment-in-time}
        \E \big[ | \psi(t, x) - \psi(t', x) |^2 \big] = \frac{2 n}{\pi} \sum_{k=1}^\infty k^{2 \delta - \frac{\alpha}{\beta}} \bigl(\fC_\beta(0) - \fC_\beta(|k|^{\alpha / \beta}(t' - t))\bigr).
    \end{equation}
    Recalling the definition \eqref{eq:psi-covariance-function-constant}, we write 
    \begin{align*}
    \fC_\beta(0) - \fC_\beta(z) &= \int^{\infty}_{0} E_{\beta}(- s^\beta) \bigl(E_{\beta}(- s^\beta) - E_{\beta}(- (z+s)^\beta)\bigr) \,\d s.
    \end{align*}
    We use the mean value theorem, the bound \eqref{eq:ML-bound}, the complete monotonicity of the Mittag--Leffler function \cite{gorenflo2020mittag} along with the elementary inequality $(z+s)^\beta - s^\beta \leq z^\beta$ provided $0 \leq \beta \leq 1$ to obtain
    \begin{equation*}
        E_{\beta}(- s^\beta) - E_{\beta}(- (z+s)^\beta) \lesssim \frac{z^\beta}{1 + s^{2\beta}}.
    \end{equation*}
    On the other hand, monotonicity and the bound \eqref{eq:ML-bound} yield 
    \begin{equation*}
        E_{\beta}(- s^\beta) - E_{\beta}(- (z+s)^\beta) \leq E_{\beta}(- s^\beta) \lesssim \frac{1}{1 + s^{\beta}}.
    \end{equation*}
    Interpolating between these two estimates, we get 
    \begin{equation*}
        \fC_\beta(0) - \fC_\beta(z) \lesssim \int_0^\infty \left( \frac{1}{1+s^\beta} \right)^{ 2 - \frac{2\beta'}{\beta}} \left( \frac{z^\beta}{1+s^{2\beta}} \right)^{ \frac{2\beta'}{\beta} } ds \lesssim z^{2 \beta'},
    \end{equation*}
    for any $0 \leq \beta' \leq \frac{\beta}{2}$ and $\beta > \frac{1}{2}$. Then the sum in \eqref{eq:psi-moment-in-time} is estimated by a constant times
    \begin{align*}
        (t' - t)^{2 \beta'} \sum_{k=1}^\infty k^{2 \delta - \frac{\alpha}{\beta} + \frac{2 \alpha \beta'}{\beta}}.
    \end{align*}
    The latter converges if $\beta' < \frac{\beta}{2 \alpha} (\frac{\alpha}{\beta} - 2 \delta - 1)$, which yields the desired bound \eqref{eq:psi-bound-2}.
\end{proof}

\begin{rem}
    The result of Theorem \ref{lem:reg-of-stationary-sol} demonstrates an interesting phenomenon. In the equation \eqref{eq:FSH}, we take $\delta$ times derivative in space on the forcing term (the space-time white noise), so it is no surprise that the regularity of $\psi$ will decrease by $\delta$ in space. However, the interesting part lies in the regularity in time. In the equation \eqref{eq:FSH}, we take $\beta$ times derivative in time and $\alpha$ times derivative in space. Therefore, heuristically, taking $\beta$ times derivative in time ``equals" taking $\alpha$ times derivative in space. And we take $\delta$ times derivative in space on the forcing term, which ``equals to" taking $\frac{\beta}{\alpha} \delta$ times derivative in time. And this is exactly the amount of regularity $\psi$ gets decreased in time.
\end{rem}

\subsection{Lift to a rough path}

Lemma~\ref{lem:reg-of-stationary-sol} yields the spatial regularity $\CC^\gamma(\T)$ of the stationary process $\psi$ a.s., where $\gamma$ is as in \eqref{eq:gamma-0}. In this paper, we consider only the interval $\frac{1}{3} < \gamma < \gamma_0 \leq \frac{1}{2}$, i.e., the region where $\psi$ does not have enough regularity to define the nonlinear part in \eqref{eq:FSBL} in the classical sense, and one can define it using the theory of rough paths. 

We start by showing that $\psi(t, \cdot)$ can be lifted to a rough path. We note that we consider the rough paths in the spatial variable $x$, while the time variable $t$ is fixed. This dependence on time will be indicated by the subscript, e.g., $\psi_t$.

\begin{lemma}\label{lem:lift}
    For every $\gamma \in (\frac{1}{3}, \gamma_0)$ and every $t \geq0$, the process $\psi_t$ can be almost surely canonically lifted to a $\gamma$-rough path $\Psi_t = (\psi_t, \uppsi_t)$, so that for any $T > 0$ and any $p \geq 1$ there is a constant $C = C(T, p) > 0$ for which one has 
    \begin{equation}\label{eq:lift-bound}
        \E \big[ \vvvert \Psi \vvvert_{\gamma, T} \big] \leq C.
    \end{equation}
\end{lemma}

\begin{proof}
    Because $\psi$ is a Gaussian process with stationary increments in space, the result will follow from \cite[Theorem~10.9 and Corollary~10.10]{MR4174393} if we show that the expectation of the increment
    \begin{equation*}
        \sigma^2(u) \coloneqq \E \left[ \big( \psi(t, x+u) - \psi(t,x) \big)^2 \right]
    \end{equation*}
    has the following properties:
    \begin{enumerate}
        \item there exists $h > 0$ such that $\sigma^2$ is concave and non-decreasing on $[0, h]$,
        \item for any $\gamma' \in (\gamma, \gamma_0)$, there exists $L > 0$ such that $| \sigma^2(u) | \leq L |u|^{2 \gamma'}$ for all $u \in [0, h]$.
    \end{enumerate}
    
    The second condition holds by \eqref{eq:psi-bound-1} with $L = L(\beta, \alpha, \delta)$, which is independent of $t$. For the first condition, we will prove it at the end of the proof. So for each time $t$, we have a second-order lift $\uppsi_t$ of $\psi_t$. Then combining \cite[Theorem~10.4]{MR4174393}, we get for every $t \in [0, T)$, 
    \begin{equation*}
        \E \big( \| \uppsi(t; \cdot, \cdot ) \|_{2 \gamma} \big) \lesssim L. 
    \end{equation*}
    Next, to get this bound uniformly in time, we need to apply Kolmogorov's continuity theorem \cite[Theorem~4.23]{MR4226142}. Interpolate the \eqref{eq:psi-bound-1} and \eqref{eq:psi-bound-2} to get for any $\varepsilon \in [0,1]$
    \begin{equation*}
        \mathbb{E} \left[ \left| \big( \psi(t, x+u) - \psi(s, x+u) \big) - \big( \psi(t, x) - \psi(s, x) \big) \right|^2 \right] \leq L |t-s|^{ 2 \beta' \varepsilon } |u|^{ 2 \gamma' (1 - \varepsilon) }.
    \end{equation*}
    Then it follows from \cite[Theorem~10.9]{MR4174393} that 
    \begin{equation*}
        \left\| R_{\psi(t,\cdot) - \psi(s,\cdot)} \right\|_{(1 / 2\gamma)-\mathrm{var}; [x, x']^2} \lesssim_{\rho} L |t-s|^{ 2 \beta' \varepsilon } |x-x'|^{ 2 \gamma }
    \end{equation*}
    for $\gamma \leq \gamma' (1 - \varepsilon) $ and follows from \cite[Theorem~10.5]{MR4174393} that 
    \begin{equation*}
        \left( \mathbb{E} \left[ \left\| \uppsi(t; \cdot,\cdot) - \uppsi(s; \cdot,\cdot) \right\|_{2 \gamma}^q \right] \right)^{1/q} \lesssim |t-s|^{ \beta' \varepsilon } L.
    \end{equation*}
    Finally, by Kolmogorov's continuity theorem \cite[Theorem~4.23]{MR4226142}, one obtains the estimate $\E \big( \| \uppsi \|_{2 \gamma, T} \big) \lesssim L$. Combining it with Theorem \ref{lem:reg-of-stationary-sol}, we get \eqref{eq:lift-bound}, as desired. 
    
    Now it just remains to show that $\sigma^2$ is concave and non-decreasing on some interval $[0, h]$. We get from \eqref{eq:psi-covariance} the formula 
    \begin{equation*}
        \sigma^2 ( u ) = \fC_\beta(0) \frac{2 n}{\pi} \sum_{k=1}^\infty k^{2 \delta - \frac{\alpha}{\beta}} \bigl(1- \cos(k u)\bigr),
    \end{equation*}
    and the concavity and non-decreasing on $[0, \frac{\pi}{2}]$ follows from a more general result that for any $\rho \in (0,1]$ the function
    \[
    \eta(u) \coloneqq \displaystyle \sum_{k=1}^\infty k^{ - \rho - 1} \big( 1 - \cos ( ku ) \big)
    \]
    is concave and non-decreasing on $[0, \frac{\pi}{2}]$. By the Dirichlet test, the derivative of $\eta(u)$ is just the sum of the derivatives of each term
    \begin{equation*}
        \eta'(u) = \sum_{k=1}^\infty \frac{\sin(ku) }{k^\rho}, 
    \end{equation*}
    which is easy to see $\eta(u)$ is non-decreasing on $[0, \frac{\pi}{2}]$. Thus, it remains to prove $\eta''(u) \leq 0$ on $(0, \frac{\pi}{2})$.

    Using $\sin(ku) = \frac{e^{iku} - e^{-iku}}{2i}$, we can then rewrite $\eta'$ as
    \begin{align*}
        \eta'(u) = \frac{1}{2i} \sum_{k=1}^\infty \left( \frac{e^{iku}}{ k^{\rho} } - \frac{e^{-iku}}{ k^{\rho} } \right).
    \end{align*}
    Recall that the polylogarithm is defined by $\Li_\rho (x) = \sum_{k=1}^\infty \frac{x^k}{k^\rho}$ for $\rho \in (0, 1]$ and $x \in \{ z \in \mathbb{C} : \|z\| \leq 1 , z \neq 1 \}$, and the following integral representation holds for $\rho > 0, x \in \mathbb{C} \setminus \{ z \in \R : z \geq 1 \}$: 
    \begin{equation*}
        \Li_\rho (e^{iu}) = \frac{1}{\Gamma(\rho)} \int_0^\infty \frac{x^{\rho-1}}{e^{x-iu} - 1 } \,\d x.
    \end{equation*}
    Then we get
    \begin{align*}
        \eta'(u) &= \frac{1}{2i \Gamma(\rho)} \int_0^\infty x^{\rho-1} \left( \frac{1}{e^{x-iu} -1 } - \frac{1}{e^{x+iu} -1 } \right)\d x \\
        &= \frac{1}{\Gamma(\rho)} \int_0^\infty \frac{1}{x^{1-\rho} } \frac{e^x \sin(u)}{e^{2x} - 2\cos(u) e^x + 1 } \,\d x.
    \end{align*}
Differentiating, yields 
    \begin{align*}
        \eta''(u)
        &= \frac{1}{\Gamma(\rho)} \int_0^\infty \frac{e^x}{x^{1-\rho} } \frac{\cos(u) ( e^{2x} +1 ) - 2 e^x}{( e^{2x} - 2\cos(u) e^x + 1)^2 } \,\d x \\
        &= \frac{1}{\Gamma(\rho)} \int_0^\infty \frac{1}{x^{1-\rho} } \frac{\cos(u) ( e^{x} + e^{-x} ) - 2}{( e^{x} + e^{-x} - 2\cos(u))^2 } \,\d x \\
        &= \frac{\cos(u) }{2 \Gamma(\rho)} \int_0^\infty \frac{1}{x^{1-\rho} } \frac{\cosh(x) - \frac{1}{\cos(u)}}{\big( \cosh(x) - \cos(u) \big)^2 } \,\d x
    \end{align*}
    for $u \in (0, \frac{\pi}{2})$. Then we get $\eta'' (u) \leq 0$ for $u \in (0, \frac{\pi}{2})$ from the following result: 
    \begin{equation}\label{eq:integral-is-negative}
        \int_0^\infty \frac{1}{x^{1-\rho}} \frac{\cosh(x) - 1 / v}{( \cosh(x) - v)^2 } \, \d x < 0 
    \end{equation}
    for any $\rho, v \in (0,1)$.

    It is left to prove \eqref{eq:integral-is-negative}. We start by proving the following identity:
\[
\int_0^\infty \frac{\cosh(x) - 1/v}{(\cosh(x) - v)^2 } \,\d x = - \frac{1}{v}.
\]
For this, we define the function $f(v) = \int_0^\infty \frac{1}{\cosh(x) - v} \,\d x$ and write the left-hand side as $f(v) + ( v-\frac{1}{v}) f'(v)$. By changing the integration variable to $e^x$, we can compute 
    \begin{align*}
        f(v) &= 2 \int_0^\infty \frac{e^x}{e^{2x} + 1 - 2v e^x } \,\d x = 2 \int_1^\infty \frac{1}{x^2 - 2vx + 1 } \,\d x,
    \end{align*}
    which is equal to $\frac{2}{ \sqrt{1-v^2} } \arctan( \frac{1+v}{ \sqrt{1-v^2} } )$. Therefore 
\[
f'(v) = \displaystyle \frac{2}{ \sqrt{1-v^2} } \left( \frac{v}{1-v^2} \arctan \left( \frac{1+v}{ \sqrt{1-v^2} } \right) + \frac{1}{2 \sqrt{1-v^2}} \right)
\]
and we get $f(v) + ( v-\frac{1}{v}) f'(v) = - \frac{1}{v}$, as desired.

Let us now look at the integral in \eqref{eq:integral-is-negative}. Setting $x^* \coloneqq \cosh^{-1} \left( \frac{1}{v} \right)$, we write the integral as
    \begin{align*} 
        &\int_0^{x^*} \frac{1}{x^{1-\rho}} \frac{\cosh(x) - 1/v}{( \cosh(x) - v)^2 } \, \d x + \int_{x^*}^\infty \frac{1}{x^{1-\rho}} \frac{\cosh(x) - 1/v}{( \cosh(x) - v)^2 } \, \d x \\
        &\qquad \leq \int_0^{x^*} \frac{1}{(x^*)^{1-\rho}} \frac{\cosh(x) - 1/v}{( \cosh(x) - v)^2 } \, \d x + \int_{x^*}^\infty \frac{1}{(x^*)^{1-\rho}} \frac{\cosh(x) - 1/v}{(\cosh(x) - v)^2 } \, \d x \\
        &\qquad = \frac{1}{(x^*)^{1-\rho}} \int_0^\infty \frac{\cosh(x) - 1 / v}{( \cosh(x) - v)^2 } \, dx = - \frac{1}{(x^*)^{1-\rho}} \frac{1}{v} < 0,
    \end{align*}
    where we used $\cosh(x) - \frac{1}{v} \leq 0$ when $x \in [0,x^*]$.
\end{proof}

\section{Schauder-type estimates}
\label{sec:schauder}

Now we discuss how the regularity of a function improves after convolution with the fractional heat kernel. We prove the result in a more general context. Namely, by analogy with the fractional heat kernel, defined in \eqref{eq:heat-kernel}, we consider a function $P: \R_+ \times \T \to \R$ whose Fourier transform is of the form 
\begin{equation}\label{eq:P-Fourier}
\widehat{P}_t(k) \coloneqq \frac{1}{ \sqrt{2\pi} } \int_{-\pi}^\pi P(t, x) e^{-ikx} \, \d x = \phi(t^\beta |k|^\alpha) 
\end{equation}
for $\alpha > 1$ and $\beta > \frac{1}{2}$, where the function $\phi: \R_+ \to \R$ has the following properties:

\begin{assumption}\label{assump:function-phi}
The function $\phi: \R_+ \to \R$ is positive, strictly decreasing, and has the following form and decay: 
\begin{equation}\label{eq:assumption-on-phi}
        \left| \frac{\d^j}{\d x^j} \phi(x) \right| \leq \frac{C}{ 1+ x^{\ell_\star+j}} 
\end{equation} 
for all $x \in \R_+$ and $j \in \N_0$, and for some $\ell_\star \in \N$. Moreover, $-\phi'(x)$ is positive and decreasing on $x \geq 0$. 
\end{assumption}

Then, for a function $f$ on $[0, T] \times \T$ of suitable regularity, which will be specified later, we want to study the regularity of the convolution 
\begin{equation*}
(t,x)\quad \mapsto\quad \int_0^t \int_{-\pi}^\pi P(t-s, x-y) f(s, y) \, \d y\, \d s.
\end{equation*}
Before proving a Schauder-type estimate, we need to obtain several auxiliary results. 

\begin{lemma}\label{L^1-bound of P}
    If $m \in \N_0$ satisfies $m < \alpha \ell_\star - 1$, then there is a constant $C=C(m) > 0$ such that
    \begin{equation*}
        \| \partial_x^m  P_t \|_{L^1(\T)} \leq C t^{ - \frac{\beta}{\alpha} m } 
    \end{equation*}
    uniformly in $t > 0$.
\end{lemma}

\begin{proof}
    Using the Fourier series and \eqref{eq:P-Fourier}, we write 
    \begin{equation}\label{eq1}
        \partial_x^m P(t, x) = \sum_{k \in \Z} (ik)^m \widehat{P}_t(k) e^{ikx} = \sum_{k \in \Z} (ik)^m \phi(t^\beta |k|^\alpha) e^{ikx}.
    \end{equation}
    The desired bound will follow from bounds on the function $P$ in the two regimes: $|x|<t^{\beta / \alpha}$ and $|x|\geq t^{\beta / \alpha}$, which require different analysis. 

    In the case $|x|<t^{\beta / \alpha}$, we simply use the assumption \eqref{eq:assumption-on-phi} and get
    \begin{align*}
        | \partial_x^m P(t, x) | 
         \leq \sum_{k \in \Z} |k|^m \phi(t^\beta |k|^\alpha) \lesssim \sum_{k \in \Z} \frac{|k|^m}{1 + (t^\beta |k|^\alpha)^{\ell_\star}}.
    \end{align*}
    Estimating the sum by the respective integral, we bound this expression by a constant times 
    \begin{equation}
    \int_0^\infty \frac{k^m}{1 + (t^\beta k^\alpha)^{\ell_\star}} \, \d k = t^{-\frac{\beta}{\alpha} (m+1) } \int_0^{\infty} \frac{k^m}{1+ k^{\alpha \ell_\star}} \, \d k \lesssim t^{-\frac{\beta}{\alpha} (m+1)}, \label{eq2}
    \end{equation}
    where we rescaled the integration variable by $t^{\beta/\alpha}$. The integral converges because of our assumption $\alpha \ell_\star > m + 1$. 

    In the case $|x| \geq t^{\beta / \alpha}$, we define the discrete Laplacian 
    \[
    \underline{\Delta}\, g(k) \coloneqq g(k+1) - 2 g(k) + g(k-1).
    \]
    Then we have $\underline{\Delta}\, e^{ikx} = - 4\sin^2 (x / 2) e^{ikx}$. Expressing $e^{ikx}$ from the right-hand side and substituting into \eqref{eq1}, we get
    \begin{equation}\label{eq:P-discrete-Laplace}
        \partial_x^m P(t, x) = - \frac{1}{4\sin^2 (x / 2)} \sum_{k \in \Z} (ik)^m \phi(t^\beta |k|^\alpha) \left( \underline{\Delta} e^{ikx} \right). 
    \end{equation}
    We note that the discrete Laplacian satisfies the summation by parts formula 
    \[
    \sum_{k \in \Z} f(k) \underline{\Delta} g(k) = \sum_{k \in \Z} \underline{\Delta} f(k) g(k),
    \]
    provided the sums converge. Applying it to the preceding expression, we get 
    \begin{equation}\label{eq:P-bound-1}
        | \partial_x^m P(t, x) | \leq \frac{1}{4\sin^2 (x / 2)} \sum_{k \in \Z} \left| \underline{\Delta} \left[ k^m \phi(t^\beta |k|^\alpha) \right] \right|. 
    \end{equation}
    We bound $\underline{\Delta} \left[ k^m \phi(t^\beta |k|^\alpha) \right]$ using the mean value theorem. For this, we need to distinguish the cases $k=0$ and $k \neq 0$.

    For $k=0$ we have 
    \begin{equation}\label{eq:Delta-phi-k0}
        | \underline{\Delta} \left[ k^m \phi(t^\beta |k|^\alpha) \right](0) | = | 1^m \phi(t^\beta) - 2 \cdot 0^m \phi(0) + (-1)^m \phi(t^\beta) |
    \end{equation}
    with the convention $0^0=1$. The latter vanishes if $m$ is odd. If $m$ is a non-zero even, then we bound \eqref{eq:Delta-phi-k0} by $2 | \phi(t^\beta) | \lesssim \frac{1}{1+ t^{\beta \ell_\star}} \lesssim t^{- \beta \ell_\star} \wedge 1$, where we made use of our assumption \eqref{eq:assumption-on-phi}. In the case $m=0$ we use the mean value theorem to bound \eqref{eq:Delta-phi-k0} by 
    \[
    2| t^\beta \phi'(t^\beta \eta) | \lesssim \frac{t^\beta}{1+ \eta^{(\ell_\star+1)} t^{\beta(\ell_\star+1)}} \wedge 1 \lesssim t^\beta \wedge 1 \lesssim t^{\frac{\beta}{\alpha}}
    \]
    for some $\eta \in [0, 1]$, where we used our assumption \eqref{eq:assumption-on-phi} again. In all these cases, we can bound \eqref{eq:Delta-phi-k0} by a constant times $t^{-\frac{\beta}{\alpha} (m-1)}$ since $\alpha > 1$ and $m < \alpha \ell_\star -1$. 

    For $|k| \geq 1$ we consider only $k \geq 1$, while the analysis for $k \leq -1$ is the same. We apply the mean value theorem to bound 
    \begin{align*}
        | \underline{\Delta} \left[ k^m \phi(t^\beta k^\alpha) \right] | \lesssim \left| \frac{\d^2}{\d \eta^2} \left[ \eta^m \phi(t^\beta \eta^\alpha) \right] \right|
    \end{align*}
    for some point $\eta \in [k-1, k+1]$. Computing the derivative and using the assumption \eqref{eq:assumption-on-phi}, we bound this expression by a constant multiple of 
    \begin{align*}
        &|\eta|^{m-2} \phi(t^\beta \eta^\alpha) + t^\beta |\eta|^{m + \alpha-2} |\phi'(t^\beta \eta^\alpha)| + t^{2\beta} |\eta|^{m + 2 \alpha-2} |\phi''(t^\beta \eta^\alpha)| \\
        &\qquad \lesssim \frac{m(m-1) k^{m-2}}{1+ (t^\beta k^\alpha)^{\ell_\star}} + \frac{m t^\beta k^{m + \alpha-2}}{1+ (t^\beta k^\alpha)^{\ell_\star+1}} + \frac{t^{2\beta} k^{m + 2 \alpha-2}}{1+ (t^\beta k^\alpha)^{\ell_\star+2}}.
    \end{align*}

    Using the preceding estimates, we bound \eqref{eq:P-bound-1} by a constant times 
    \[
   \frac{1}{4\sin^2 (x / 2)} \left( t^{ - \frac{\beta}{\alpha} (m-1) } + \sum_{k \neq 0} \left( \frac{m(m-1) k^{m-2}}{1+ (t^\beta k^\alpha)^{\ell_\star}} + \frac{m t^\beta k^{m + \alpha-2}}{1+ (t^\beta k^\alpha)^{\ell_\star+1}} + \frac{t^{2\beta} k^{m + 2 \alpha-2}}{1+ (t^\beta k^\alpha)^{\ell_\star+2}}\right) \right).
    \]
    Estimating the sum by the respective integral and rescaling the integration variable by $t^{\beta/\alpha}$, we get a bound of order
    \[
    \frac{t^{ - \frac{\beta}{\alpha} (m-1) }}{\sin^2 (x / 2)} \left( 1 + \int_{0}^\infty \left( \frac{m(m-1) k^{m-2}}{1+ k^{\alpha \ell_\star}} + \frac{m k^{m + \alpha-2}}{1+ k^{\alpha(\ell_\star+1)}} + \frac{k^{m + 2 \alpha-2}}{1+ k^{\alpha(\ell_\star+2)}}\right) \d k \right).
    \]
    The expression in the parentheses is finite, and we get 
    \begin{equation}\label{eq3}
    | \partial_x^m P(t, x) | \lesssim \frac{t^{ - \frac{\beta}{\alpha} (m-1) }}{\sin^2 (x / 2)}.
    \end{equation}

    Now we can go back to estimate the $L^1$-norm of $P(t, x)$ by using \eqref{eq2} and \eqref{eq3}. We get
    \begin{align*}
        \| \partial_x^m P(t, \cdot) \|_{L^1(\T)} 
        &= \int_{|x|< (t^{\beta / \alpha} \wedge \pi)} | \partial_x^m P(t, x) | \, \d x + \int_{t^{\beta / \alpha} < |x| \leq \pi} | \partial_x^m P(t, x) | \, \d x \\
        &\lesssim t^{-\frac{\beta}{\alpha} m} + t^{ - \frac{\beta}{\alpha} (m-1) } \int_{t^{\beta / \alpha} < |x| \leq \pi} \frac{\d x}{\sin^2 (x / 2)}.
    \end{align*}
    We use the inequality $\sin(\theta) \geq \frac{\theta}{2}$ for $0 \leq \theta \leq \frac{\pi}{2}$ to estimate the preceding expression by 
    \[
    t^{-\frac{\beta}{\alpha} m} + 16 t^{ - \frac{\beta}{\alpha} (m-1) } \int_{t^{\beta / \alpha} < |x| \leq \pi} \frac{\d x}{x^2},
    \]
    which is of order $t^{-\frac{\beta}{\alpha} m}$.
\end{proof}

The following lemma studies the time dependence of the fractional heat kernel, which will lead to the result of time continuity in Schauder-type estimates. 

\begin{lemma}\label{L^1-bound of P-time}
    In the setting of Lemma~\ref{L^1-bound of P}, for any $\varepsilon \in [0, \beta]$ one has
    \begin{equation}
        \| \partial_x^m  ( P_{t'} - P_t ) \|_{L^1(\T)} \leq C (t'-t)^\varepsilon t^{-\frac{\beta}{\alpha} m - \varepsilon} \label{bdd: L^1-norm of P_t - P_t'}
    \end{equation}
    uniformly in $0 < t < t'$.
\end{lemma}

\begin{proof}
Using Lemma~\ref{L^1-bound of P}, we have a trivial bound 
\[
\| \partial_x^m  ( P_{t'} - P_t ) \|_{L^1(\T)} \leq \| \partial_x^m P_{t'} \|_{L^1(\T)} + \| \partial_x^m P_t \|_{L^1(\T)} \lesssim (t')^{ - \frac{\beta}{\alpha} m } + t^{ - \frac{\beta}{\alpha} m } \lesssim t^{ - \frac{\beta}{\alpha} m }.
\]
If $t' \geq 2t$ then $t'-t \geq t$ and we can bound it by $(t'-t)^\varepsilon t^{ - \frac{\beta}{\alpha} m - \varepsilon}$ for any $\varepsilon \geq 0$.

Let us now consider the case $t < t' < 2t$. Let $Q_{t,t'}(x) := P(t', x) - P(t, x)$. Similarly to the proof of Lemma~\ref{L^1-bound of P}, we write 
    \begin{equation*}
        \partial_x^m Q_{t,t'}(x) = \sum_{k \in \Z} (ik)^m \widehat{ Q_{t,t'} }(k) e^{ikx} = \sum_{k \in \Z} (ik)^m ( \phi((t')^\beta |k|^\alpha) - \phi(t^\beta |k|^\alpha) ) e^{ikx},
    \end{equation*}
    and as in the preceding proof, we bound this function in different ways in different regimes. 

    For $|x|<t^{\beta / \alpha}$, we use the assumption \eqref{eq:assumption-on-phi} and get
    \begin{align*}
        | \partial_x^m Q_{t,t'}(x) | 
        &\leq \sum_{k \in \Z} |k|^m | \phi((t')^\beta |k|^\alpha) - \phi(t^\beta |k|^\alpha) | \\
          &\lesssim \sum_{k \in \Z} |k|^m \frac{ ((t')^\beta - t^\beta) |k|^\alpha }{ 1 + (t^\beta |k|^\alpha)^{\ell_\star+1} } \lesssim (t'-t)^{\beta} \sum_{k \in \Z} \frac{ |k|^{m + \alpha} }{ 1 + (t^\beta |k|^\alpha)^{\ell_\star+1} } .
    \end{align*}
    As in \eqref{eq2}, we bound this expression by a constant multiple of
    \begin{equation}
        (t'-t)^{\beta} \int_0^\infty \frac{ |k|^{m + \alpha} }{ 1 + (t^\beta |k|^\alpha)^{\ell_\star+1} } \, \d k 
        \lesssim (t'-t)^{\beta} t^{-\frac{\beta}{\alpha} (m + \alpha +1)}. \label{eq2-new}
    \end{equation}

In the case $|x| \geq t^{\beta / \alpha}$, we use the same computations as in \eqref{eq:P-discrete-Laplace} and write 
\begin{equation*}
        \partial_x^m Q_{t,t'}(x) = - \frac{1}{4\sin^2 (x / 2)} \sum_{k \in \Z} (ik)^m ( \phi((t')^\beta |k|^\alpha) - \phi(t^\beta |k|^\alpha) ) \left( \underline{\Delta} e^{ikx} \right). 
\end{equation*}
Estimating the function $\phi$ as above and repeating the derivation of \eqref{eq3}, we get 
\begin{equation*}
    | \partial_x^m Q_{t,t'}(x) | \lesssim (t'-t)^{\beta} \frac{t^{ - \frac{\beta}{\alpha} (m + \alpha-1) }}{\sin^2 (x / 2)}.
\end{equation*}
Recalling that we consider $t < t' < 2t$, we have $t'-t < t$ and hence 
\begin{equation}\label{eq3-new}
    | \partial_x^m Q_{t,t'}(x) | \lesssim (t'-t)^{\varepsilon} \frac{t^{ - \frac{\beta}{\alpha} (m-1) - \varepsilon}}{\sin^2 (x / 2)},
\end{equation}
for any $\varepsilon \in [0, \beta]$. The bound \eqref{bdd: L^1-norm of P_t - P_t'} follows from \eqref{eq2-new} and \eqref{eq3-new} in the same way as in the proof of Lemma~\ref{L^1-bound of P}.
\end{proof}

\begin{defn}\label{def:concave-down}
We call a sequence $\{ a_k \}_{k=L}^M$ ``concave down" if there exists an integer $L \leq k_0 \leq M$ such that 
\[
a_L \leq a_{L+1} \leq \cdots \leq a_{k_0}, \qquad\qquad a_{k_0} \geq a_{k_0+1} \geq \cdots \geq a_{M}.
\]
\end{defn}

We are going to use the Littlewood--Paley theory to measure the regularity of functions, and the following lemmas will be useful for this purpose. 

\begin{lemma}
    \label{L^1-bound_for_sum_of_cos/sin}
    For two natural numbers $L \leq M$, 
    suppose that $\{ b_k \}_{k=L}^M$ is an increasing positive sequence and $\{ c_k \}_{k=L}^M$ is a decreasing positive sequence. Let $\{ a_k \}_{k=L}^M$ be a $[0,1]$-valued ``concave down" sequence. Then there is a universal constant $C > 0$, such that 
    \begin{equation}\label{eq:sum-of-cosines}
        \left\| \sum_{k=L}^M a_k b_k c_k \cos(k\, \cdot) \right\|_{L^1(\T)} \leq C b_M c_L \bigl(1 + \log(M-L+1)\bigr)  
    \end{equation}
    and
    \begin{equation}\label{eq:sum-of-sines}
        \left\| \sum_{k=L}^M a_k b_k c_k \sin(k\, \cdot) \right\|_{L^1(\T)} \leq C b_M c_L \bigl(1 + \log(M-L+1)\bigr).  
    \end{equation}
\end{lemma}

\begin{proof}
Let us first prove \eqref{eq:sum-of-cosines} in the case when the sequence $a_k$ takes only the value $1$. We use summation by parts to separate $b_k$'s out:
\begin{equation}\label{eq:sum-of-cosines-0}
\sum_{k=L}^M b_k c_k \cos(kx) = b_L \sum_{k=L}^M c_k \cos(kx) + \sum_{l=L+1}^M ( b_l - b_{l-1} ) \sum_{k=l}^M c_k \cos(kx).
\end{equation}
Since the sequence $b_k$ is increasing, we bound the $L^1$-norm of this function by 
    \begin{align}
        &b_L \left\| \sum_{k=L}^M c_k \cos(k\, \cdot) \right\|_{L^1(\T)} + \sum_{l=L+1}^M ( b_l - b_{l-1} ) \left\| \sum_{k=l}^M c_k \cos(k\, \cdot) \right\|_{L^1(\T)} \notag \\
        &\hspace{7cm} \leq b_M \sup_{L \leq l \leq M} \left\| \sum_{k=l}^M c_k \cos(k\, \cdot) \right\|_{L^1(\T)}. \label{eq:sum-of-cosines-1}
    \end{align}
    We apply summation by parts again and get 
    \begin{align}
        \left\| \sum_{k=l}^M c_k \cos(k\, \cdot) \right\|_{L^1(\T)} 
        &= \left\| \sum_{m=l}^{M-1} ( c_m - c_{m+1} ) \sum_{k=l}^m \cos(k\, \cdot) + c_M \sum_{k=l}^M \cos(k\, \cdot) \right\|_{L^1(\T)} \notag \\
        &\leq \sum_{m=l}^{M-1} ( c_m - c_{m+1} ) \left\| \sum_{k=l}^m \cos(k\, \cdot) \right\|_{L^1(\T)} + c_M \left\| \sum_{k=l}^M \cos(k\, \cdot) \right\|_{L^1(\T)} \notag \\
        &\leq c_l \sup_{l \leq m \leq M} \left\| \sum_{k=l}^m \cos(k\, \cdot) \right\|_{L^1(\T)}, \label{eq:sum-of-cosines-2}
    \end{align}
    where the last bound holds because the sequence $c_k$ is decreasing. To estimate the $L^1$ norm, we need two different bounds for the sum. The first one is a trivial bound 
    \begin{equation}
        \left| \sum_{k=l}^m \cos(kx) \right| \leq \sum_{k=l}^m \left| \cos(kx) \right| \leq m-l+1. 
        \label{sum_cos_bound_for_x_small}
    \end{equation}
    The second one uses a trigonometric identity 
    \begin{equation*}
        \sum_{k=l}^m \cos(kx) = \frac{ \sin ( ( m+ \frac{1}{2} ) x ) - \sin ( ( l- \frac{1}{2} ) x ) }{2 \sin(x / 2)}
    \end{equation*}
    for $x \notin 2 \pi \Z$. 
    Combining it with the inequality $| \sin (x/2)| \geq \frac{|x|}{4}$ for $x \in [-\pi, \pi]$, we have a second bound 
    \begin{equation}
        \left| \sum_{k=l}^m \cos(kx) \right| \leq \frac{4}{|x|}
        \label{sum_cos_bound_for_x_large}
    \end{equation}
    for $x \neq 0$. 

    Using the two bounds \eqref{sum_cos_bound_for_x_large} and \eqref{sum_cos_bound_for_x_small} on $x$ large and small, respectively, we have the following estimate: 
    \begin{align*}
        \left\| \sum_{k=l}^m \cos(k\, \cdot) \right\|_{L^1(\T)} 
        &= \int_{ |x| \leq \frac{1}{m-l+1}} \left| \sum_{k=l}^m \cos(kx) \right| \,\d x + \int_{  \frac{1}{m-l+1} \leq |x| \leq \pi } \left| \sum_{k=l}^m \cos(kx) \right| \, \d x \\
        &\leq \int_{ |x| \leq \frac{1}{m-l+1}} (m-l+1) \,\d x + 2 \int_{\frac{1}{m-l+1} }^\pi \frac{4}{x} \, \d x \\
        &= 2 + 8 \log \left( (m-l+1) \pi \right). 
    \end{align*}
    Applying this estimate to \eqref{eq:sum-of-cosines-2} and using \eqref{eq:sum-of-cosines-0}-\eqref{eq:sum-of-cosines-1}, we get 
    \begin{equation}\label{eq:sum-of-cosines-proof}
        \left\| \sum_{k=L}^M b_k c_k \cos(k\, \cdot) \right\|_{L^1(\T)} \leq C b_M c_L \bigl(1 + \log(M-L+1)\bigr).
    \end{equation}

    Now, we will prove \eqref{eq:sum-of-cosines} for a ``concave down" sequence $a_k$. Let $L \leq k_0 \leq M$ be a value as in Definition~\ref{def:concave-down}. Then 
    \[
    \left\| \sum_{k=L}^M a_k b_k c_k \cos(k\, \cdot) \right\|_{L^1(\T)} \leq \left\| \sum_{k=L}^{k_0} a_k b_k c_k \cos(k\, \cdot) \right\|_{L^1(\T)} + \left\| \sum_{k=k_0+1}^M a_k b_k c_k \cos(k\, \cdot) \right\|_{L^1(\T)}.
    \]
    For $L \leq k \leq k_0$, the sequence $a_k$ is increasing, and we can bound the first $L^1$ norm on the right-hand side by \eqref{eq:sum-of-cosines-proof}, but with the sequence $a_k b_k$ in place of $b_k$. Furthermore, for $k_0+1 \leq k \leq M$, the sequence $a_k$ is decreasing, and we bound the last $L^1$ norm by \eqref{eq:sum-of-cosines-proof} but replacing $c_k$ by $a_k c_k$. Hence, we bound the preceding expression by 
    \[
    C a_{k_0} b_{k_0} c_L \bigl(1 + \log(k_0-L+1)\bigr) + C b_M a_{k_0+1} c_{k_0+1} \bigl(1 + \log(M-k_0)\bigr).
    \]
    Recalling that the values $a_k$ are bounded by $1$, and the sequences $b_k$ and $c_k$ are monotonic, we estimate this expression by 
    \[
    C b_{M} c_L \Bigl(2 + \log(k_0-L+1) + \log(M-k_0)\Bigr) \leq 2 C b_{M} c_L \bigl(1 + \log(M-L+1)\bigr),
    \]
    which is the desired bound \eqref{eq:sum-of-cosines}.

    The bound \eqref{eq:sum-of-sines} is proved in exactly the same way, but using the identity 
    \begin{equation*}
        \sum_{k=l}^m \sin(kx) = \frac{ \cos ( ( l- \frac{1}{2} ) x ) - \cos ( ( m+ \frac{1}{2} ) x ) }{2 \sin (x/2)}
    \end{equation*}
    for $x \notin 2 \pi \Z$.
\end{proof}

Using the preceding result, we can estimate how a convolution with the function $P$ changes regularity. See Appendix~\ref{sec:besov space} for a review of the Littlewood--Paley theory. 

\begin{lemma}\label{Sch-est w/o t}
    Let $\gamma \in [0, \infty)$ and $f \in \CC^\gamma (\T)$. Then for any $m \in \N_0$ satisfying $m < \gamma + \alpha \ell_\star$, any $\widetilde{\gamma} \in (0, \infty) \setminus \N$ satisfying $0 \leq \widetilde{\gamma} - \gamma + m < \alpha \ell_\star$ and any $\kappa \in (0, \gamma - \widetilde{\gamma} + \alpha \ell_\star - m]$ one has 
    \begin{equation}\label{eq:Sch-est w/o t}
        \left\| \partial_x^m (P_t * f) \right\|_{\CC^{\widetilde{\gamma}}} \leq C t^{- \frac{\beta}{\alpha} ( \widetilde{\gamma} - \gamma + m + \kappa ) } \| f \|_{\CC^{\gamma}}
    \end{equation}
    uniformly over $t > 0$, where $C = C(\kappa, m) > 0$. 

    Furthermore, for any $\varepsilon \in [0, \beta]$ one has 
    \begin{equation}\label{eq:Sch-est w/o t for t-t'}
        \left\| \partial_x^m (P_{t'} - P_t) * f \right\|_{\CC^{\widetilde{\gamma}}} \leq C (t'-t)^{\varepsilon} t^{ - \frac{\beta}{\alpha} ( \widetilde{\gamma} - \gamma +m +\kappa) - \varepsilon} \| f \|_{\CC^\gamma},
    \end{equation}
    uniformly over $0 < t < t'$.
\end{lemma}

We note that the derivative on the left-hand side is always defined in terms of distributions, and we do not have to restrict $m$. We did it only because we work with functions in our applications, i.e., with $\widetilde{\gamma} > 0$. Although we do not characterize the $\CC^n$ spaces as the Besov spaces, we can use the fact that $\CC^n \subsetneq \CB^n_{\infty, \infty}$ for $n \in \N_0$ to allow $\gamma$ to be an integer in the assumption of Lemma~\ref{Sch-est w/o t}. 

\begin{proof}
We are going to use the characterization of the spaces $\CC^\gamma$ in terms of the Littlewood--Paley theory as explained in Appendix~\ref{sec:besov space}. We have 
\[
\left\| \partial_x^m (P_t * f) \right\|_{\CC^{\widetilde{\gamma}}} \lesssim \left\| P_t * f \right\|_{\CC^{\widetilde{\gamma} + m}},
\]
and we need to estimate the Littlewood--Paley blocks $\Delta_j ( P_t * f)$, what we will do in the two cases: $t < 2^{ -j \frac{\alpha}{\beta} }$ and $t \geq 2^{ -j \frac{\alpha}{\beta} }$.

First, we deal with the case $t < 2^{ -j \frac{\alpha}{\beta} }$. We have
    \begin{equation}
        \| \Delta_j ( P_t * f) \|_{L^\infty}
        = \left\| \sum_{k \in \Z } \varphi(2^{-j} k) \widehat{P_t}(k) \widehat{f}(k) e^{ik \cdot } \right\|_{L^\infty} = \left\| P_t \ast \Delta_j f \right\|_{L^\infty} \leq \left\| P_t \right\|_{L^1} \left\| \Delta_j f \right\|_{L^\infty} \label{eq: block bound-t small},
    \end{equation}
    where the last bound holds by Young's convolution inequality. Recalling that $\left\| \Delta_j f \right\|_{L^\infty} \leq 2^{-\gamma j} \| f \|_{\CC^\gamma}$, we estimate it furthermore by  
    \begin{align*}
        2^{-\gamma j} \left\| P_t \right\|_{L^1} \| f \|_{\CC^\gamma} &\lesssim 2^{-\gamma j} \| f \|_{\CC^\gamma},
    \end{align*}
    where we used Lemma~\ref{L^1-bound of P} with the value $m=0$. Since $2^j < t^{ - \frac{\beta}{\alpha} }$, we conclude that
    \[
    2^{ j (\widetilde{\gamma} + m)} \left\| \Delta_j (P_t * f ) \right\|_{L^\infty} \lesssim 2^{(\widetilde{\gamma} + m-\gamma) j} \| f \|_{\CC^\gamma} \lesssim t^{ - \frac{\beta}{\alpha} ( \widetilde{\gamma} - \gamma + m ) } \| f \|_{\CC^\gamma}
    \]
    for $\widetilde{\gamma} + m \geq \gamma$.

     In the case $t \geq 2^{ -j \frac{\alpha}{\beta} }$, we bound the Littlewood--Paley block in the following way: 
    \begin{align}
        \| \Delta_j (P_t * f) \|_{L^\infty}
        &= \left\| \sum_{k \in \Z } \varphi(2^{-j} k) \widehat{P_t}(k) \widehat{f}(k) e^{ikx} \right\|_{L^\infty} \notag \\
          &= \left\| \sum_{k \in \Z } \varphi(2^{-j} k) \widetilde{\varphi}(2^{-j} k) \widehat{P_t}(k) \widehat{f}(k) e^{ikx} \right\|_{L^\infty} \notag \\
        &= \| \Delta_j P_t \ast \widetilde{\Delta}_j f \|_{L^\infty} \leq \| \Delta_j P_t \|_{L^1} \| \widetilde{\Delta}_j f \|_{L^\infty}, \label{eq: block bound-t big}
    \end{align}
    where we again used Young's convolution inequality. Here, $\widetilde{\Delta}_j f$ denotes the Littlewood--Paley block associated with a different dyadic partition of unity, determined by $\widetilde{\varphi}$ and $\widetilde{\chi}$, chosen such that $\widetilde{\varphi}$ and $\widetilde{\chi}$ are identically equal to $1$ on the supports of $\varphi$ and $\chi$, respectively. One can see the explanation in Remark~\ref{rem: dyadic partition of unity}.  

To proceed with the proof, we need the following estimate
    \begin{equation}
        \left\| \Delta_j P_t \right\|_{L^1(\T)} \leq C \frac{1 + \log 2^j}{1+ (t^{\beta} 2^{j \alpha})^{\ell_\star}}
        \label{eq:L^1-norm of LP-block  of derivative of P}
    \end{equation}
    for any $j \geq -1$ and with a constant $C > 0$. Let us now prove it.
    
    For $j=-1$, we can bound it directly by a constant. For $j \geq 0$, we use the definition of the Littlewood--Paley block \eqref{eq:PL-block} and \eqref{eq:P-Fourier} to write
    \begin{align}\label{eq:P-block-proof}
        \Delta_j P_t (x) 
        &= \sum_{k \in \mathbb{Z}} \varphi(2^{-j} k) \widehat{P} ( t , k) e^{ikx} 
        = 2 \displaystyle\sum_{ \frac{3}{4} 2^j \leq k \leq \frac{8}{3} 2^{j} } \varphi(2^{-j} k) \phi( t^\beta k^\alpha) \cos(kx),
    \end{align}
where we used the fact that $\varphi$ is supported in the annulus $\{ \xi \in \R : \frac{3}{4} \leq |\xi| \leq \frac{8}{3} \}$. Moreover, $\varphi$ may be chosen to be increasing on $[0,\xi_0]$ and decreasing on $[\xi_0, \infty)$ for some $\xi_0 \in \R$. For example, the cutoff function constructed in Appendix~\ref{sec:besov space} has these properties. In addition, $\phi$ is decreasing by the assumption made at the beginning of this section. Thus we apply Lemma~\ref{L^1-bound_for_sum_of_cos/sin} with $a_k = \varphi(2^{-j} k)$, $b_k = 1$, and $c_k = \phi( t^\beta k^\alpha)$ to get 
    \begin{align*}
        \left\| \Delta_j P_t  \right\|_{L^1(\T)} 
        &\lesssim \phi( t^\beta ( \tfrac{3}{4} 2^{j} )^\alpha ) \big( 1 + \log 2^{j} \big).
    \end{align*} 
    Using \eqref{eq:assumption-on-phi}, we estimate it by a constant times $\frac{1 + \log 2^{j}}{1+ (t^{\beta} 2^{j \alpha})^{\ell_\star}}$ and get the bound \eqref{eq:L^1-norm of LP-block  of derivative of P}.
    
     Now we go back to the proof of the lemma. Applying \eqref{eq:L^1-norm of LP-block  of derivative of P} to \eqref{eq: block bound-t big} yields 
    \begin{align*}
        2^{ j (\widetilde{\gamma} + m)} \left\| \Delta_j ( P_t * f) \right\|_{L^\infty} &\lesssim 2^{ j (\widetilde{\gamma} + m - \gamma)} \frac{1 + \log 2^j}{1+ (t^{\beta} 2^{j \alpha})^{\ell_\star}}  \| f \|_{\CC^\gamma} \\
        &\lesssim 2^{j(\widetilde{\gamma} - \gamma - \ell_\star \alpha + m)} t^{ -\ell_\star \beta} \big( 1 + \log 2^{j}) \| f \|_{\CC^\gamma},
    \end{align*}
    where in the last bound we used our assumption $t^{\beta} 2^{j \alpha} \geq 1$. Bounding $1 + \log 2^j \lesssim 2^{j \kappa}$ for $\kappa > 0$, we estimate the preceding expression by a constant times 
    \begin{align*}
        2^{j(\widetilde{\gamma} - \gamma - \ell_\star \alpha + m + \kappa)} t^{ -\ell_\star \beta} \| f \|_{\CC^\gamma} \lesssim t^{ - \frac{\beta}{\alpha} ( \widetilde{\gamma} - \gamma + m + \kappa )} \| f \|_{\CC^\gamma},
    \end{align*}
where we used our assumption $2^{j} \geq t^{-\frac{\beta}{\alpha}}$ with $\kappa \in (0, \gamma - \widetilde{\gamma} + \alpha \ell_\star - m]$.

    Combining the preceding bounds on the Littlewood--Paley blocks of $P_{t} * f$, we get claim \eqref{eq:Sch-est w/o t}.

    To prove \eqref{eq:Sch-est w/o t for t-t'}, we set $c_\star = (\frac{7}{4})^{\alpha/\beta} > 1$ and consider the two cases $t' \geq c_\star t$ and $t' > c_\star t$. This choice of the constant will be clear from our analysis below and will allow us to use Lemma~\ref{L^1-bound_for_sum_of_cos/sin}.
    
    For $t' \geq c_\star t$, we use \eqref{eq:Sch-est w/o t} and get the trivial bound 
    \begin{align*}
    &\left\| \partial_x^m (P_{t'} - P_t) * f \right\|_{\CC^{\widetilde{\gamma}}} \leq \left\| \partial_x^m (P_{t'} * f)\right\|_{\CC^{\widetilde{\gamma}}} + \left\| \partial_x^m (P_t * f) \right\|_{\CC^{\widetilde{\gamma}}} \\
    &\qquad \lesssim \left((t')^{- \frac{\beta}{\alpha} ( \widetilde{\gamma} - \gamma + m + \kappa ) } + t^{- \frac{\beta}{\alpha} ( \widetilde{\gamma} - \gamma + m + \kappa ) }\right) \| f \|_{\CC^{\gamma}} \lesssim t^{- \frac{\beta}{\alpha} ( \widetilde{\gamma} - \gamma + m + \kappa ) } \| f \|_{\CC^{\gamma}}.
    \end{align*}
    We use $t \leq \frac{t'-t}{c_\star-1}$ to estimate it by a constant times $(t'-t)^\varepsilon t^{- \frac{\beta}{\alpha} ( \widetilde{\gamma} - \gamma + m + \kappa ) - \varepsilon} \| f \|_{\CC^{\gamma}}$ for any $\varepsilon \geq 0$. 

    Now, we consider the case $t < t' < c_\star t$, in which we bound 
\[
\left\| \partial_x^m (P_{t'} - P_t) * f \right\|_{\CC^{\widetilde{\gamma}}} \lesssim \left\| (P_{t'} - P_t) * f \right\|_{\CC^{\widetilde{\gamma} + m}},
\]
and we need to estimate the Littlewood--Paley blocks of the function inside the norm. As above, we are going to bound the blocks in the two regimes: $t < 2^{ -j \frac{\alpha}{\beta} }$ and $t \geq  2^{ -j \frac{\alpha}{\beta} }$. In the case $t < 2^{ -j \frac{\alpha}{\beta} }$ we proceed as in \eqref{eq: block bound-t small} and using Lemma~\ref{L^1-bound of P-time} we get
\begin{align*}
        2^{ j (\widetilde{\gamma} + m)} \| \Delta_j ((P_{t'} - P_t) * f) \|_{L^\infty} &\leq 2^{ j (\widetilde{\gamma} + m)} \| (P_{t'} - P_t) \|_{L^1} \left\| \Delta_j f \right\|_{L^\infty} \\
        &\lesssim (t'-t)^\varepsilon 2^{ j (\widetilde{\gamma} - \gamma + m)} t^{ - \varepsilon} \| f \|_{\CC^\gamma}.
\end{align*}
For $\widetilde{\gamma} - \gamma + m \geq 0$ we recall that $t < 2^{ -j \frac{\alpha}{\beta} }$ and estimate it by 
\begin{align*}
        2^{ j (\widetilde{\gamma} + m)} \| \Delta_j ((P_{t'} - P_t) * f) \|_{L^\infty} 
        \lesssim (t'-t)^\varepsilon t^{ - \frac{\beta}{\alpha} ( \widetilde{\gamma} - \gamma + m ) - \varepsilon} \| f \|_{\CC^\gamma}.
\end{align*}

In the case $t \geq 2^{ -j \frac{\alpha}{\beta} }$, we can derive a bound similar to \eqref{eq:L^1-norm of LP-block  of derivative of P}. As in \eqref{eq:P-block-proof}, we get 
\begin{align*}
        \Delta_j (P_{t'} - P_t) (x)  = 
            - 2 \displaystyle\sum_{ \frac{3}{4} 2^j \leq k \leq \frac{8}{3} 2^{j} } \varphi(2^{-j} k) \bigl(\phi( t^\beta k^\alpha) - \phi( (t')^\beta k^\alpha)\bigr) \cos(kx).
    \end{align*}
The only difference from the proof of \eqref{eq:L^1-norm of LP-block  of derivative of P} is that now we are working with $\phi( t^\beta k^\alpha) - \phi( (t')^\beta k^\alpha)$ in place of $\phi( t^\beta k^\alpha)$. To use Lemma~\ref{L^1-bound_for_sum_of_cos/sin}, as we did above, we need the sequence  $c_k = \phi( t^\beta k^\alpha) - \phi( (t')^\beta k^\alpha)$ to be positive and decreasing. To see this, we write 
\[
c_k = - \int_{t^\beta k^\alpha}^{(t')^\beta k^\alpha} \phi'(\lambda)\, \d \lambda.
\]
By Assumption~\ref{assump:function-phi}, the function $-\phi'(\lambda)$ is positive and hence the sequence $c_k$ is positive. Moreover, our choice of $c_\star$ yields $(t')^\beta k^\alpha \leq t^\beta (k+1)^\alpha$ for any $\frac{3}{4} 2^j \leq k \leq \frac{8}{3} 2^{j}$, and the integration intervals for different $c_k$ do not intersect. Since $-\phi'(\lambda)$ is decreasing, we get $c_k \geq c_{k+1}$. Hence, Lemma~\ref{L^1-bound_for_sum_of_cos/sin} yields 
\begin{align*}
        \left\| \Delta_j (P_{t'} - P_t) \right\|_{L^1(\T)} 
        &\lesssim \bigl(\phi( t^\beta ( \tfrac{3}{4} 2^{j} )^\alpha ) - \phi( (t')^\beta ( \tfrac{3}{4} 2^{j} )^\alpha )\bigr) \big( 1 + \log 2^{j} \big) \\
        &\lesssim (t'-t)^\beta 2^{j\alpha} \frac{1 + \log 2^{j}}{1+ (t^{\beta} 2^{j \alpha})^{\ell_\star+1}},
    \end{align*} 
where we used \eqref{eq:assumption-on-phi}. Using $1 + \log 2^{j} \lesssim 2^{j\kappa}$ for $\kappa > 0$, we bound it furthermore by a constant multiple of 
\[
(t'-t)^\beta 2^{-j \alpha \ell_\star + j\kappa} t^{- \beta(\ell_\star+1)}.
\]
Recall that we have $t < t' < c_\star t$, which means that $t'-t \leq  (c_\star -1)t$ and this expression is bounded by a constant times
\[
(t'-t)^\varepsilon 2^{-j \alpha \ell_\star + j\kappa} t^{- \beta\ell_\star - \varepsilon}
\]
for any $\varepsilon \in [0,\beta]$. Hence, we get 
    \begin{align*}
        2^{ j (\widetilde{\gamma} + m)} \left\| \Delta_j (P_{t'} - P_t) * f \right\|_{L^\infty} 
        &\lesssim (t'-t)^\varepsilon 2^{j(\widetilde{\gamma} - \gamma - \ell_\star \alpha + m + \kappa)} t^{- \beta\ell_\star - \varepsilon} \| f \|_{\CC^\gamma}.
    \end{align*}
Using assumption $t \geq c 2^{ -j \frac{\alpha}{\beta} }$, we bound it by a constant multiple of 
\[
(t'-t)^\varepsilon t^{ - \frac{\beta}{\alpha} ( \widetilde{\gamma} - \gamma + m + \kappa ) - \varepsilon} \| f \|_{\CC^\gamma},
\]
for any $\kappa \in (0, \gamma - \widetilde{\gamma} + \alpha \ell_\star - m]$.

The estimates on the Littlewood--Paley blocks give the desired bound.
\end{proof}

\subsection{Schauder-type estimates for fractional heat kernel}

We are going to consider convolutions of the fractional kernel with functions given by the special form $\partial^{1-\beta}_s f(s,\cdot)$. That is why we need to prove a version of Schauder-type estimates for such functions. We will do it only for the fractional heat kernel $P_{\beta, \alpha}$ defined in \eqref{eq:heat-kernel} rather than a general kernel considered above. Such functions arise when we apply the Duhamel principle to solve the fractional equation
\begin{equation}\label{eq:fractional-equation}
    \partial_t^\beta X(t,x) = - (-\Delta)^{\frac{\alpha}{2}} X(t,x) + f(t,x)
\end{equation}
with the initial data $X(0,\cdot) \equiv 0$. The solution can be written as 
\begin{equation*}
    X(t,x) = \int_0^t \left( P_{\beta, \alpha}(t-s) \ast \partial^{1-\beta}_s f(s) \right)(x) \, \d s,
\end{equation*}
using the fractional heat kernel. 
This formula is well-defined for a function $f$ which is $\CC^1$ in the time variable (see the Definition~\ref{def: Caputo derivative}), but for functions of lower temporal regularity the expression on the right-hand side is undefined. We can make sense of it by using the fractional integration by parts \eqref{eq:fractional-IBP}. Namely, for a temporarily $\CC^1$ function $f$ we can write
\begin{equation}\label{eq:heat-solution}
    X(t,x) = \int_0^t \left(\RL_{t-, s}^{1-\beta} P_{\beta, \alpha}(t-s) \ast f(s) \right)(x) \, \d s + \left[\bigl(I_{t-, s}^{\beta} P_{\beta, \alpha}(t-s) \ast f(s) \bigr)(x) \right]_{s=0}^t.
\end{equation}
The expression $\RL_{t-, s}^{1-\beta} P_{\beta, \alpha}(t-s)$ means that the Riemann--Liouville derivative $\RL_{t-}^{1-\beta}$ is applied to the variable $s$ in the kernel. In the case when the temporal regularity of $f$ is lower than $\CC^1$, we use \eqref{eq:heat-solution} as the definition of the solution to the fractional equation \eqref{eq:fractional-equation} whenever it is defined. We show in Appendix~\ref{sec:Duhamel} that the kernels on the right-hand side are well defined and the expression makes sense if the kernels can be convolved with the function $f$.

The following result studies the regularity of this function. 

\begin{prop} \label{sch est for fractional heat eq}
    Let the values $\gamma, \gamma_{\ini} \in [0, \infty)$, $\widetilde{\gamma} \in (0, \infty) \setminus \N$ and $m \in \N_0$ satisfy $\gamma -m < \widetilde{\gamma} < \min( \gamma, \gamma_{\ini} ) + \alpha -m$ and let $0 < \kappa < \min( \gamma, \gamma_{\ini} ) + \alpha -m - \widetilde{\gamma}$. For a fixed $T > 0$, let $f$ be a continuous function on $[0,T] \times \T$ with $f(0, \cdot) \in \CC^{\gamma_{\ini}}( \T )$ and $f(s, \cdot) \in \CC^\gamma( \T )$ for all $s \in (0, T]$, and satisfying 
    \[
    \|f\|_{ \CC^{\gamma, \lambda}_T} < \infty
    \]
    for some $\lambda \in [0,1)$, where we use the norm \eqref{def: C gamma t-lambda}. Then there is a constant $C = C(\kappa, m)$ such that 
    \begin{equation}\label{eq:sch est for fractional heat eq}
        \| \partial_x^m X(t) \|_{\CC^{\widetilde{\gamma}} } \leq C \left(t^{ \beta - \frac{\beta}{\alpha}(\widetilde{\gamma} - \gamma + m + \kappa) - \lambda } \|f\|_{ \CC^{\gamma, \lambda}_t} + t^{ \beta - \frac{\beta}{\alpha} ( \widetilde{\gamma} - \gamma_{\ini} + m +\kappa ) } \left\| f(0) \right\|_{\CC^{\gamma_{\ini}}}\right)
    \end{equation}
    holds uniformly over $t \in (0, T]$ and 
    \begin{align}\label{eq:sch est for fractional heat eq for t - t'}
        & \| \partial_x^m X(t) - \partial_x^m X(t') \|_{\CC^{\widetilde{\gamma}} } \\
          &\qquad \leq C (t-t')^{\varepsilon} \left((t')^{ \beta - \frac{\beta}{\alpha}(\widetilde{\gamma} - \gamma + m + \kappa) - \lambda - \varepsilon } \|f\|_{ \CC^{\gamma, \lambda}_t } + (t')^{\beta - \frac{\beta}{\alpha} ( \widetilde{\gamma} - \gamma_{\ini} +m +\kappa ) - \varepsilon } \left\| f(0) \right\|_{\CC^{\gamma_{\ini}}} \right) \notag
    \end{align}
    holds uniformly over $t' < t \in (0, T]$ for any $\varepsilon > 0$ small enough.

    Furthermore, we have $\partial_x^m X \in \CC([0,T]; \CC^{ \widetilde{\gamma}}(\T))$ if 
    \[
    \beta - \frac{\beta}{\alpha}(\widetilde{\gamma} - \gamma + m + \kappa) - \lambda > 0, \qquad \beta - \frac{\beta}{\alpha} ( \widetilde{\gamma} - \gamma_{\ini} +m +\kappa ) > 0,
    \]
    and $\partial_x^m X \in \CC((0,T]; \CC^{ \widetilde{\gamma}}(\T))$ otherwise.
\end{prop}

\begin{proof}
We need to show that $P_{\beta, \alpha}$ satisfies the assumption of Section~\ref{sec:schauder}. From \eqref{eq:heat-kernel} we have the Fourier transform \eqref{eq:P-Fourier} with the function $\phi(x) = \frac{1}{\sqrt{2\pi}} E_{\beta}(- x)$. All the properties from Assumption~\ref{assump:function-phi} follow readily from the properties of the Mittag--Leffler function $E_{\beta}$. In particular \eqref{lem:ML-bound} yields \eqref{eq:assumption-on-phi} with $\ell_\star=1$.

Recalling our definition of the function \eqref{eq:heat-solution}, we need to bound the three terms on the right-hand side. Let us denote them by $X_1(t)$, $X_2(t)$ and $X_3(t)$ respectively, so that $X(t) = X_1(t) + X_2(t) - X_3(t)$.

Using the Fourier series, we write $X_1$ as
\[
X_1(t,x) = \frac{1}{ \sqrt{2 \pi} } \int_0^t \sum_{k \in \Z} \widehat{\RL_{t-, s}^{1-\beta} P_{\beta, \alpha}(t-s)}(k) \, \widehat{f}(s,k) e^{ikx}\, \d s.
\]
Since the Fourier transform commutes with the fractional derivative, we use \eqref{eq:heat-kernel} to write it further as 
\[
\frac{1}{2 \pi} \int_0^t \sum_{k \in \Z} \RL_{t-, s}^{1-\beta} E_{\beta}(- (t-s)^\beta |k|^\alpha) \, \widehat{f}(s,k) e^{ikx}\, \d s.
\]
Using Lemma~\ref{lem:frac_D_of_ML} we compute the fractional derivative and get
\begin{align}
&\frac{\beta}{2 \pi} \int_0^t \sum_{k \in \Z} (t-s)^{\beta - 1} E'_{\beta}(-(t-s)^\beta |k|^\alpha) \, \widehat{f}(s,k) e^{ikx}\, \d s \notag \\
&\qquad = \frac{\beta}{2 \pi} \int_0^t (t-s)^{\beta - 1} \left(P_{1}(t-s) \ast f(s) \right)(x)\, \d s \label{eq: expression of X_1},
\end{align}
with a kernel $P_1$ defined by its Fourier transform as 
\[
\widehat{P_1}(s,k) = \frac{\beta}{\sqrt{2\pi}} E'_{\beta}(-s^\beta |k|^\alpha).
\]
We note that Lemma~\ref{lem:ML-bound} implies that this kernel has the properties \eqref{eq:P-Fourier} and \eqref{eq:assumption-on-phi} with $\ell_\star=2$. Applying then Lemma \ref{Sch-est w/o t} with the kernel $P_1$ in place of $P$ (we can interchange $\partial_x^m$ and $\int_0^t$ because $\partial_x^m (P_1 \ast f)(x)$ is continuous) we get
    \begin{align*}
        \| \partial_x^m X_1(t) \|_{\CC^{\widetilde{\gamma}}} 
        &\lesssim \int_0^t (t-s)^{\beta-1 - \frac{\beta}{\alpha}(\widetilde{\gamma} - \gamma + m + \kappa)} \| f(s, \cdot) \|_{\CC^\gamma} \, \d s \\
          &\lesssim \int_0^t (t-s)^{\beta-1 - \frac{\beta}{\alpha}(\widetilde{\gamma} - \gamma + m + \kappa)} s^{- \lambda} \|f\|_{ \CC^{\gamma, \lambda}_t} \, \d s = C t^{ \beta - \frac{\beta}{\alpha}(\widetilde{\gamma} - \gamma + m + \kappa) - \lambda } \|f\|_{ \CC^{\gamma, \lambda}_t}, 
    \end{align*}
    where the integral converges if $\widetilde{\gamma} - \gamma < \alpha - m - \kappa$ and $\lambda < 1 $ and the constant $C$ is given in terms of the Beta function by 
    \begin{equation*}
        C = B\left( \beta - \frac{\beta}{\alpha}(\widetilde{\gamma} - \gamma + m + \kappa) , 1-\lambda  \right). 
    \end{equation*}
    Now, we are going to bound the function $X_2$. We have 
    \begin{align*}
    X_2(t,x) &= \lim_{s \nearrow t} \bigl(I_{t-, s}^{\beta} P_{\beta, \alpha}(t-s) \ast f(s) \bigr)(x) = \lim_{s \nearrow t} \frac{1}{ \sqrt{2 \pi} } \sum_{k \in \Z} \widehat{I_{t-, s}^{\beta} P_{\beta, \alpha}(t-s)}(k) \, \widehat{f}(s,k) e^{ikx} \\
    &= \lim_{s \nearrow t} \frac{1}{2 \pi} \sum_{k \in \Z} I_{t-,s}^{\beta} E_{\beta}( - (t-s)^\beta |k|^\alpha) \, \widehat{f}(s,k) e^{ikx},
    \end{align*}
    where we used \eqref{eq:heat-kernel} again. Swapping the limit and the sum (because the summation converges uniformly near $t$) and using Lemma~\ref{lem:LM-integral}, which says that the fractional integral of the Mittag--Leffler function vanishes in the limit, we conclude that $X_2(t,x) \equiv 0$. 

    Finally, we will analyze $X_3$. Proceeding as before, we get
    \begin{align*}
    X_3(t,x) &= \lim_{s \searrow 0} \bigl(I_{t-, s}^{\beta} P_{\beta, \alpha}(t-s) \ast f(s) \bigr)(x) = \lim_{s \searrow 0} \frac{1}{ \sqrt{2 \pi} } \sum_{k \in \Z} \widehat{I_{t-, s}^{\beta} P_{\beta, \alpha}(t-s)}(k) \, \widehat{f}(s,k) e^{ikx} \\
        &= \lim_{s \searrow 0} \frac{1}{2 \pi} \sum_{k \in \Z} I_{t-,s}^{\beta} E_{\beta}( - (t-s)^\beta |k|^\alpha) \, \widehat{f}(s,k) e^{ikx}.
    \end{align*}
    Using Lemma~\ref{lem:LM-integral} we write it furthermore as 
    \begin{align*}
    \lim_{s \searrow 0} \frac{1}{2 \pi} \sum_{k \in \Z \setminus \{0\}} \frac{1}{|k|^\alpha} \bigl( E_{\beta}( - (t-s)^\beta |k|^\alpha) -1 \bigr) \widehat{f}(s,k) e^{ikx} - \lim_{s \searrow 0} \frac{1}{2 \pi} \frac{(t-s)^\beta}{\Gamma(1+\beta)}\widehat{f}(s,0).
    \end{align*}
    Swapping the sum and the limit, we get
    \begin{align}
        X_3(t,x) &= \frac{1}{2 \pi} \sum_{k \in \Z \setminus \{0\}} \frac{1}{|k|^\alpha} \bigl( E_{\beta}( - t^\beta |k|^\alpha) -1 \bigr) \widehat{f}(0,k) e^{ikx} - \frac{1}{2 \pi} \frac{t^\beta}{\Gamma(1+\beta)}\widehat{f}(0,0) \notag \\
        &= t^\beta \left(P_{3}(t) \ast f(0) \right)(x), \label{eq: expression of X_3}
    \end{align}
    where the kernel $P_{3}$ is defined via the Fourier transform as 
    \[
    \widehat{P_3}(t,k) = \frac{1}{\sqrt{2\pi}} \frac{1}{ t^{\beta} |k|^\alpha } \bigl( E_{\beta}( - t^\beta |k|^\alpha) -1 \bigr)
    \]
    for $k \neq 0$ and $\widehat{P_3}(t,0) = - \frac{1}{\sqrt{2\pi} \Gamma(1+\beta)}$. In order to apply Lemma~\ref{Sch-est w/o t} to this kernel, we need to show that it has the required properties \eqref{eq:P-Fourier} and \eqref{eq:assumption-on-phi}. For this, we write 
    \begin{equation*}
    \widehat{P_3}(t, k) = - \phi(t^\beta |k|^\alpha) 
    \end{equation*}
    with the function 
    \begin{equation*} 
    \phi(x) = 
    \begin{cases}
    \frac{1}{\sqrt{2\pi}} \frac{1}{x} \bigl(1- E_{\beta}( - x) \bigr) & \text{for} ~ x \neq 0, \\
    \frac{1}{\sqrt{2\pi} \Gamma(1+\beta)} & \text{for} ~ x = 0.
    \end{cases}
    \end{equation*}
    As follows from the properties of the Mittag--Leffler functions, this function is smooth on $\R_+$, positive, and strictly decreasing (we can see it from differentiating the function twice). The bound \eqref{eq:assumption-on-phi} with $\ell_\star = 1$ follows from \eqref{eq:ML-bound}. Hence, Lemma~\ref{Sch-est w/o t} yields 
    \[
    \| \partial_x^m X_3(t) \|_{\CC^{\widetilde{\gamma}} } \lesssim t^{ \beta - \frac{\beta}{\alpha} ( \widetilde{\gamma} - \gamma_{\ini} + m +\kappa ) } \left\| f(0) \right\|_{\CC^{\gamma_{\ini}}}.
    \]
    
    Combining the proved bounds on the processes $X_1$, $X_2$ and $X_3$, we get the desired bound \eqref{eq:sch est for fractional heat eq}.
    
    Next, we will prove the time dependence of the Schauder-type estimate. For this, we will bound the time increments of $X_1$ and $X_3$ (recall that $X_2(t)$ vanishes). From \eqref{eq: expression of X_1}, we get 
    \begin{align}\nonumber
        \partial_x^m \left( X_1(t) - X_1(t') \right)
        &= \frac{\beta}{2 \pi} \partial_x^m \int_0^{t'} \big[ (t-s)^{\beta - 1} - (t'-s)^{\beta - 1} \big] \left( P_{1}(t-s) \ast f(s) \right)(x)\, \d s \nonumber \\
          &\quad + \frac{\beta}{2 \pi} \partial_x^m \int_0^{t'} (t'-s)^{\beta - 1} \left( \big[ P_{1}(t-s) - P_{1}(t'-s) \big] \ast f(s) \right)(x)\, \d s \nonumber \\
        &\quad + \frac{\beta}{2 \pi} \partial_x^m \int_{t'}^{t} (t-s)^{\beta - 1} \left( P_{1}(t-s) \ast f(s) \right)(x)\, \d s. \label{eq: X_1(t) - X_1(t')}
    \end{align}
    For $t' < t$, we have 
    $\big| (t'-s)^{\beta - 1} - (t-s)^{\beta - 1} \big| 
      \lesssim (t-t')^{ \varepsilon} (t'-s)^{ \beta-1 - \varepsilon}$
    for any $\varepsilon \in [0,1-\beta]$. Then by \eqref{eq:Sch-est w/o t}, the $\CC^{\widetilde{\gamma}}$-norm of the first term in \eqref{eq: X_1(t) - X_1(t')} can be bounded by a constant multiple of
    \begin{align}
         \int_0^{t'} (t-t')^{ \varepsilon} &(t'-s)^{ \beta-1 - \varepsilon} (t-s)^{ - \frac{\beta}{\alpha}(\widetilde{\gamma} - \gamma + m + \kappa)} s^{- \lambda} \|f\|_{ \CC^{\gamma, \lambda}_t} \, \d s \notag \\
        &\qquad \lesssim (t-t')^{ \varepsilon} (t')^{ \beta - \frac{\beta}{\alpha}(\widetilde{\gamma} - \gamma + m + \kappa) - \lambda - \varepsilon} \|f\|_{ \CC^{\gamma, \lambda}_t}, \label{eq: bound of X_1(t) - X_1(t')-1} 
    \end{align}
    where $C = B\left( \beta - \frac{\beta}{\alpha}(\widetilde{\gamma} - \gamma + m + \kappa) - \varepsilon, 1-\lambda  \right)$. Using \eqref{eq:Sch-est w/o t for t-t'}, we can bound the $\CC^{ \widetilde{\gamma} }$ of the second term in \eqref{eq: X_1(t) - X_1(t')} by a constant times 
    \begin{align}
        \int_0^{t'} (t'-s)^{\beta - 1}  &(t-t')^{\varepsilon} (t'-s)^{ - \frac{\beta}{\alpha} ( \widetilde{\gamma} - \gamma +m +\kappa ) - \varepsilon } s^{- \lambda} \|f\|_{ \CC^{\gamma, \lambda}_t} \, \d s \notag \\
          &\qquad = C (t-t')^{\varepsilon} {t'}^{ \beta - \frac{\beta}{\alpha}(\widetilde{\gamma} - \gamma + m + \kappa) - \lambda - \varepsilon } \|f\|_{ \CC^{\gamma, \lambda}_t }, \label{eq: bound of X_1(t) - X_1(t')-2}
    \end{align}
    where $C = B\left( \beta - \frac{\beta}{\alpha}(\widetilde{\gamma} - \gamma + m + \kappa) - \varepsilon , 1-\lambda  \right)$. We can bound the $\CC^{ \widetilde{\gamma} }$-norm of the third term in \eqref{eq: X_1(t) - X_1(t')} by a constant times 
    \begin{align}
        \int_{t'}^{t} (t-s)^{ \beta - 1 - \frac{\beta}{\alpha}(\widetilde{\gamma} - \gamma + m + \kappa) } s^{- \lambda} \|f\|_{ \CC^{\gamma, \lambda}_t } \, \d s
         \leq (t-t')^{\varepsilon} t^{ \beta - \frac{\beta}{\alpha}(\widetilde{\gamma} - \gamma + m + \kappa) - \lambda - \varepsilon } \|f\|_{ \CC^{\gamma, \lambda}_t }, \label{eq: bound of X_1(t) - X_1(t')-3} 
    \end{align}
    where $C = B\left( \beta - \frac{\beta}{\alpha}(\widetilde{\gamma} - \gamma + m + \kappa) - \varepsilon , 1-\lambda  \right)$ and $\varepsilon \geq 0$. Therefore combining \eqref{eq: bound of X_1(t) - X_1(t')-1}, \eqref{eq: bound of X_1(t) - X_1(t')-2}, and \eqref{eq: bound of X_1(t) - X_1(t')-3} together, we then get 
    \begin{equation*}
        \| \partial_x^m \left( X_1(t) - X_1(t') \right) \|_{\CC^{ \widetilde{\gamma} }} \lesssim (t-t')^{\varepsilon} t^{ \beta - \frac{\beta}{\alpha}(\widetilde{\gamma} - \gamma + m + \kappa) - \lambda - \varepsilon } \|f\|_{ \CC^{\gamma, \lambda}_t }
    \end{equation*}
    for any $\varepsilon \in [0, 1-\beta]$ (recall that $\beta \in (\frac{1}{2}, 1]$). 
    
    We use \eqref{eq: expression of X_3} to write the time increment of $X_3$ as 
    \begin{align*} 
        & \partial_x^m \left( X_3(t) - X_3(t') \right) \notag \\
          &\qquad = (t^\beta - {t'}^\beta)\partial_x^m   \left(P_{3}(t) \ast f(0) \right) + {t'}^\beta \partial_x^m \left( (P_{3}(t) - P_{3}(t')) \ast f(0) \right). 
    \end{align*}
    Using the similar argument as above, we bound the $\CC^{ \widetilde{\gamma} }$-norms of these two terms by a constant multiple of $(t-t')^{\varepsilon} (t')^{ \beta - \frac{\beta}{\alpha} ( \widetilde{\gamma} - \gamma_{\ini} +m +\kappa ) - \tilde \varepsilon } \left\| f(0) \right\|_{\CC^{\gamma_{\ini}}}$ for any $\varepsilon \in [0, 1 - \beta]$.

Combining \eqref{eq:sch est for fractional heat eq} and \eqref{eq:sch est for fractional heat eq for t - t'}, we get the last statement of this proposition.
\end{proof}

\section{Definition of solutions and well-posedness} \label{sec: definition of sol and main proof}

In this section, we define a mild solution of equation \eqref{eq:FSBL} and prove its existence and uniqueness. 

Recall that the stationary solution $\psi$ is not sufficiently smooth to allow the nonlinear term $G(u) \partial_x u$ in \eqref{eq:FSBL} to be defined classically. Indeed, the solution $u$, provided that it exists in a reasonable sense, is expected to have a regularity no better than that of $\psi$, and $G(\psi) \partial_x \psi$ is not defined. 

Therefore, in order to define the solution $u$, we use the Da Prato--Debussche ansatz \cite{MR2016604} and write 
\begin{equation}\label{eq:DPD-ansatz}
u = \psi + \tilde v,
\end{equation}
where $\tilde v$ has a higher regularity than $\psi$. The idea of the method, which was first used in \cite{MR2016604} and then was applied to a stochastic Burgers-type equation in \cite{HairerRoughSPDEs}, is to formally define an equation for $\tilde v$ and show that it can be solved rigorously. Then the solution $u$ to the equation \eqref{eq:FSBL} is defined as \eqref{eq:DPD-ansatz}.

Subtracting the two equations \eqref{eq:FSBL} and \eqref{eq:FSH}, we formally get 
\begin{equation}\label{eq:equation-for-tilde-v}
    \partial_t^{\beta} \tilde v = - (-\Delta)^{\alpha/2} \tilde v + F(\psi + \tilde v) + G(\psi + \tilde v) \partial_x (\psi + \tilde v) - (-\Delta)^{\delta/2} I_t^{1-\beta} \bar \Pi_0 \xi
\end{equation}
with the initial condition $\tilde v_0 = u_0 - \psi_0$, 
where the operator $\bar \Pi_0 = \Id - \Pi_0$ leaves only the zero-Fourier mode of the function/distribution it is applied to. We note that the last term is independent of the spatial variable and hence it vanishes if $\delta > 0$. In the case $\delta = 0$, if we apply Duhamel's principle to this equation, the last term will be convolved in space and time with a fractional heat kernel. Since the heat kernel integrates to $1$, we will get
\[
- \int_0^t (\bar \Pi_0 \xi)(s)\, \d s,
\]
which is an $n$-dimension standard Brownian motion. Indeed, it is a Gaussian process with independent increments, and we can use It{\^o} isometry to compute its covariance 
\begin{align*}
    &\E \left[ \int_0^t (\bar \Pi_0 \xi)(s)\, \d s \otimes \int_0^t (\bar \Pi_0 \xi)(s)\, \d s \right] \\
      &\qquad\qquad\qquad = \frac{1}{2 \pi} \E \left[ \int_0^t \int_{\T} \xi(s,x)\, \d s\, \d x \otimes \int_0^t \int_{\T} \xi(s',x')\, \d s'\, \d x' \right]  \\
    &\qquad\qquad\qquad = \frac{1}{2 \pi} \int_0^t \int_{\T} \int_0^t \int_{\T} \E [ \xi(s,x) \otimes \xi(s,x) ] \, \d s\, \d x\, \d s'\, \d x' = t \, \Id. 
\end{align*}
So, it is an $n$-dimensional Brownian motion. 

Hence, it will be convenient to subtract this term from equation \eqref{eq:equation-for-tilde-v} by setting
\[
v(t,x) = \tilde v(t,x) - B_t, \qquad B_t = - \1{\delta = 0} \int_0^t (\bar \Pi_0 \xi)(s)\, \d s,
\]
which satisfies 
\begin{equation}\label{eq:equation-for-v}
    \partial_t^{\beta} v = - (-\Delta)^{\alpha/2} v + F(\psi + v + B) + G(\psi + v + B) \partial_x (\psi + v + B) \xi
\end{equation}
with the initial condition $v_0 = u_0 - \psi_0$. 
As we will see below, this equation can be solved if we make sense of the product $G(\psi + v) \partial_x \psi$, because the solution will be $v \in \CC^1$ and the product $G(\psi + v) \partial_x v$ is defined in the classical sense. We use the lift $\Psi = (\psi, \uppsi)$ of $\psi$ to a $\gamma$-rough path in $\CD^{\gamma} ( [-\pi,\pi], \R^n)$, defined in Lemma~\ref{lem:lift}, and the rough path integral to define the function 
\begin{equation}\label{def: CZ(t,x)}
    \CZ_v(t,x) \coloneqq \int_{-\pi}^x G(\psi + v + B)(t,z)\, \d_z \Psi(t,z),
\end{equation}
where the integration is with respect to the variable $z$. Having the expression on the right-hand side defined, we need to have $\psi_t \in \CC^\gamma$ with $\frac{1}{3} < \gamma \leq \frac{1}{2}$, $v_t \in \CC^1$ and $(G (\psi_t + v_t + B_t), DG(\psi_t + v_t + B_t))$ to be a rough path controlled by $\Psi_t$ (see Proposition~\ref{prop:controlled-function} and the comment above it). Then we postulate 
\begin{align*}
        \bigl(G(\psi + v + B) \partial_x \psi\bigr)(t,x) \coloneqq \partial_x \CZ_v(t,x),  
\end{align*}
where the derivative is in the sense of distributions. We note that if we convolve this distribution with a heat kernel, we can use integration by parts to move the spatial derivative onto the kernel. This motivates our definition of a solution. 

\begin{defn}\label{def: sol of FSBL}
    A continuous stochastic process $u$ is a mild solution to \eqref{eq:FSBL} on the time interval $[0, T]$ if for any $t \in (0, T]$ it can be written as $u_t = \psi_t + v_t + B_t$, where the process $v$ satisfies  
    \begin{equation}\label{contraction_M}
    \begin{aligned}
        v(t, x) &= \partial_x \int_0^t \int_{-\pi}^\pi P_{\beta, \alpha}(t-s, x-y) \; \partial_s^{1-\beta} \CZ_v(s,y) \, \d y\, \d s \\
          &\qquad+ \int_0^t \int_{-\pi}^\pi P_{\beta, \alpha}(t-s, x-y) \partial_s^{1-\beta} \big[ G ( \psi + v + B) \partial_y v \big] (s,y) \, \d y\, \d s \\
        &\qquad+ \int_0^t \int_{-\pi}^\pi P_{\beta, \alpha}(t-s, x-y) \partial_s^{1-\beta} F (\psi(s,y) + v(s,y) + B_s) \, \d y\, \d s \\
          &\qquad+ \int_{-\pi}^\pi P_{\beta, \alpha}(t, x-y) ( u_0(y) - \psi_0(y)) \, \d y
    \end{aligned}
    \end{equation}
    almost surely. 
\end{defn}

\begin{rem} 
    The forcing terms considered in \eqref{contraction_M} are not, in general, $\CC^1$ in time, and hence the Caputo derivative cannot be applied to them directly in the classical sense. We therefore define all three integrals above via the integration-by-parts formula \eqref{eq:heat-solution}. See the discussion surrounding \eqref{eq:heat-solution} for further details.
\end{rem}

With the notion of a solution in place, we now derive several auxiliary results which will be used to prove Theorem~\ref{thm:main}, the main result of this article. 

For a sufficiently regular function $v$, let $\CM(v)$ denote the expression on the right-hand side of \eqref{contraction_M}. Our goal is to apply the Banach fixed-point theorem to the map $\CM$ on a suitable space, and for this we need to show that $\CM$ is a contraction map. Because we consider two types of initial conditions in Assumption~\ref{assump:initial}, we will study the map $\CM$ on different Banach spaces.

To specify the spaces, we fix two values 
\begin{equation} \label{def: gamma_*}
     \gamma_* \in (\tfrac{1}{3}, \gamma_0), \qquad \kappa_0 \in (0, \alpha + \gamma_* - 2),
\end{equation}
where the former will be used as the spatial regularity of $\psi$. We fix furthermore a value 
\[
\widetilde{\gamma} \in (1, \gamma_* + \alpha -1),
\]
where the interval is non-empty due to Assumption~\ref{assump:parameters}. Then we define, for any $T > 0$, the Banach spaces 
\begin{align*}
    \CB_T 
    &\coloneqq \left\{ v \in \CC ( [0,T] ; \CC^{1}(\T) ) \cap \CC ( (0,T] ; \CC^{\widetilde{\gamma}}(\T) ) : \| v \|_{\CB_T} < \infty \right\}
\end{align*}
and 
\begin{align*}
      \widetilde \CB_T 
      &\coloneqq \left\{ v \in \CC ( [0,T] ; \CC^{\gamma}(\T) ) \cap \CC ( (0,T] ; \CC^{\widetilde{\gamma}}(\T) ) : \| v \|_{\widetilde \CB_T} < \infty \right\}
\end{align*}
with the respective norms
\begin{align}
    \| v \|_{\CB_T} 
    &\coloneqq \| v \|_{\CC^1_T} + \| v \|_{\CC^{ \widetilde{\gamma}, \frac{\beta}{\alpha} ( \widetilde{\gamma} - 1 + \kappa_0 )}_T} \label{def: CB_T norm}
\end{align}
and 
\begin{align}
      \| v \|_{\widetilde \CB_T} 
      &\coloneqq \| v \|_{\CC^{\gamma_*}_T} + \sum_{ \eta = 2 \gamma_*, 1, \widetilde{\gamma} }  \| v \|_{\CC^{\eta, \frac{\beta}{\alpha} ( \eta - \gamma_* + \kappa_0 )}_T}, \label{def: tilde CB_T norm}
\end{align}
where we use the weighted norms \eqref{def: C gamma t-lambda}. Basically, it means that the higher the regularity we look at, the faster the blow-up rate $v_t$ has as $t \to 0$. The spaces $\CB_T$ and $\widetilde \CB_T$ will be used to prove a Banach space fixed-point argument under the Assumption~\ref{assump:initial}\eqref{assump: initial 1st kind} and \eqref{assump: initial 2nd kind}, respectively. 

In order to perform the Picard iterations for the map $\CM$, we need to have a priori bounds on the involved processes. For this, we use a constant $K > 0$ and introduce the following stopping time:
\begin{align}\label{stop time sigma_K}
    \sigma_K 
    &\coloneqq \inf \{ t \geq 0 : \vvvert \Psi \vvvert_{\gamma_*, t} \geq K \} \wedge \inf \{ t \geq 0 : |B_t| \geq K \}.
\end{align}
Then, for any time $t \geq 0$, we define its restriction 
\begin{equation}
    t_K \coloneqq t \wedge \sigma_K \label{stop time t_K}
\end{equation}
which is random. The advantage of using this time is that we have the bounds 
\begin{equation}\label{eq:noise-bounded}
  \vvvert \Psi \vvvert_{\gamma_*, t} \leq K, \qquad |B_t| \leq K
\end{equation}
on the interval $t \in [0, T_K]$. In particular, for $\| v\|_{\CB_{T_K}} \leq L$ or $\| v\|_{\widetilde\CB_{T_K}} \leq L$, we will often use the bounds 
\begin{equation}\label{eq:bound-on-G}
\| G^{(m)}(\psi_t + v_t + B_t) \|_{L^\infty} \leq \sup_{|z| \leq 2K + L} |G^{(m)}(z)|,
\end{equation}
for $0 \leq m \leq 2$ and $t \in [0, T_K]$, on the function appearing in \eqref{eq:equation-for-v}.

We are going to prove that if the initial condition satisfies Assumption~\ref{assump:initial}\eqref{assump: initial 1st kind}, then $\CM$ is a contraction map on a ball in $\CB_{T_K}$ with a suitable $T > 0$; namely, we will prove 
\begin{equation*}
    \left\| \CM(v) \right\|_{\CB_{T_K}} \leq L,
    \qquad \qquad \| \CM(v) - \CM(\overline{v}) \|_{\CB_{T_K}} \leq \frac{1}{2} \| v - \overline{v} \|_{\CB_{T_K}}
\end{equation*}
provided $\| v \|_{\CB_{T_K}} \leq L$ and $\| \overline{v} \|_{\CB_{T_K}} \leq L$ for some $L > 0$ and some $T > 0$ depending on $K$ and $L$. We will also establish analogous estimates in the space $\widetilde{\CB}_{T_K}$ in the case when the initial condition satisfies Assumption~\ref{assump:initial}\eqref{assump: initial 2nd kind}. Then, applying the Banach fixed point theorem, we conclude that $\CM$ has a unique fixed point on the time interval $[0, T_K]$ in both $\CB_{T_K}$ and $\widetilde{\CB}_{T_K}$, given the respective initial conditions.

Let us write 
\begin{equation}\label{def: M_i}
    \CM(v) = \sum_{i=1}^4 \CM_i(v),
\end{equation}
where $\CM_i(v)$ is the $i$-th term on the right-hand side of \eqref{contraction_M}. In the following sections, we will bound these terms in the space $\CB_{T_K}$ one by one. To bound the map $\CM$ in the $\widetilde\CB_{T_K}$ norm, we will rewrite it in the form \eqref{def: M_i-tilde} by combining $\CM_1$ and $\CM_2$ together. A precise explanation for this can be found around~\eqref{def: M_i-tilde}.

\subsection{Estimates on $\CZ$}
\label{sec:Z-bounds}

We have assumed that the functions $F$ and $G$ are sufficiently regular, but we have not assumed their derivatives to be uniformly bounded. Since they are locally bounded, we introduce the following local norms:
\begin{align}\label{def: M_F and M_G}
   M^{(F)}_x 
    \coloneqq \| F \|_{\CC^1 ( [-x, x]^n ) }, \qquad\qquad 
    M^{(G)}_x 
    \coloneqq \| G \|_{\CC^3 ( [-x, x]^n ) },
\end{align}
for any $x > 0$.

We start by computing the regularity of the process $\CZ_v$ defined in \eqref{def: CZ(t,x)}.

\begin{lemma} \label{Holder_norm_of_mathcal(Z)}
    For any $t \geq 0$ we have 
    \begin{equation}\label{eq:Z-bound-1}
        \| \CZ_v(t_K) \|_{\CC^{\gamma_*}} 
        \leq C K \bigl( 1 + K^2 + L \bigr) M^{(G)}_{2 K+L},
    \end{equation}
    provided $\| v \|_{\CB_{t_K}} \leq L$, and for any $t > 0$ we have 
    \begin{equation}\label{eq:Z-bound-2}
        \| \CZ_v (t_K) \|_{\CC^{\gamma_*}} 
        \leq C K \bigl( 1 + K^2 + ( 1+ t^{ - \frac{\beta}{\alpha} ( \gamma_* + \kappa_0 ) } ) L \bigr) M^{(G)}_{2 K+L}, 
    \end{equation}
    provided $\| v \|_{\widetilde \CB_{t_K}} \leq L$. Here, we use the values \eqref{def: gamma_*}. 
\end{lemma}

\begin{proof}
    We will restrict all the time variables to the interval $[0, \tau_K]$ and write $t$ throughout the proof. Let us denote $\widetilde{G}_v(t,x) = G(\psi(t,x) + v(t,x) + B_t)$ and $\widetilde{DG}_v(t,x) = DG(\psi(t,x) + v(t,x) + B_t)$. From Lemma~\ref{Gubinelli deriv of G(v+x)}, we know that the process $( \widetilde{G}_v, \widetilde{DG}_v)(t,x)$ is controlled by the rough path $\Psi_t(x)$ as a process in the spatial variable $x$. By the rough path theory \cite[Theorem~4.10]{MR4174393}, for any $a, b \in [-\pi, \pi]$ we have
    \begin{align*}
        &\biggl| \int_a^b \widetilde{G}_v(t,x)\, \d_x \Psi_t(x) - \widetilde{G}_v(t,a) \big( \psi_t(b) - \psi_t(a) \big) - \widetilde{DG}_v(t, a) \uppsi_t(a, b) \biggr| \\
          &\qquad\qquad\qquad \lesssim \left( \| \psi_t \|_{\gamma_*} \| R^G_t \|_{2 \gamma_*} + \| \uppsi_t \|_{2 \gamma_*} \| DG_v(t) \|_{\gamma_*} \right) | b-a |^{3 \gamma_*},
    \end{align*}
    where 
    \begin{equation*}
        \| R^G_t \|_{2 \gamma_*} \coloneqq \sup_{a < b \in [-\pi , \pi]} \frac{ | \widetilde{G}_v(t,b) - \widetilde{G}_v(t,a) - \widetilde{DG}_v(t,a) ( \psi_t(b) - \psi_t(a)) |}{(b-a)^{2 \gamma_*}}.
    \end{equation*}
    Therefore, using the estimate above with  and the triangle inequality, we get
    \begin{align}
        | \CZ_v(t,b) - \CZ_v(t,a) | 
        &= \left| \int_a^b \widetilde{G}_v(t,x)\, \d_x \Psi_t(x) \right| \notag \\
          &\lesssim | \widetilde{G}_v(t,a) | | \psi_t(b) - \psi_t(a) | + | \widetilde{DG}_v(t,a)| | \uppsi_t(a, b) | \notag \\
        &\qquad + \left( \| \psi_t \|_{\gamma_*} \| R^G_t \|_{2 \gamma_*} + \| \uppsi_t \|_{2 \gamma_*} \| DG_v(t) \|_{\gamma_*} \right) | b-a |^{3 \gamma_*}. 
        \label{eq8}
    \end{align}
    Using \eqref{bdd: remainder R of G}, \eqref{eq:bound-on-G}, the function $M^{(G)}$, and the condition $\vvvert \Psi_t \vvvert_{\gamma_*} \leq K$, we bound the previous expression by
    \begin{align*}
        M^{(G)}_{2 K+L} \bigl( K (b-a)^{\gamma_*} + K^2 (b-a)^{2\gamma_*} \bigr) &+ \frac{1}{2} K M^{(G)}_{2 K+L} \left( K^2 + 2L \right) (b-a)^{3 \gamma_*} \\
        &\qquad + K M^{(G)}_{2 K+L} ( K + L ) (b-a)^{3 \gamma_*}.
    \end{align*}
Since $\CZ_v(t,-\pi) = 0$, the last estimate yields a bound on the supremum of $\CZ_v$, and we get
    \begin{align*}
        \left\| \CZ_v(t) \right\|_{\CC^{\gamma_*}} 
        &= \sup_{b \in [-\pi, \pi]} |\CZ_v(t,b)| + \sup_{a \neq b \, \in [-\pi , \pi] } \frac{ | \CZ_v(t,b) - \CZ_v(t,a) | }{ |b-a|^{\gamma_*} } \lesssim M^{(G)}_{2 K+L} K \bigl( 1 + K^2 + L \bigr). 
    \end{align*} 
    which yields the desired bound \eqref{eq:Z-bound-1}. 

    Now we will prove the second bound \eqref{eq:Z-bound-2}. Using the same estimates as above, we bound expression \eqref{eq8} by
    \begin{align*}
        2 M^{(G)}_{2 K+L} K \bigl((b-a)^{\gamma_*} + (b-a)^{2\gamma_*}\bigr) &+ \frac{1}{2} K M^{(G)}_{2 K+L} \left( K^2 + 2 t^{ - \frac{\beta}{\alpha} ( \gamma_* + \kappa_0 ) } L \right) (b-a)^{3 \gamma_*} \\
        &\qquad + K M^{(G)}_{2 K+L} ( K + L ) (b-a)^{3 \gamma_*}, 
    \end{align*}
    where $t^{ - \frac{\beta}{\alpha} ( \gamma_* + \kappa_0 ) }$ comes from the norm $\|v_t\|_{\CC^{2 \gamma_*}}$ in the definition of the norm $\|v\|_{\widetilde \CB_{t}}$ in \eqref{def: tilde CB_T norm}. Using again $\CZ_v(t,-\pi) = 0$, we get
    \begin{align*}
        \| \CZ_v(t) \|_{\CC^{\gamma_*}} &\lesssim M^{(G)}_{2 K+L} K \bigl( 1 + K^2 + ( 1+ t^{ - \frac{\beta}{\alpha} ( \gamma_* + \kappa_0 ) } ) L \bigr). 
    \end{align*}
    which is the desired bound \eqref{eq:Z-bound-2}. 
\end{proof}

For a function $f_v$, depending on $v$, we introduce the notation 
\begin{equation}
    \delta_{v, \bar v} f := f_v - f_{\bar v}, \label{def: delta v, v bar}
\end{equation}
for different $v$ and $\bar v$. Then we can study how $\CZ_v$ depends on $v$.

\begin{lemma} 
    For any $t \geq 0$ we have
    \begin{equation}\label{eq:delta-Z-bound-1}
        \| \delta_{v, \bar v} \CZ(t_K) \|_{\CC^{\gamma_*}} 
        \leq C K (1+ K^2 + L) M^{(G)}_{2K+L} \| v - \bar v \|_{\CB_{t_K}}, 
    \end{equation}
    provided $\| v \|_{\CB_{t_K}} \leq L$ and $\| \bar{v} \|_{\CB_{t_K}} \leq L$, and for any $t > 0$ we have
    \begin{equation}\label{eq:delta-Z-bound-2}
        \| \delta_{v, \bar v} \CZ (t_K) \|_{\CC^{\gamma_*}} 
        \leq C K \left( 1 + K^2 + ( 1+ t^{ - \frac{\beta}{\alpha} ( \gamma_* + \kappa_0 ) } ) L \right) M^{(G)}_{2K+L} \| v-\overline{v} \|_{ \widetilde \CB_{t_K} }, 
    \end{equation}
    provided $\| v \|_{ \widetilde \CB_{t_K}} \leq L$ and $\| \bar{v} \|_{ \widetilde \CB_{t_K}} \leq L$. Here, we use the values \eqref{def: gamma_*}.
\end{lemma}

\begin{proof}
    As in the proof of the preceding lemma, we restrict the time variable $t$ to $[0, \tau_K]$. Moreover, it will be convenient not to write it to make the formulas below look lighter. Using the notation from that proof, the process $( \delta_{v, \bar v} \widetilde{G}, \delta_{v, \bar v} \widetilde{DG})$ is controlled by the rough path $\Psi$. Then \eqref{eq8} yields 
    \begin{align}
        \left| \delta_{v, \bar v} \CZ(b) - \delta_{v, \bar v} \CZ(a) \right| 
        &\leq | \delta_{v, \bar v} \widetilde{G}(a) | | \psi(b) - \psi(a) | + | \delta_{v, \bar v} \widetilde{DG} (a) | | \uppsi(a,b) | \label{eq10} \\
          &\quad + C \left( \| \psi \|_{\gamma_*} \| R^{\delta_{v, \bar v} \widetilde{G}} \|_{2 \gamma_*} + \| \uppsi \|_{2 \gamma_*} \| \delta_{v, \bar v} \widetilde{DG} \|_{\CC^{\gamma_*}} \right) | b-a |^{3 \gamma_*}. \notag
    \end{align}
    By the mean value theorem and \eqref{eq:bound-on-G}, we get $| \delta_{v, \bar v} \widetilde{G} (a) | \leq M^{(G)}_{2K+L} \| v - \bar v \|_{L^\infty}$ and $| \delta_{v, \bar v} \widetilde{DG} (a) | \leq M^{(G)}_{2K+L} \| v-\overline{v} \|_{L^\infty}$. 
    
    Next, we bound the term $\| \delta_{v, \bar v} \widetilde{DG} \|_{\CC^{\gamma_*}}$. Let us denote for brevity $w := \overline{v} - v$. Using the fundamental theorem of calculus, we can write
    \begin{align*}
        \delta_{v, \bar v} \widetilde{DG}(b) - \delta_{v, \bar v} \widetilde{DG}(a)  
        &= \int_0^1 D^2 G( \psi + v + \lambda w + B )(b) w (b) \, \d \lambda \\
          &\qquad - \int_0^1 D^2 G( \psi + v + \lambda w + B )(a) w (a) \, \d \lambda, 
    \end{align*}
    where we have the matrix-vector products inside the integrals. Thus, by the triangle inequality, 
    \begin{align*}
        &\left| \delta_{v, \bar v} \widetilde{DG}(b) - \delta_{v, \bar v} \widetilde{DG}(a) \right| \\
          &\quad \leq | w (b) - w (a) | \int_0^1 \left| D^2 G( \psi + v + \lambda w +B )(b) \right| \d \lambda \\
        &\qquad \qquad + | w (a) | \int_0^1 \left| D^2 G( \psi + v + \lambda w + B )(b) - D^2 G( \psi + v + \lambda w +B )(a) \right| \d \lambda \\
        &\quad \leq \| w \|_{\CC^{\gamma_*}} |b-a|^{\gamma_*} M^{(G)}_{2K+L} + \| w \|_{L^\infty} |b-a|^{\gamma_*} \, M^{(G)}_{2K+L} \big( K + L + \| w \|_{\CC^{\gamma_*}} \big). 
    \end{align*}
    We use $D^3 G$ to bound the integrand in the last inequality, and this is the place where we need to assume that $G$ is $\CC^3$. In the same way, we can bound 
    \begin{align*}
        \left| \delta_{v, \bar v} \widetilde{DG}(a) \right| &\leq |w (a) | \int_0^1 \left| D^2 G( \psi + v + \lambda w +B )(a) \right| \d \lambda \\
        &\leq \| w \|_{L^\infty} \, M^{(G)}_{2K+L} \big(K + L + \| w \|_{L^\infty} \big).
    \end{align*}
    Then, using $\| w \|_{\CC^{\gamma_*}} \leq \|\overline{v}\|_{\CB_{t_K}} + \|v\|_{\CB_{t_K}} \leq 2L$ and $\| w \|_{L^\infty} \leq \| w \|_{\CC^{\gamma_*}}$, we get 
    \begin{equation}
        \| \delta_{v, \bar v} \widetilde{DG} \|_{\CC^{\gamma_*}} \lesssim M^{(G)}_{2K+L} (1+K+L) \| v-\overline{v} \|_{\CC^{\gamma_*}}. \label{eq11}
    \end{equation}

    Similarly, we can write the remainder $R^{\delta_{v, \bar v} \widetilde{G}}$ as 
    \begin{align*}
        &R^{\delta_{v, \bar v} \widetilde{G}} 
        = \delta_{v, \bar v} \widetilde{G}(b) - \delta_{v, \bar v} \widetilde{G}(a) - \delta_{v, \bar v} \widetilde{DG}(a) ( \psi(b) - \psi(a) ) \\
          &\qquad= \int_0^1 DG( \psi + v + \lambda w +B )(b) w(b) \, \d \lambda - \int_0^1 DG( \psi + v + \lambda w +B )(a) w(a) \, \d \lambda \\
          &\qquad\qquad\qquad - \left( \int_0^1 D^2 G( \psi + v + \lambda w )(a) w(a) \, \d \lambda \right) ( \psi(b) - \psi(a)). 
    \end{align*}
    Then by adding and subtracting $\int_0^1 DG( \psi + v + \lambda w +B )(b) w (a) \, \d \lambda$ and using the triangle inequality together with \eqref{eq:bound-on-G}, we get
    \begin{align*}
        &| R^{\delta_{v, \bar v} \widetilde{G}} | 
        \leq \| w \|_{\CC^{2 \gamma_*} } |b-a|^{2 \gamma_*} M^{(G)}_{2K+L} \\
          &\qquad + \| w \|_{L^\infty} \int_0^1 \Big| DG( \psi + v + \lambda w +B )(b) - DG( \psi + v + \lambda w +B )(a) \Big. \\
        &\hspace{7cm} - D^2 G( \psi + v + \lambda w +B )(a) ( \psi(b) - \psi(a) ) \Big| \, \d \lambda \\
          &\leq \| w \|_{\CC^{2 \gamma_*} } |b-a|^{2 \gamma_*} M^{(G)}_{2K+L} \\
        &\qquad + \| w \|_{L^\infty} \int_0^1 \Bigl( | D^2 G( \psi(b) + v(a) + \lambda w(a) +B ) | \big( | \overline{v}(b) - \overline{v}(a) | + \lambda | w (b) - w (a) | \big) \\
        &\hspace{8cm}  + \frac{1}{2} M^{(G)}_{2K+L} \big( \psi(b) - \psi(a) \big)^2\Bigr) \, \d \lambda. 
    \end{align*}
    We use the Taylor expansion inside the integral in the last inequality, and this is the place we need to assume $G$ is $\CC^3$. Using the function $M^{(G)}$, we can bound it by
    \begin{align*}
        \| w \|_{\CC^{2 \gamma_*} } |b-a|^{2 \gamma_*} M^{(G)}_{2K+L} & + \| w \|_{L^\infty} |b-a|^{2 \gamma_*} M^{(G)}_{2K+L} \big( \| \overline{v} \|_{\CC^{2 \gamma_*} } + \| w \|_{\CC^{2 \gamma_*} } \big) \\
        &\quad + \frac{1}{2} \| w \|_{L^\infty} |b-a|^{2 \gamma_*} M^{(G)}_{2K+L} \| \psi \|_{\CC^{\gamma_*}}^2. 
    \end{align*}
    With this, we can get the estimate
    \begin{align}
        \| R^{\delta_{v, \bar v} \widetilde{G}} \|_{2 \gamma_*} 
        &\lesssim M^{(G)}_{2K+L} (1 + K^2 + L) \| w \|_{\CC^{2 \gamma_*} }. \label{eq12} 
    \end{align}

    Using \eqref{eq11} and \eqref{eq12} in \eqref{eq10}, the H\"{o}lder norm of $\delta_v \CZ(s, \cdot)$ can be bounded by
    \begin{align*}
        \left\| \delta_{v, \bar v} \CZ - \delta_{v, \bar v} \CZ \right\|_{\CC^{\gamma_*}} 
        &\lesssim M^{(G)}_{2K+L} K \| w \|_{\CC^{\gamma_*}} + M^{(G)}_{2K+L} K (1 + K + L) \| w \|_{\CC^{2 \gamma_*} } \\
          &\qquad + M^{(G)}_{2K+L} K (1 + K^2 + L) \| w \|_{\CC^{2 \gamma_*} } \\
        &\lesssim  M^{(G)}_{2K+L} K (1 + K^2 + L) \| w \|_{\CC^{2 \gamma_*} }, 
    \end{align*}
    which is exactly \eqref{eq:delta-Z-bound-1}. 

    Now, we will prove the bound \eqref{eq:delta-Z-bound-2}. From the definition of the $\widetilde{\CB}_{t_K}$-norm in \eqref{def: tilde CB_T norm}, we have 
    \begin{equation*}
        \| w \|_{ \CC^{2 \gamma_*} } \leq t_K^{ - \frac{\beta}{\alpha} ( \gamma_* + \kappa_0 ) } \| w \|_{ \widetilde{\CB}_{t_K} }. 
    \end{equation*}
    Using this bound and following the calculation above, we can derive, similarly to \eqref{eq11} and \eqref{eq12},
    \begin{align}
        \| \delta_{v, \bar v} \widetilde{DG} \|_{\CC^{\gamma_*}} 
        &\lesssim M^{(G)}_{2K+L} (1+K+L) \| w \|_{\CC^{\gamma_*}} \label{eq11-new} \\
          \| R^{\delta_{v, \bar v} \widetilde{G}} \|_{2 \gamma_*} 
          &\lesssim M^{(G)}_{2K+L} \Bigl( K^2 + t_K^{ - \frac{\beta}{\alpha} ( \gamma_* + \kappa_0 ) } L \Bigr) \| w \|_{ \widetilde \CB_{t_K} }, \label{eq12-new} 
    \end{align}
    Applying \eqref{eq11-new} and \eqref{eq12-new} to \eqref{eq10}, the H{\"o}lder norm of $\delta_v \CZ(s, \cdot)$ can be bounded by
    \begin{align*}
        \left\| \delta_{v, \bar v} \CZ - \delta_{v, \bar v} \CZ \right\|_{\CC^{\gamma_*}} 
        &\lesssim M^{(G)}_{2K+L} K \| w \|_{ \widetilde \CB_{t_K} } + M^{(G)}_{2K+L} K (1+K+L) \| w \|_{ \widetilde \CB_{t_K} } \\
          &\qquad + M^{(G)}_{2K+L} K \Bigl( K^2 + t_K^{ - \frac{\beta}{\alpha} ( \gamma_* + \kappa_0 ) } L \Bigr) \| w \|_{ \widetilde \CB_{t_K} } \\
        &\lesssim M^{(G)}_{2K+L} K \Bigl( 1 + K^2 + ( 1+ t^{ - \frac{\beta}{\alpha} ( \gamma_* + \kappa_0 ) } ) L \Bigr) \| w \|_{ \widetilde \CB_{t_K} }, 
    \end{align*}
    which is the desired bound \eqref{eq:delta-Z-bound-2}. 
\end{proof}

\subsection{Estimates on $\CM_1$ and $\CM_2$}

Using the estimates on $\CZ_v$ from the previous section, we can bound the function $\CM_1$ defined in \eqref{def: M_i}. 

\begin{lemma} \label{norm of CM_1}
    Let $\widetilde{\gamma} \in (0, \infty) \setminus \N$ be such that $0 < \widetilde{\gamma} - \gamma_* < \alpha - 1$ and let $0 < \kappa < \gamma_* - \widetilde{\gamma} + \alpha$, where $\gamma_*$ is defined in \eqref{def: gamma_*}. Then, for any $t \geq 0$, we have
    \begin{align}\label{eq:M1-bound-1}
        \| \CM_1(v) \|_{\CB_{t_K}} 
        &\leq C t_K^{ \beta - \frac{\beta}{\alpha} ( 2 - \gamma_* + \kappa ) } K \bigl( 1 + K^2 + L \bigr) M^{(G)}_{2 K+L}, \\
        \| \CM_1(v) - \CM_1(\overline{v}) \|_{\CB_{t_K}} \label{eq:M1-bound-2}
        &\leq C t_K^{ \beta - \frac{\beta}{\alpha} ( 2 - \gamma_* + \kappa) } K (1+ K^2 + L) M^{(G)}_{2K+L} \| v - \bar v \|_{\CB_{t_K}},
    \end{align}
    provided $\| v \|_{\CB_{t_K}} \leq L$ and $\| \bar v \|_{\CB_{t_K}} \leq L$.
\end{lemma}

\begin{proof}
We restrict the time variables $t'<t$ to the interval $[0, \tau_K]$ and use them throughout the proof.  For $\zeta=\widetilde{\gamma}$ or $\zeta=1$, under the assumption $0 < \zeta - \gamma_* < \alpha -1 -\kappa$, we apply Proposition~\ref{sch est for fractional heat eq} with $m=1$ and $\gamma_{\ini} = \gamma_*$ and get
    \begin{align*}
        \| \CM_1(v) (t) \|_{\CC^{\zeta}}
        \lesssim t^{ \beta - \frac{\beta}{\alpha} (\zeta - \gamma_* + 1 + \kappa ) } \| \CZ_v \|_{\CC^{\gamma_*}_{t} } + t^{ \beta - \frac{\beta}{\alpha} (\zeta - \gamma_* + 1 + \kappa) } \left\| \CZ_v(0) \right\|_{\CC^{\gamma_*}}. 
    \end{align*}
    Using Lemma~\ref{Holder_norm_of_mathcal(Z)}, we bound both terms by a constant times
    \begin{align*}
        &t^{ \beta - \frac{\beta}{\alpha}(\zeta - \gamma_* + 1 + \kappa ) } K \bigl( 1 + K^2 + L \bigr) M^{(G)}_{2 K+L}.
    \end{align*}
Thus, using the last part of Proposition~\ref{sch est for fractional heat eq}, we obtain the desired bound \eqref{eq:M1-bound-1}. 

    To prove \eqref{eq:M1-bound-2}, we use the map \eqref{def: delta v, v bar} and write
    \begin{align*}
        \delta_{v, \bar v}\CM_1(t,x)
        &= \partial_x \int_0^t P_{\beta, \alpha}(t-s) \ast \partial_s^{1-\beta} \delta_{v, \bar v} \CZ(s, x) \, \d s.
    \end{align*}
   As above, we apply Proposition~\ref{sch est for fractional heat eq} and obtain  
    \begin{align*}
        \| \delta_{v, \bar v}\CM_1(t) \|_{ \CC^{\zeta} } 
        &\lesssim t^{ \beta - \frac{\beta}{\alpha} ( \zeta - \gamma_* + 1 + \kappa)} \| \delta_{v, \bar v} \CZ \|_{ \CC^{\gamma_*}_{t}} + t^{ \beta - \frac{\beta}{\alpha} ( \zeta - \gamma_* + 1 + \kappa)} \left\| \delta_{v, \bar v} \CZ(0) \right\|_{ \CC^{\gamma_*} },
    \end{align*}
    which is bounded, using \eqref{eq:delta-Z-bound-1}, by a constant multiple of 
    \begin{align*}
        t^{ \beta - \frac{\beta}{\alpha} (\zeta - \gamma_* + 1 + \kappa)} M^{(G)}_{2K+L} K (1+ K^2 + L^2) \| v - \bar v \|_{\CB_{t}},
    \end{align*}
    Combining these bounds with the last part of Proposition~\ref{sch est for fractional heat eq}, we get \eqref{eq:M1-bound-2}.
\end{proof}

Now, we will estimate the term $\CM_2(v)$ defined in \eqref{def: M_i}. Defining 
\begin{equation} \label{def: CV(t,x)}
    \CV_v(s,y) \coloneqq G(\psi + v +B)(s,y) \partial_y v(s,y), 
\end{equation}
we write 
\begin{equation*}
    \CM_2(v) = \int_0^t \left( P_{\beta, \alpha}(t-s) \ast \partial_s^{1-\beta} \CV_v(s) \right) (x) \, \d s,
\end{equation*}
which, as before, we interpret using integration by parts. 

As follows from the Schauder estimates, $v(s, \cdot )$ is $\CC^{\widetilde{\gamma}}$ for $s > 0$, where $\widetilde{\gamma} < \gamma_* + \alpha - 1$. Hence, $\CV_v$ is well-defined and with the H{\"o}lder regularity $\min( \gamma_* , \gamma_* + \alpha -2 )$. However, $\CV(0, \cdot)$ does not have this spatial regularity because of our assumption on the initial condition $v_0$. This is why we bound only the sup-norm of $\CV_v(s, \cdot )$. 

\begin{lemma} 
    For any $t \geq 0$ we have 
    \begin{align}
        \| \CV_v(t_K) \|_{L^\infty}  
        &\leq L M^{(G)}_{2 K+L},
        \label{Holder_norm_of_mathcal(V)} \\
          \| \delta_{v, \bar v} \CV(t_K) \|_{L^\infty} 
          &\leq (1+L) M^{(G)}_{2 K+L} \| v - \overline{v} \|_{\CB_{t_K}},
          \label{Holder_norm_of_delta mathcal(V)} 
    \end{align}
    provided $\|v\|_{\CB_{t_K}} \leq L$ and $\|\bar v\|_{\CB_{t_K}} \leq L$. Furthermore, for any $t > 0$ we have 
    \begin{align}
        \| \CV_v (t_K) \|_{L^\infty}  
        &\leq t_K^{ - \frac{\beta}{\alpha} ( 1 - \gamma_* + \kappa_0 ) } L M^{(G)}_{2 K+L},
        \label{Holder_norm_of_mathcal(V)-new} \\
          \| \delta_{v, \bar v} \CV (t_K)\|_{L^\infty} 
          &\leq t_K^{ - \frac{\beta}{\alpha} ( 1 - \gamma_* + \kappa_0 ) } (1+L) M^{(G)}_{2 K+L} \| v - \overline{v} \|_{\widetilde \CB_{t_K}},
          \label{Holder_norm_of_delta mathcal(V)-new} 
    \end{align}
    provided $\|v\|_{\widetilde\CB_{t_K}} \leq L$ and $\|\bar v\|_{\widetilde\CB_{t_K}} \leq L$.
\end{lemma}

\begin{proof}
As in the previous proofs, we restrict $t \in [0, \tau_K]$ and write $t$ everywhere. Recalling the definitions of time $t_K$ in \eqref{stop time t_K}, bound \eqref{eq:bound-on-G} and norm \eqref{def: CB_T norm}, we get
    \begin{equation*}
        \| \CV_v(t) \|_{L^\infty} 
        \leq M^{(G)}_{2 K+L} \| \partial_y v_t \|_{L^\infty} 
        \leq M^{(G)}_{2 K+L} \| v_t \|_{\CC^1} 
        \leq M^{(G)}_{2 K+L} L. 
    \end{equation*}
    For $\delta_{v, \bar v} \CV$, we get by the triangle inequality
    \begin{align*}
        &\| \delta_{v, \bar v} \CV(t) \|_{L^\infty}
        \leq \| G(\psi_t + v_t +B_t) - G(\psi_t + \bar v_t +B_t) \|_{L^\infty} \| \partial_y v_t \|_{L^\infty} \\
        &\hspace{4cm} + \| G(\psi_t + \bar v_t +B_t) \|_{L^\infty} \| \partial_y v_t -\partial_y \bar v_t \|_{L^\infty} \\
        &\qquad \leq M^{(G)}_{2 K+L} \| v_t - \overline{v}_t \|_{L^\infty} \| v_t \|_{\CC^1} + M^{(G)}_{2 K+L} \| v_t - \overline{v}_t \|_{\CC^1} \leq M^{(G)}_{2 K+L} (1+L) \| v_t - \overline{v}_t \|_{\CC^1}. 
    \end{align*}
    We prove the bounds concerning the $\widetilde \CB_{t}$-norm similarly, except that this time we use the estimate $\|v_t\|_{\CC^1} \leq t^{ - \frac{\beta}{\alpha} ( 1 - \gamma_* + \kappa_0 ) } \|v\|_{\widetilde \CB_{t}}$. 
\end{proof}

Using the preceding result, we can now estimate $\CM_2(v)$. 

\begin{lemma}\label{norm of CM_2}
    Let $\widetilde{\gamma} \in (0, \infty) \setminus \N$ be such that $0 < \widetilde{\gamma} < \alpha$ and let $0 < \kappa < \alpha - \widetilde{\gamma}$. Then for any $t > 0$ the function $\CM_2(v)$, defined in \eqref{def: M_i}, can be bounded by
    \begin{align}\label{eq:M2-bound-1}
        \| \CM_2(v) \|_{\CB_{t_K}} 
        &\leq C t_K^{ \beta - \frac{\beta}{\alpha} ( 1 + \kappa)} L M^{(G)}_{2 K+L}, \\
        \| \CM_2(v) - \CM_2(\overline{v}) \|_{\CB_{t_K}} 
        &\leq C t_K^{ \beta - \frac{\beta}{\alpha} ( 1 + \kappa)} ( 1 + L ) M^{(G)}_{2 K+L} \| v - \overline{v} \|_{\CB_{t_K}}, \label{eq:M2-bound-2}
    \end{align}
    provided $\| v \|_{\CB_{t_K}} \leq L$ and $\| \bar v \|_{\CB_{t_K}} \leq L$.
    \end{lemma}

\begin{proof} As before, we restrict $t \in [0, \tau_K]$ and write $t$ throughout the proof. Under assumption $0 < \zeta < \alpha -\kappa$, we apply Proposition~\ref{sch est for fractional heat eq} with $m=0$ and $\gamma = \gamma_{\ini} = 0$ and obtain 
    \begin{align*}
        \| \CM_2(v) (t) \|_{ \CC^{\zeta} } 
        &\leq \int_0^{t} \left\| P_{\beta, \alpha}(t-s) \ast \partial^{1-\beta}_s \CV_v(s) \right\|_{ \CC^{\zeta}} \d s \\
          &\lesssim t^{ \beta - \frac{\beta}{\alpha} (\zeta + \kappa)} \| \CV_v \|_{\CC_{t}} + t^{ \beta - \frac{\beta}{\alpha} (\zeta + \kappa)} \left\| \CV_v(0) \right\|_{ L^\infty }.  
    \end{align*}
    Combining this estimate with \eqref{Holder_norm_of_mathcal(V)}, we get 
    \begin{equation*}
        \| \CM_2(v) (t) \|_{ \CC^{\zeta} } 
        \lesssim t^{ \beta - \frac{\beta}{\alpha} (\zeta + \kappa)} M^{(G)}_{2 K+L} L.
    \end{equation*}
    Taking $\zeta=\widetilde{\gamma}$ and $\zeta=1$ and using the last part of Proposition~\ref{sch est for fractional heat eq}, we get \eqref{eq:M2-bound-1}.

    To prove the statement about the difference of $\CM_2$, we write 
    \begin{equation*}
        \delta_{v, \bar v} \CM_2(t) = \int_0^{t} \left( P_{\beta, \alpha}(t-s) \ast \partial^{1-\beta}_s \delta_{v, \bar v} \CV(s) \right)(x)\, \d s. 
    \end{equation*}
    As above, we apply Proposition~\ref{sch est for fractional heat eq} and get
    \begin{align*}
        \| \delta_{v, \bar v} \CM_2 (t) \|_{ \CC^{\zeta} } 
        &\lesssim t^{ \beta - \frac{\beta}{\alpha} (\zeta + \kappa)} \| \delta_{v, \bar v} \CV(t) \|_{L^\infty} + t^{ \beta - \frac{\beta}{\alpha} (\zeta + \kappa)} \left\| \delta_{v, \bar v} \CV(0) \right\|_{ L^\infty }, 
    \end{align*}
    and the estimate \eqref{Holder_norm_of_delta mathcal(V)} and the last part of Proposition~\ref{sch est for fractional heat eq} yield \eqref{eq:M2-bound-2}. 
\end{proof}

Next, we will consider Assumption~\ref{assump:initial}~\eqref{assump: initial 2nd kind} on the initial condition. The key difference is that now $\CV_v(0)$ is not well-defined because $v_0 = u_0 - \psi_0$ is no longer differentiable. However, under the Assumption~\ref{assump:initial}~\eqref{assump: initial 2nd kind}, $u_0 \in \CC^{\zeta_0}$ with $\zeta_0 > \frac{1}{2}$, so we have a chance to define the nonlinear term through Young's integral \cite{MR1555421}. 

Before starting the estimation, let us specify the functions that will be considered. Recalling the map defined in \eqref{def: M_i}, the functions $\CM_3(v)$ and $\CM_4(v)$ remain unchanged under the two kinds of initial condition assumptions. However, it will be convenient to combine the two terms $\CM_1(v)$ and $\CM_2(v)$ and define 
\begin{equation} \label{def: tilde CM_12}
    \widetilde{\CM}_{1,2}(v)(t,x) 
    \coloneqq \partial_x \int_0^t \left(\RL_{t-, s}^{1-\beta} P_{\beta, \alpha}(t-s) \ast \CU_v(s) \right)(x) \, \d s + \bigl(I_{t-, 0}^{\beta} P_{\beta, \alpha}(t) \ast \CU_v(0) \bigr)(x)
\end{equation}
instead, where  
\begin{equation*} 
    \CU_v(t,x) \coloneqq
    \begin{cases}
        \displaystyle\CZ_v(t,x) + \int_{-\pi}^x \CV_v(t,z)\, \d z & \text{for} ~ t > 0, \\[0.3cm]
        \displaystyle\int_{-\pi}^x G(u_0)\, \d u_0 & \text{for} ~ t = 0. 
\end{cases}
\end{equation*}
Notice that our assumption $u_0 \in \CC^{\zeta_0}$ with $\zeta_0 > \frac{1}{2}$ allows us to use Young's integral \cite{MR1555421} to define $\CU_v(0)$ as a function in $\CC^{\zeta_0}$. Recall that $\CZ_v$ and $\CV_v$ are defined in \eqref{def: CZ(t,x)} and $\eqref{def: CV(t,x)}$ respectively, and we write $\CV_v(s,x) = \partial_x \int_{-\pi}^x \CV_v(s,z) \d z$ and then pull $\partial_x$ outside the integral over time in \eqref{def: tilde CM_12} for convenience. For this initial condition, we are going to work with the expansion 
\begin{equation}\label{def: M_i-tilde}
    \CM(v) = \widetilde{\CM}_{1,2} + \CM_3(v) + \CM_4(v).
\end{equation}

After we specify the function we are going to estimate, we can bound $\widetilde{\CM}_{1,2}$. 

\begin{lemma} \label{norm of CM_12}
    Let $\widetilde{\gamma} \in (0, \infty) \setminus \N$ be such that $0 < \widetilde{\gamma} - \gamma_* < \alpha - 1$ and let $0 < \kappa < \alpha + \gamma_* - 2 - \kappa_0$, where we use the values \eqref{eq:gamma-0}. Then for any $t > 0$ we have 
    \begin{align}\label{eq:M12-bound-1}
        \| \widetilde \CM_{1,2} (v) \|_{ \widetilde \CB_{t_K} } 
        &\leq C t_K^{ \beta - \frac{\beta}{\alpha} ( 2 - \gamma_* + \kappa + \kappa_0 ) } ( 1 + K ) \bigl( 1 + K^2 + L \bigr) M^{(G)}_{2 K+L} \\
        &\hspace{5cm} + t_K^{ \beta - \frac{\beta}{\alpha} ( 1 - \zeta - \gamma_* +\kappa ) } \| u_0 \|_{\CC^{\zeta_0}}^2, \notag \\
        \| \widetilde \CM_{1,2}(v) - \widetilde \CM_{1,2}(\overline{v}) \|_{ \widetilde \CB_{t_K} } 
        &\leq C t_K^{ \beta - \frac{\beta}{\alpha} ( 2 - \gamma_* + \kappa + \kappa_0 ) } ( 1 + K ) \bigl( 1 + K^2 + L \bigr) \| v-\overline{v} \|_{ \widetilde \CB_{t_K} }, \label{eq:M12-bound-2}
    \end{align}
    provided $\| v \|_{ \widetilde \CB_{t_K}} \leq L$ and $\| \bar v \|_{ \widetilde \CB_{t_K}} \leq L$. 
\end{lemma}

\begin{proof}
    First, we need to estimate $\CU_v$. At time $0$, Young's integral \cite{MR1555421} has the bound
    \begin{equation*}
        \| \CU_v(0) \|_{\CC^{\zeta_0}} \lesssim \| u_0 \|_{\CC^{\zeta_0}}^2.
    \end{equation*}
As usual, we consider only $t \in [0, \tau_K]$. Then for $t > 0$ we have 
    \begin{align*}
        \| \CU_v(t) \|_{\CC^{\gamma_*}} 
        &\leq \| \CZ_v(t) \|_{\CC^{\gamma_*}} + \left\| \int_{-\pi}^{\cdot} \CV_v(t,z)\, \d z \right\|_{\CC^{\gamma_*}}
        \leq \| \CZ_v(t) \|_{\CC^{\gamma_*}} + \left\| \CV_v \right\|_{L^\infty_t}\\
          &\lesssim \big( 1 + t^{ - \frac{\beta}{\alpha} ( 1 - \gamma_* + \kappa_0 ) } \big) ( 1 + K ) \bigl( 1 + K^2 + L \bigr) M^{(G)}_{2 K+L}, 
    \end{align*}
    where we used \eqref{eq:Z-bound-2} and \eqref{Holder_norm_of_mathcal(V)-new}.

    Under the assumption $0 < \zeta - \gamma_* < \alpha -1 -\kappa$ and $\zeta_0 > \frac{1}{2}$ (see the Assumption~\ref{assump:initial}~\eqref{assump: initial 2nd kind}), we can applying Proposition~\ref{sch est for fractional heat eq} with $m=1$, $\gamma_{\ini} = \zeta_0$, and $\lambda = \frac{\beta}{\alpha} ( 1 - \gamma_* + \kappa_0 )$, and obtain 
    \begin{align*}
        &\| \widetilde \CM_{1,2} (v) (t) \|_{ \CC^{\zeta} } 
        \lesssim t^{ \beta - \frac{\beta}{\alpha} ( \zeta - 2\gamma_* + 2 + \kappa + \kappa_0 ) } \| \CU_v \|_{\CC^{\gamma_*}_t} + t^{ \beta - \frac{\beta}{\alpha} ( \zeta - \zeta_0 + 1 +\kappa ) } \left\| \CU_v(0) \right\|_{\CC^{\zeta}} \\
        &\qquad \lesssim t^{ \beta - \frac{\beta}{\alpha} ( \zeta - 2\gamma_* + 2 + \kappa + \kappa_0 ) } M^{(G)}_{2 K+L} ( 1 + K ) \bigl( 1 + K^2 + L \bigr) + t^{ \beta - \frac{\beta}{\alpha} ( \zeta - \zeta_0 + 1 +\kappa ) } \| u_0 \|_{\CC^{\zeta_0}}^2.
    \end{align*}
    Using this bound with $\zeta \in \{\widetilde{\gamma}, 1, 2 \gamma_*, \gamma_*\}$ and using the last part of Proposition~\ref{sch est for fractional heat eq}, we get \eqref{eq:M12-bound-1}.

    Next, we will estimate the difference of $\widetilde \CM_{1,2}$. Using the notation \eqref{def: delta v, v bar}, we write 
    \begin{align*}
        \delta_{v, \bar v}\widetilde \CM_{1,2}(t,x)
        &= \partial_x \int_0^t \left(\RL_{t-, s}^{1-\beta} P_{\beta, \alpha}(t-s) \ast \delta_{v, \bar v}\, \CU(s) \right)(x) \, \d s \\
          &\hspace{5cm} + \bigl(I_{t-, 0}^{\beta} P_{\beta, \alpha}(t) \ast \delta_{v, \bar v}\, \CU(0) \bigr)(x) \\
        &= \partial_x \int_0^t \left(\RL_{t-, s}^{1-\beta} P_{\beta, \alpha}(t-s) \ast \delta_{v, \bar v}\, \CU(s) \right)(x) \, \d s, 
    \end{align*}
    where the boundary term $\delta_{v, \bar v}\, \CU(0)$ vanishes because its value at time $0$ is independent of $v$. For $t > 0$ we have 
    \begin{align*}
        \| \delta_{v, \bar v}\, \CU(t) \|_{\CC^{\gamma_*}} 
        &\leq \| \delta_{v, \bar v} \CZ(t) \|_{\CC^{\gamma_*}} + \left\| \int_{-\pi}^{\cdot} \delta_{v, \bar v} \CV(t,z)\, \d z \right\|_{\CC^{\gamma_*}}
        \leq \| \delta_{v, \bar v} \CZ(t) \|_{\CC^{\gamma_*}} + \left\| \delta_{v, \bar v} \CV \right\|_{L^\infty_t}\\
          &\lesssim \big( 1 + t^{ - \frac{\beta}{\alpha} ( 1 - \gamma_* + \kappa_0 ) } \big) M^{(G)}_{2K+L} ( 1 + K ) \bigl( 1 + K^2 + L \bigr) \| v-\overline{v} \|_{ \widetilde \CB_{t} },
    \end{align*}
    where we used \eqref{eq:delta-Z-bound-2} and \eqref{Holder_norm_of_delta mathcal(V)-new}. Using $0 < \zeta - \gamma_* < \alpha -1 -\kappa$, we apply Proposition~\ref{sch est for fractional heat eq} and obtain 
    \begin{align*}
        \| \delta_{v, \bar v}\widetilde \CM_{1,2} (t) \|_{ \CC^{\zeta} } \lesssim t^{ \beta - \frac{\beta}{\alpha} (\zeta - 2\gamma_* + 2 + \kappa + \kappa_0 ) } M^{(G)}_{2K+L} ( 1 + K ) \bigl( 1 + K^2 + L \bigr) \| v-\overline{v} \|_{ \widetilde \CB_{t} }. 
    \end{align*}
    Taking the value of $\zeta$ as above and using the last part of Proposition~\ref{sch est for fractional heat eq}, we get \eqref{eq:M12-bound-2}. 
\end{proof}

\subsection{Estimates on $\CM_3$}

Now, we turn to the term $\CM_3(v)$ defined in \eqref{def: M_i}, which is very similar to $\CM_1$. It will be convenient to write  
\begin{align*}
    \CM_3(v) (s,y) 
    &= \int_0^t \left( P_{\beta, \alpha}(t-s) \ast \partial_s^{1-\beta} \CW_v(s) \right) (x) \; \d s,
\end{align*}
where $\CW_v (s,y) \coloneqq F\big( \psi(s,y) + v(s,y) + B_s \big)$.

\begin{lemma} \label{norm of CM_3}
    Let $\widetilde{\gamma} \in (0, \infty) \setminus \N$ be such that $0 < \widetilde{\gamma} < \alpha$ and  let $0 < \kappa < \alpha - \widetilde{\gamma}$. Then for any $t > 0$ we have
    \begin{align}\label{eq:M3-bound-1}
        \| \CM_3(v) \|_{\CB_{t_K}} 
        &\leq C t_K^{ \beta - \frac{\beta}{\alpha} ( 1 + \kappa)} M^{(F)}_{2 K+L}, \\
        \| \CM_3(v) - \CM_3(\overline{v}) \|_{\CB_{t_K}} 
        &\leq C t_K^{ \beta - \frac{\beta}{\alpha} ( 1 + \kappa)} M^{(F)}_{2 K+L} \| v - \overline{v} \|_{\CB_{t_K}}, \label{eq:M3-bound-2}
    \end{align}
    provided $\| v \|_{\CB_{t_K}} \leq L$ and $\| \bar v \|_{\CB_{t_K}} \leq L$, and 
    \begin{align}\label{eq:M3-bound-3}
        \| \CM_3(v) \|_{\widetilde \CB_{t_K}} 
        &\leq C t_K^{ \beta - \frac{\beta}{\alpha} ( \gamma_* + \kappa)} M^{(F)}_{2 K+L}, \\
        \| \CM_3(v) - \CM_3(\overline{v}) \|_{ \widetilde \CB_{t_K} } 
        &\leq C t_K^{ \beta - \frac{\beta}{\alpha} ( \gamma_* + \kappa)} M^{(F)}_{2 K+L} \| v - \overline{v} \|_{\widetilde \CB_{t_K}}, \notag 
    \end{align}
    provided $\| v \|_{ \widetilde \CB_{t_K} } \leq L$ and $\| \bar v \|_{ \widetilde \CB_{t_K} } \leq L$.
    \end{lemma}

\begin{proof}
As always we restrict $t \in [0, \tau_K]$. For any $0 < \zeta < \alpha -\kappa$, we apply Proposition~\ref{sch est for fractional heat eq} with $m=0$ and $\gamma = \gamma_{\ini} = 0$ and obtain 
    \begin{align*}
        \| \CM_3(v) (t) \|_{ \CC^{\zeta} } 
        \lesssim t^{ \beta - \frac{\beta}{\alpha} ( \zeta + \kappa)} \| \CW_v \|_{L^\infty_t} + t^{ \beta - \frac{\beta}{\alpha} ( \zeta + \kappa)} \left\| \CW_v(0) \right\|_{ L^\infty }.  
    \end{align*}
    Using the maximum function $M^{(F)}_x$ defined in \eqref{def: M_F and M_G}, we bound it by a constant multiple of $t^{ \beta - \frac{\beta}{\alpha} (\zeta + \kappa)} M^{(F)}_{2 K+L}$, which gives \eqref{eq:M3-bound-1}. Here, we used the last part of Proposition~\ref{sch est for fractional heat eq}.

    Furthermore, we write 
    \begin{equation*}
        \delta_{v, \bar v} \CM_3(t, x) = \int_0^{t} \left( P_{\beta, \alpha}(t-s) \ast \partial^{1-\beta}_s \delta_{v, \bar v} \CW(s) \right)(x)\, \d s. 
    \end{equation*}
    As above, we use Proposition~\ref{sch est for fractional heat eq} and get
    \begin{align*}
        \| \delta_{v, \bar v} \CM_3 (t) \|_{ \CC^{\zeta} } 
        &\lesssim t^{ \beta - \frac{\beta}{\alpha} ( \delta_{v, \bar v} \CM_3 + \kappa)} \| \delta_{v, \bar v} \CW \|_{L^\infty_{t}} + t^{ \beta - \frac{\beta}{\alpha} ( \delta_{v, \bar v} \CM_3 + \kappa)} \left\| \delta_{v, \bar v} \CW(0) \right\|_{ L^\infty }, 
    \end{align*}
    Using the estimate $\| \delta_{v, \bar v} \CW \|_{L^\infty_{t}} \leq M^{(F)}_{2 K+L} \| v - \overline{v} \|_{\CC_{t}} \leq M^{(F)}_{2 K+L} \| v - \overline{v} \|_{\CB_{t}}$, we get 
    \begin{align*}
        \| \delta_{v, \bar v} \CM_3 (t) \|_{ \CC^{\zeta} } 
        &\lesssim t^{ \beta - \frac{\beta}{\alpha} ( \zeta + \kappa)} M^{(F)}_{2 K+L} \| v - \overline{v} \|_{\CB_{t}}
    \end{align*}
    and hence \eqref{eq:M3-bound-2}, where we used the last part of Proposition~\ref{sch est for fractional heat eq}.
    
    The bounds \eqref{eq:M3-bound-3} can be proved in the same way.
\end{proof}

\subsection{Estimates on $\CM_4$}

Recall that the term $\CM_4(v)$, defined in \eqref{def: M_i}, is 
\begin{equation*}
    \CM_4(v) (t,x) = \int_{-\pi}^\pi P_{\beta, \alpha}(t, x-y) (u_0 - \psi_0)(y) \, \d y,
\end{equation*}
and we have the following estimates for it. 

\begin{lemma} \label{norm of CM_4}
Let $\widetilde{\gamma} \in (0, \infty) \setminus \N$ such that $0 < \widetilde{\gamma} - \gamma_* < \alpha$. 
    For any $t > 0$ and $0 < \kappa < \kappa_0$, where $\kappa_0$ is defined in \eqref{def: CB_T norm}, we have 
    \begin{align}\label{eq:M4-bound-1}
        \| \CM_4(v) \|_{\CB_t} &\leq C \bigl(1 + t^{ \frac{\beta}{\alpha} ( \kappa_0 - \kappa ) } \bigr) \| u_0 - \psi_0 \|_{\CC^1}, \\
        \| \CM_4(v) \|_{ \widetilde \CB_t } &\leq C \bigl(1 + t^{ \frac{\beta}{\alpha} ( \kappa_0 - \kappa ) }\bigr) \bigl(\| u_0 \|_{\CC^{\gamma_*}} + \| \psi_0 \|_{\CC^{\gamma_*}}\bigr). 
        \label{eq:M4-bound-2}
    \end{align}
\end{lemma}

\begin{proof}
    We apply Lemma~\ref{Sch-est w/o t} with $m=0$ and $\gamma = 1$, to get for any $0 < t < t'$
    \begin{align*}
        \| \CM_4(v) (t) \|_{ \CC^{\widetilde{\gamma}} } &\lesssim t^{ - \frac{\beta}{\alpha} ( \widetilde{\gamma} - 1 + \kappa)} \| u_0 - \psi_0 \|_{\CC^1}.
    \end{align*}
    Furthermore, Lemma~\ref{L^1-bound of P} yields 
    \begin{equation*}
        \| \CM_4(v) (t) \|_{\CC^1} \leq \| P_{\beta, \alpha}(t) \|_{L^1} \| u_0 - \psi_0 \|_{\CC^1} \lesssim \| u_0 - \psi_0 \|_{\CC^1}.
    \end{equation*}
    The time continuity of the map $\CM_4(v)$ can be proved as for the other maps $\CM_i(v)$, and we get \eqref{eq:M4-bound-1}.

    The bound \eqref{eq:M4-bound-2} is proved in the same way, but using Lemma~\ref{Sch-est w/o t} with $\gamma = \gamma_*$.
\end{proof}

\subsection{Proof of Theorem~\ref{thm:main}}
\label{sec:main-proof}

Combining the preceding results, we can prove Theorem~\ref{thm:main}. Throughout this proof, we fix a realization of the noise $\xi$, such that the stationary solution $\psi$ of the linear equation \eqref{eq:FSH} can be lifted to a rough path $\Psi$, which we use to define the rough path integral in \eqref{def: CZ(t,x)}. Due to Lemma~\ref{lem:lift}, we can do it for almost every realization of $\xi$.

First, we will prove the theorem under Assumption~\ref{assump:initial}\eqref{assump: initial 1st kind} on the initial condition. Recalling \eqref{def: M_i} and using Lemmas~\ref{norm of CM_1}, \ref{norm of CM_2}, \ref{norm of CM_3} and \ref{norm of CM_4}, we get for any $T > 0$
\begin{align*}
    \| \CM(v) \|_{\CB_{T_K}} 
    &\leq \sum_{i=1}^4 \| \CM_i(v) \|_{\CB_{T_K}} \\
      &\leq C_1 T_K^{ \beta - \frac{\beta}{\alpha} ( 2 - \gamma_* + \kappa) } K \bigl( 1 + K^2 + L \bigr) M^{(G)}_{2 K+L} + C_2 T_K^{ \beta - \frac{\beta}{\alpha} ( 1 + \kappa)} L M^{(G)}_{2 K+L} \\
    &\qquad + C_3 T_K^{ \beta - \frac{\beta}{\alpha} ( 1 + \kappa)} M^{(F)}_{2 K+L} + C_4 \bigl(1 +  T_K^{ \frac{\beta}{\alpha} ( \kappa_0 - \kappa ) }\bigr) \| u_0 - \psi_0 \|_{\CC^1}. 
\end{align*}
We take a value $K > 0$ sufficiently large so that $\| u_0 - \psi_0 \|_{\CC^1} \leq K$. Since all powers of $T_K$ are strictly positive, we can then estimate the preceding expression as 
\begin{equation}\label{eq:M-bound-proof}
\| \CM(v) \|_{\CB_{T_K}} \leq C_4 K + C(K, L) T_K^{\eta},
\end{equation}
for some $\eta > 0$. Recall the definition of $T_K$ in \eqref{stop time t_K}, which means that we restrict the time variable in this bound to the interval $(0, \sigma_K]$. Then we fix $L > 0$ such that $C_4 K \leq \frac{1}{2} L$ and choose $T^*(K,L) \in (0, \sigma_K]$ sufficiently small, depending on $K$ and $L$, so that the right-hand side is bounded by $L$ provided $\|v \|_{\CB_{T^*(K,L)}} \leq L$. It means that $\CM$ maps the ball $\{v \in \CB_{T} : \|v \|_{\CB_{T}} \leq L\}$ to itself for any $0 < T \leq T^*(K,L)$.

Next, we are going to show that $\CM(v)$ is a contraction map. Since the initial condition $u_0 - \psi_0$ is fixed, we have  
\begin{equation*}
    \CM(v) - \CM(\overline{v}) = \sum_{i=1}^3 \bigl(\CM_i(v) - \CM_i(\overline{v})\bigr). 
\end{equation*}
Applying Lemmas~\ref{norm of CM_1}, \ref{norm of CM_2} and \ref{norm of CM_3}, we get 
\begin{align*}
    &\| \CM(v) - \CM(\overline{v}) \|_{\CB_{T_K}} 
    \leq C_5 T_K^{ \beta - \frac{\beta}{\alpha} ( 2 - \gamma_* + \kappa) } K (1+ K^2 + L) M^{(G)}_{2K+L} \| v - \bar v \|_{\CB_{T_K}} \\
      &\qquad\qquad + C_6 T_K^{ \beta - \frac{\beta}{\alpha} ( 1 + \kappa)} ( 1 + L ) M^{(G)}_{2 K+L} \| v - \bar v \|_{\CB_{T_K}} + C_7 T_K^{ \beta - \frac{\beta}{\alpha} ( 1 + \kappa)} M^{(F)}_{2 K+L} \| v - \bar v \|_{\CB_{T_K}},
\end{align*}
where the powers are again all strictly positive. Taking $0< T^{**}(K,L) \leq T^*(K,L)$ small enough, also depending on $K$ and $L$, we get
\begin{equation*}
    \| \CM(v) - \CM(\overline{v}) \|_{\CB_{T^{**}(K,L)}} \leq \frac{1}{2} \| v - \overline{v} \|_{\CB_{T^{**}(K,L)}}
\end{equation*} 
provided $\| v \|_{\CB_{T^{**}(K,L)}} \leq L$ and $\| \overline{v} \|_{\CB_{T^{**}(K,L)}} \leq L$.

We conclude that $\CM$ is a contraction map on the ball $\{v \in \CB_{T^{**}(K,L)} : \|v \|_{\CB_{T^{**}(K,L)}} \leq L\}$. Therefore, by the Banach fixed point theorem, there is a unique $v \in \CB_{T^{**}(K,L)}$ such that $v = \CM(v)$ on $[0, T^{**}(K,L)]$, which means that \eqref{eq:FSBL} has a unique local solution. 

A similar argument can be applied in the case of an initial condition satisfying Assumption~\ref{assump:initial}\eqref{assump: initial 2nd kind}. Recalling the expansion \eqref{def: M_i-tilde} and applying Lemmas~\ref{norm of CM_12}, \ref{norm of CM_3} and \ref{norm of CM_4}, we get
\begin{align*}
    &\| \CM(v) \|_{\widetilde \CB_{T_K}} 
    \leq \| \widetilde \CM_{1,2}(v) \|_{\widetilde \CB_{T_K}} + \| \CM_3(v) \|_{\widetilde \CB_{T_K}} + \| \CM_4(v) \|_{\widetilde \CB_{T_K}} \\
      &\qquad \leq C_1 T_K^{ \beta - \frac{\beta}{\alpha} ( 2 - \gamma_* + \kappa + \kappa_0 ) } ( 1 + K ) \bigl( 1 + K^2 + L \bigr) M^{(G)}_{2 K+L} + C_2 T_K^{ \beta - \frac{\beta}{\alpha} ( 1 - \zeta - \gamma_* +\kappa ) } \| u_0 \|_{\CC^{\zeta_0}}^2 \\
    &\qquad\qquad + C_3 T_K^{ \beta - \frac{\beta}{\alpha} ( \gamma_* + \kappa)} M^{(F)}_{2 K+L} + C_4 \bigl(1 + t^{ \frac{\beta}{\alpha} ( \kappa_0 - \kappa ) }\bigr) \bigl(\| u_0 \|_{\CC^{\gamma_*}} + K\bigr),
\end{align*}
where we used $\|\psi(0)\|_{\CC^{\gamma_*}} \leq K$ in the last term. Taking $K > 0$ such that $\| u_0 \|_{\CC^{\zeta_0}} \leq K$, and hence $\| u_0 \|_{\CC^{\gamma_*}} \leq K$, we get a bound of the form \eqref{eq:M-bound-proof} but with different constants. Then the same argument as above gives us that $\CM$ maps the ball $\{v \in \widetilde\CB_{T_K} : \|v \|_{\widetilde\CB_{T_K}} \leq L\}$ to itself for any $0 < T \leq T^*(K,L)$, with a different value $T^*(K,L)$.

Furthermore, writing 
\begin{equation*}
    \CM(v) - \CM(\overline{v}) = \left( \widetilde \CM_{1,2}(v) - \widetilde \CM_{1,2}(\overline{v}) \right) + \left( \CM_3(v) - \CM_3(\overline{v}) \right) 
\end{equation*}
and using Lemmas~\ref{norm of CM_12} and \ref{norm of CM_3}, we get 
\begin{align*}
    \| \CM(v) - \CM(\overline{v}) \|_{ \widetilde{\CB}_{T_K} } 
    &\leq C_5 T_K^{ \beta - \frac{\beta}{\alpha} ( 2 - \gamma_* + \kappa + \kappa_0 ) } ( 1 + K ) \bigl( 1 + K^2 + L \bigr) \| v-\overline{v} \|_{ \widetilde \CB_{T_K} } \\
      &\qquad + C_6 T_K^{ \beta - \frac{\beta}{\alpha} ( \gamma_* + \kappa)} M^{(F)}_{2 K+L} \| v - \overline{v} \|_{\widetilde \CB_{T_K}}.
\end{align*}
Using the same argument as above, we get a time $0 < T^{**}(K,L) \leq T^*(K,L)$ such that 
\begin{equation*}
    \| \CM(v) - \CM(\overline{v}) \|_{ \widetilde{\CB}_{T^{**}(K,L)} } \leq \frac{1}{2} \| v - \overline{v} \|_{ \widetilde{\CB}_{T^{**}(K,L)} }
\end{equation*} 
provided $\| v \|_{ \widetilde{\CB}_{T^{**}(K,L)} } \leq L$ and $\| \overline{v} \|_{ \widetilde{\CB}_{T^{**}(K,L)} } \leq L$. By the Banach fixed-point theorem, there is a unique $v \in \widetilde{\CB}_{T^{**}(K,L)}$ such that $v = \CM(v)$ on $[0, T^{**}(K,L)]$.

We now explain how to extend the solution up to its maximal lifetime. Taking an initial condition $u_0$ satisfying one of the cases of Assumption~\ref{assump:initial}, we have constructed a solution $v$ on an interval $[0, T^{**}(K,L)]$. In both cases, the solution is of the form $u_{T^{**}(K,L)} = \psi_{T^{**}(K,L)} + v_{T^{**}((K,L)} + B_{T^{**}(K,L)}$, where $\psi_{T^{**}(K,L)} \in \CC^{\gamma_*}$, $v_{T^{**}(K,L)} \in \CC^{\widetilde{\gamma}}$, and $B_{T^{**}(K,L)}$ is independent of the spatial variable. 
Hence, the value $u_{T^{**}(K,L)}$, considered as a new initial condition, satisfies Assumption~\ref{assump:initial}\eqref{assump: initial 1st kind}. We can therefore iteratively apply the local existence result starting from $u_{T^{**}(K,L)}$, which yields a local solution on a larger time interval. Repeating this procedure, we extend the solution up to the maximal time $\bar T(K) \in (0, \sigma_K]$ such that in the case $\bar T(K) < \sigma_K$ we have 
\[
\lim_{T \nearrow \bar T(K)} \|v \|_{\widetilde{\CB}_{T}} = \infty.
\]
This extension procedure is standard in the PDEs literature (for a stochastic Burgers-type equation it is explained in full detail in \cite{MR3179667}), and we prefer to omit details. The definition of the norm \eqref{def: tilde CB_T norm} and the preceding limit yield 
\begin{equation}\label{eq:convergence-T-K}
\lim_{t \nearrow \bar T(K)} \|v_t\|_{\CC^{\tilde\gamma}} = \infty,
\end{equation}
where we recall that $\tilde \gamma$ can be chosen arbitrarily close but strictly larger than $1$. We note that the local solution $v$ on $[0, \bar T(K))$, defined for a value $K$ by the described procedure, coincides with the local solution on $[0, \bar T(K+1))$, defined for the value $K+1$. This follows from the uniqueness of the solution for every value of $K$ and the fact that if the driving noise is defined by $K$ as in \eqref{eq:noise-bounded}, then it is automatically bounded by $K+1$. Since the stopping times $\sigma_K$ are monotonically increasing, this implies that $\bar T(K) \nearrow \bar T \in (0, \infty]$. Hence, almost surely, for any $T \in (0, \bar T)$ there is $K$ such that $\bar T(K) > T$ and the solution $v$ is defined on $[0, T]$. We note that so far we have worked with a fixed realization of the noise $\xi$, and the time $\bar T$ is defined for this realization. 

Furthermore, one can conclude from \eqref{stop time sigma_K} and Lemma~\ref{lem:lift} that $\sigma_K \nearrow \infty$ almost surely. Indeed, for any $T > 0$ we have
\begin{align*}
\P(\sigma_K \leq T) &\leq \P(\vvvert \Psi \vvvert_{\gamma_*, T} \geq K ) + \P(\| B \|_{L^\infty([0, T])} \geq K\Bigr) \\
&\leq \frac{1}{K} \Bigl( \E \bigl[\vvvert \Psi \vvvert_{\gamma_*, T}\bigr] + \E \bigl[\| B \|_{L^\infty([0, T])}\bigr] \Bigr),
\end{align*}
where we used Markov's inequality. Due to Lemma~\ref{lem:lift} and standard properties of the Brownian motions, both expectations are bounded, and we get 
\begin{align*}
\lim_{K \nearrow \infty} \P(\sigma_K \leq T) = 0.
\end{align*}
Recalling that the sequence $\sigma_K$ is monotonically increasing, we conclude that $\sigma_K \nearrow \infty$ almost surely.

Hence, if $\bar T < \infty$, then $K$ can be chosen so that $\bar T(K) \leq \bar T < \tau_K$. Then the limit \eqref{eq:convergence-T-K} yields $\bar T(K) = \bar T$ for this value $K$, and we get \eqref{eq:convergence-lifetime}.

Lastly, we prove that if $F$ and $G$ are bounded together with their derivatives, then the solution is global in time. We prove it by contradiction. Assume $\bar T < \infty$, then the preceding argument gives that there exists a value $K$ such that $\bar T(K) = \bar T$ and 
\begin{equation*}
    \lim_{T \nearrow \bar T(K)} \|v \|_{\widetilde{\CB}_{T}} = \infty. 
\end{equation*}
If $M^{(F)}_x$ and $M^{(G)}_x$ are bounded, then the constant $C(K,L)$ in \eqref{eq:M-bound-proof} is only linear in $L$. That implies that $T^*(K, L)$ can be chosen to depend only on $K$ and that there exists a constant $L^* = L^*(K, \|u_0 - \psi_0\|_{\CC^1})$ such that for any $L \geq L^*$ we have
\begin{equation*}
    \| \CM(v) \|_{\CB_{T^*(K)}} \leq L,
\end{equation*}
provided $\|v\|_{\CB_{T^*(K)}} \leq L$. This implies $T^*(K) \leq \bar T(K) = \bar T$. Then we can use $v_{T^*(K)}$ as the initial condition to iterate the process to extend the solution by the amount of $T^*(K)$, which is a fixed value depending only on $K$. Therefore, we can extend the solution to infinite time, which gives a contradiction. So, $\bar T = \infty$ if $M^{(F)}_x$ and $M^{(G)}_x$ are bounded.

\appendix

\section{Besov spaces $\CB^\gamma_{\infty, \infty}$}
\label{sec:besov space}

In this section, we will use the Littlewood--Paley theory to introduce the Besov spaces. A useful feature of these spaces is that they provide a characterization of the H\"{o}lder norms in terms of the Fourier transforms of the functions. We are using this characterization to prove the Schauder-type estimates in Section~\ref{sec:schauder}. The reader may consult \cite{MR2768550} for a detailed exposition of the theory. 

Before defining the spaces, we need to introduce a \emph{dyadic partition of unity}.

\begin{lemma} \label{def: dyadic partition of unity}
    Let us define the annulus $\CA := \{ x \in \R : \frac{3}{4} \leq |x| \leq \frac{8}{3} \}$. There exist smooth even functions $\chi, \varphi : \R \to [0,1]$ with supports in the ball $B(0, \frac{4}{3})$ and the annulus $\CA$ respectively, which satisfy 
    \begin{enumerate}[(i)]
        \item \label{def: dyadic partition of unity eq1} $\chi(x) + \displaystyle\sum_{j \geq 0} \varphi(2^{-j} x) = 1$ for all $x \in \R$,
        \item \label{def: dyadic partition of unity eq3} $\supp \, \varphi(2^{-j} \cdot ) \bigcap \supp \, \varphi(2^{-j'} \cdot ) = \emptyset$ for all $j, j' \geq 0$ satisfying $|j-j'|\geq 2$, 
        \item \label{def: dyadic partition of unity eq4} $\supp \, \chi \bigcap \supp \, \varphi(2^{-j} \cdot ) = \emptyset$ for all $j \geq 1$. 
        \item \label{def: dyadic partition of unity eq5} There exists $\xi_0 \in \R_+$ such that $\varphi$ is increasing on $[0,\xi_0]$ and decreasing on $[\xi_0, \infty)$.
    \end{enumerate}
\end{lemma}

\begin{proof}
A proof of this result can be found in \cite[Proposition~2.10]{MR2768550}. Since the property (\ref{def: dyadic partition of unity eq5}) is not stated explicitly in \cite{MR2768550}, we provide some details of the argument to verify it.

We define the auxiliary function
\begin{equation*}
    \eta(x) = 
    \begin{cases}
        e^{-1/x} &\text{for } x > 0, \\
        0 &\text{for } x \leq 0,
    \end{cases}
\end{equation*}
which is increasing and smooth on $\R$. Next, we set $S(x) \coloneqq \frac{ \eta(x) }{ \eta(x) + \eta( 1-x ) }$, which is also increasing and smooth, and moreover satisfies $S(x) = 0$ for $x \leq 0$ and $S(x) = 1$ for $x \geq 1$ and 
\begin{equation} \label{eq: S'(x)}
    S'(x) = \1{0 \leq x \leq 1} \, \frac{1}{x^2 (1-x)^2( \eta(x) + \eta(1-x))^2 } e^{ - \frac{1}{x} - \frac{1}{1-x} }.
\end{equation}
So it is easy to see that $S(x)$ is strictly increasing on $[0,1]$. 

Then the functions $\chi$ and $\phi$ are defined as 
\begin{equation*}
    \chi(x) \coloneqq 1 - S\left( \frac{ |x| - \frac{3}{4} }{ \frac{4}{3} - \frac{3}{4}} \right), \qquad  \varphi(x) \coloneqq \chi\left( \frac{x}{2} \right) - \chi( x ).
\end{equation*}
Then one can verify that these $\chi$ and $\varphi$ have the required properties. The conditions \eqref{def: dyadic partition of unity eq1}, \eqref{def: dyadic partition of unity eq3}, and \eqref{def: dyadic partition of unity eq4} are easy to check, so we will verify the condition \eqref{def: dyadic partition of unity eq5}. 

Indeed, we can see  $\chi(x) = 1$ for $|x| \leq \frac{3}{4}$ and $\chi(x) = 0$ for $|x| \geq \frac{4}{3}$. Moreover from \eqref{eq: S'(x)}, we can see $\chi$ is strictly decreasing on $[\frac{3}{4}, \frac{4}{3}]$. Therefore, $\varphi(x) = 0$ for $|x| \in [0, \frac{3}{4}] \cup [ \frac{8}{3}, \infty )$, $\varphi(x) = 1$ for $|x| \in [\frac{4}{3}, \frac{3}{2} ]$ and $\varphi(x)$ is strictly increasing on $[ \frac{3}{4}, \frac{4}{3} ]$ and strictly decreasing on $[ \frac{3}{2}, \frac{8}{3} ]$. 
\end{proof}

Having this dyadic partition of unity, we can define the \emph{Littlewood--Paley block}. Recall that any distribution $\psi$ defined on the circle $\T$ can be written as the Fourier series in the sense of distributions
\begin{equation*}
    \psi(x) = \frac{1}{\sqrt{2 \pi}} \sum_{ k \in \mathbb{Z} } \widehat{\psi}(k) e^{ikx}. 
\end{equation*}
For $j \geq 0$ we define the $j$-th Littlewood--Paley block of $\psi$ as 
\begin{equation}\label{eq:PL-block}
    \Delta_j \psi(x) \coloneqq \frac{1}{\sqrt{2 \pi}} \sum_{ k \in \mathbb{Z} } \varphi(2^{-j} k) \widehat{\psi}(k) e^{ikx}
\end{equation}
and define $\Delta_{-1} \psi \coloneqq \frac{1}{\sqrt{2 \pi}} \sum_{k \in \mathbb{Z}} \chi(k) \widehat{\psi}(k) e^{ikx}$. Then we set 
\begin{equation*}
    h_j (x) \coloneqq 
    \begin{cases}
        \displaystyle \sum_{ k \in \mathbb{Z} } \varphi(2^{-j} k) e^{ikx} & \text{ for } j \geq 0, \\[0.5cm]
        \displaystyle \sum_{ k \in \mathbb{Z} } \chi(k) e^{ikx} & \text{ for } j = -1,
    \end{cases}
\end{equation*}
which yields $\Delta_j f = h_j * f$. 

\begin{defn}
    For any $\gamma \in \R$, the \emph{Besov space} $B^{\gamma}_{\infty, \infty} (\T)$ consists of those functions/distribution $\psi$ on $\T$, for which
    \begin{equation*}
        \| \psi \|_{B^{\gamma}_{\infty, \infty}} \coloneqq \sup_{j \geq -1} 2^{\gamma j} \| \Delta_j \psi \|_{L^\infty} < \infty.
    \end{equation*}
\end{defn}

For $\gamma \in \R_+ \setminus \N_0$, the Besov space $B^{\gamma}_{\infty, \infty} (\T)$ coincides with the H{\"o}lder space $\CC^{\gamma}(\T)$, a proof of which can be found in \cite{MR2768550}.

\begin{rem} \label{rem: dyadic partition of unity}
    In fact, one can show that the Besov space $\CB^\gamma_{\infty,\infty}$ is independent of the particular choice of dyadic partition of unity, in the sense that the corresponding Besov norms are equivalent. More precisely, the functions $\varphi$ and $\chi$ may be chosen to have supports with different radii of annulus and ball. Moreover, the condition~\eqref{def: dyadic partition of unity eq5} in Lemma~\ref{def: dyadic partition of unity} can even be dropped. We refer the reader to \cite{MR2768550} for a detailed discussion.
\end{rem}

\section{Mild solution of fractional equation}
\label{sec:Duhamel}

Duhamel's principle is usually used to solve linear parabolic PDEs or to write a PDE in a mild form, which allows one to perform further fixed-point arguments. An analogue of Duhamel's principle in the case when the time derivative is fractional is more complicated; we refer to \cite{MR3357610} for its detailed explanation with relevant references, and in this section we only provide the formula and formally justify it.

Consider a non-homogeneous fractional evolution equation 
\begin{equation}\label{nonhomo frac evol eq}
    ( \partial_t^\beta - L) X(t,x) = f(t,x)
\end{equation}
on $\R_+ \times \T$ with an initial condition $X_0$ at time $t=0$. We assume that $L$ is a differential operator such that $\partial_t^\beta - L$ has a fundamental solution $G(t,x)$. Then  fractional Duhamel's principle yields the solution
\begin{equation}\label{nonhomo frac evol eq-mild}
    X(t,x) = \int_{\T} G(t, x-y) X_0(y)\, \d y + \int_0^t \int_{\T} G(t-s, x-y) \partial_s^{1-\beta} f(s,y)\, \d y \d s.
\end{equation}
When the involved functions are sufficiently regular, one can readily justify that this is indeed a solution. In this case, we recall the definition of the fractional derivative \eqref{eq:frac-deriv} and get
\begin{align*}
    \partial_t^\beta X(t,x) 
    &= I_t^{1-\beta} \circ \partial_t \int_{\T} G(t, x-y) X_0(y)\, \d y \\
    &\qquad + I_t^{1-\beta} \circ \partial_t \left( \int_0^t \int_{\T} G(t-s, x-y) \partial_s^{1-\beta} f(s,y)\, \d y \d s \right) \\
      &= I_t^{1-\beta} \Big( \int_{\T} G(0, x-y) \partial_t^{1-\beta} f(t,y)\, \d y \Big) + I_t^{1-\beta} \int_{\T} \partial_t G(t, x-y) X_0(y)\, \d y \\
    &\qquad +  I_t^{1-\beta} \Big( \int_0^t \int_{\T} \partial_t G(t-s, x-y) \partial_s^{1-\beta} f(s,y)\, \d y \d s \Big).
\end{align*}
We interchange the Caputo derivative with the space integral in the first term and then use the properties \eqref{eq:derivative-of-integral} and $G(0, x) = \delta_0(x)$ to write the first term as $f(t,x)$. We use the definition \eqref{eq:RL-integral} in the other two terms to write them as 
\begin{align*}
\int_{\T} &\int_0^{t} \frac{ \partial_r G(r, x-y) }{ (t -r )^{\beta} } \d r\, X_0(y)\, \d y + \int_0^t \int_{\T} \int_0^{t-s} \frac{ \partial_r G(r, x-y) }{ (t-s -r )^{\beta} } \d r\, \partial_s^{1-\beta} f(s,y)\, \d y \d s, 
\end{align*}
where we swapped the integrals and changed the integration variables. Recalling from \eqref{eq:frac-deriv} that $\int_0^{t} \frac{ \partial_r G(r, x) }{ (t -r )^{\beta} } \d r = \partial_t^{\beta}G(t, x)$ and using the definition of the fundamental solution $G$, we write the preceding expression as 
\begin{align*}
\int_{\T} &(L G) (t,x-y) X_0(y)\, \d y + \int_0^t \int_{\T} (L G)(t-s, x-y) \partial_s^{1-\beta} f(s,y)\, \d y \d s.
\end{align*}
Summarizing these computations, we obtain exactly \eqref{nonhomo frac evol eq}.

The mild form of the equation \eqref{nonhomo frac evol eq-mild} is not defined in the way it is written if the function $f$ is not $\CC^1$ in the time variable. This is because the definition of the fractional derivative \eqref{eq:frac-deriv} involves differentiation. As in \eqref{eq:heat-solution}, we make sense of it by using integration by parts and write 
\begin{equation}\label{eq:X-mild}
    X(t,x) = \int_0^t \left(\RL_{t-, s}^{1-\beta} P_{\beta, \alpha}(t-s) \ast f(s) \right)(x) \, \d s + \left[\bigl(I_{t-, s}^{\beta} P_{\beta, \alpha}(t-s) \ast f(s) \bigr)(x) \right]_{s=0}^t.
\end{equation}
The advantage of this formula is that the fractional derivative $\RL_{t-, s}^{1-\beta}$ can be applied to the kernel $P_{\beta, \alpha}$, and this expression is well-defined. The performed computations are valid for a temporarily $\CC^1$ function $f$, and for less regular functions we use \eqref{eq:X-mild} as a definition of mild solution of equation \eqref{nonhomo frac evol eq}.

It remains to justify that the expression \eqref{eq:X-mild} is indeed well-defined. First, we will show that $\RL_{t-, s}^{1-\beta} P_{\beta, \alpha}(t-s)$ is defined. We recall from Lemma~\ref{lem:frac_D_of_ML} that
\begin{equation*}
    \RL^{1-\beta}_{t-, s} E_{\beta}( -(t-s)^\beta |k|^\alpha ) = \frac{\beta}{(t-s)^{1-\beta}} E'_{\beta}(-(t-s)^\beta |k|^\alpha) 
\end{equation*}
and that the fractional heat kernel equals $P_{\beta, \alpha}(t,x) = \frac{1}{2 \pi} \sum_{k \in \Z} E_{\beta}(- t^\beta |k|^\alpha) e^{ikx}$. Thus, to verify whether we can take the fractional derivative of this expansion term by term, we need to show that
\begin{equation*}
    \sum_{k \in \Z} \left| \RL^{1-\beta}_{t-, s} E_{\beta}( -(t-s)^\beta |k|^\alpha ) \right| < \infty
\end{equation*}
holds locally uniformly in $t > s > 0$. To see that the sum is bounded, we use Lemmas~\ref{lem:frac_D_of_ML} and \ref{lem:ML-bound} to write it as
\begin{align*}
     \beta \sum_{k \in \Z} \left| (t-s)^{\beta-1} E'_{\beta}(-(t-s)^\beta |k|^\alpha) \right| 
     &\leq C_1 \sum_{k \in \Z} \left| (t-s)^{\beta-1} \frac{1}{ 1+ (t-s)^{2 \beta} |k|^{2 \alpha} } \right| \\
       &\leq C_2 (t-s)^{\beta-1} \left( 1 + 2 \int_0^\infty \frac{1}{ 1+ (t-s)^{2 \beta} k^{2 \alpha} } \d k \right).
\end{align*}
Using the change of variable $\tilde{k} = (t-s)^{ \frac{\beta}{\alpha} } k$ to write it as 
\begin{equation*}
    C_2 (t-s)^{ \beta - 1 } \left( 1 + 2 (t-s)^{ - \frac{\beta}{\alpha} } \int_0^\infty \frac{1}{ 1+ \tilde{k}^{2 \alpha} } \d \tilde{k} \right) 
    \leq C_3 (t-s)^{ \beta - 1 - \frac{\beta}{\alpha} },
\end{equation*}
where the integral is bounded because $2 \alpha > 1$ (see Assumption~\ref{assump:parameters}) and the last bound holds uniformly over $0 \leq s \leq t \leq T$ for any fixed $T > 0$ with proportionality constant $C_3$ depending on $T$. Therefore, we can write 
\begin{equation*}
    \RL_{t-, s}^{1-\beta} P_{\beta, \alpha}(t-s, x) = \frac{\beta}{2\pi} \sum_{k \in \Z} (t-s)^{\beta-1} E'_{\beta}(-(t-s)^\beta |k|^\alpha) e^{ikx}, 
\end{equation*}
which satisfies 
\begin{equation*}
    \| \RL_{t-, s}^{1-\beta} P_{\beta, \alpha}(t-s) \|_{L^\infty} \leq C_3 (t-s)^{ \beta - 1 - \frac{\beta}{\alpha} }
\end{equation*}
uniformly over $0 \leq s \leq t \leq T$ for any fixed $T > 0$.

Second, we can show that the kernel $I_{t-, s}^{\beta} P_{\beta, \alpha}(t-s)$ is well-defined. It follows by a similar argument, but using that 
\begin{equation*}
    \sum_{k \in \Z \setminus \{0\}} \frac{1}{|k|^\alpha} \bigl| E_{\beta}( - t^\beta |k|^\alpha) -1 \bigr| < \infty
\end{equation*}
holds locally uniformly in the time variable. 

Hence, expression \eqref{eq:X-mild} is well-defined if we can show that the kernels can be convolved with the function $f$. For example, if $f \in L^\infty_T$, then we can apply Young's convolution inequality to obtain
\begin{align*}
    &\left| \int_0^t \left(\RL_{t-, s}^{1-\beta} P_{\beta, \alpha}(t-s) \ast f(s) \right)(x) \, \d s \right| 
    \leq \int_0^t \| \RL_{t-, s}^{1-\beta} P_{\beta, \alpha}(t-s) \|_{L^1} \| f(s) \|_{L^\infty} \, \d s \\
      &\hspace{4cm}\leq 2 \pi C_3 \|f\|_{L^\infty_t} \int_0^t (t-s)^{ \beta - 1 - \frac{\beta}{\alpha} } \d s \leq C_4 t^{ \beta - \frac{\beta}{\alpha} } \|f\|_{L^\infty_t}
\end{align*}
for any $t \in (0, T]$ with any fixed $T > 0$, where the constant $C_4$ depends on $T$. The integral is finite due to Assumption~\ref{assump:parameters} on the parameters $\alpha$ and $\beta$. Thus, the first term on the right-hand side of \eqref{eq:X-mild} is well-defined. The second term is bounded similarly. 

\bibliographystyle{imsart-number}
\bibliography{references}

\end{document}